\documentclass[12pt,reqno]{amsart}
\usepackage[a4paper, margin=1in]
{geometry}
\usepackage[svgnames]{xcolor}
\usepackage[ruled,vlined]{algorithm2e}
\usepackage{mathtools,multicol,etoolbox}
\usepackage{makecell}
\usepackage{nameref}
\usepackage{caption}
\makeatletter
\patchcmd{\algocf@makecaption@ruled}{\hsize}{\textwidth}{}{} 
\patchcmd{\@algocf@start}{-1.5em}{0em}{}{} 
\makeatother

\SetAlCapHSkip{0pt} 
\makeatletter
\@namedef{subjclassname@2020}{%
  \textup{2020} Mathematics Subject Classification}
\makeatother

\definecolor{lightyellow}{RGB}{255, 255, 197}

\title{Inverse curve problems on del Pezzo surfaces}

\author[Kaya]{Enis Kaya}
\address{Department of Mathematics, Bilkent University, 06800 Ankara, Turkey}
\email{enis.kaya@bilkent.edu.tr}

\author[McKean]{Stephen McKean}
\address{Department of Mathematics\\ Brigham Young University\\ Provo, UT\\ USA}
\email{mckean@math.byu.edu}

\author[Streeter]{Sam Streeter}
\address{School of Mathematics, University of Bristol, Fry Building, Woodland Road, BS8 1UG, UK}
\email{sam.streeter@bristol.ac.uk}

\author[Uppal]{H.\ Uppal}
\address{School of Mathematics, University of Bristol, Fry Building, Woodland Road, BS8 1UG, UK}
\email{huppal.maths@gmail.com}

\usepackage{amsmath}
\usepackage{amssymb}
\usepackage{amsthm}
\usepackage{enumerate}
\usepackage{tikz-cd}
\usepackage{soul}
\usepackage[backref=page]{hyperref}
\usepackage{mathrsfs}
\usepackage{parskip}
\usepackage{booktabs}
\usepackage{tikz}
\usepackage{longtable}
\usepackage{multirow}

\theoremstyle{plain}
\newtheorem{thm}{Theorem}[section]
\newtheorem{lem}[thm]{Lemma}
\newtheorem{cor}[thm]{Corollary}
\newtheorem{prop}[thm]{Proposition}

 \theoremstyle{definition}
\newtheorem{defn}[thm]{Definition}

\newtheorem{rem}[thm]{Remark}
\newtheorem{ex}[thm]{Example}
\newtheorem{notn}[thm]{Notation}
\newtheorem{ques}[thm]{Question}

\newcommand{\thistheoremname}{}
\newtheorem{genericthm}[thm]{\thistheoremname}

\newtheorem*{genericthm*}{\thistheoremname}
\newenvironment{namedthm*}[1]
  {\renewcommand{\thistheoremname}{#1}%
   \begin{genericthm*}}
  {\end{genericthm*}}

\let\originalleft\left
\let\originalright\right
\renewcommand{\left}{\mathopen{}\mathclose\bgroup\originalleft}
\renewcommand{\right}{\aftergroup\egroup\originalright}

\renewcommand{\emptyset}{\varnothing}
\newcommand{\mb}[1]{\mathbb{#1}}

\newcommand{\mf}[1]{\mathfrak{#1}}

\DeclareMathOperator{\PP}{\mathbb{P}}

\DeclareMathOperator{\Aut}{Aut}
\DeclareMathOperator{\Bl}{Bl}

\DeclareMathOperator{\ch}{char}

\DeclareMathOperator{\Frob}{Frob}
\DeclareMathOperator{\Gal}{Gal}

\DeclareMathOperator{\ICP}{ICP}
\DeclareMathOperator{\IGP}{IGP}

\DeclareMathOperator{\Pic}{Pic}
\DeclareMathOperator{\Sp}{Sp}

\DeclareMathOperator{\rank}{rank}

\DeclareMathOperator{\Spec}{Spec}

\DeclareMathOperator{\im}{im}
\newcommand{\Char}{\operatorname{char}}

\newcommand*\Q{\mathbb{Q}}

\newcommand*\F{\mathbb{F}}

\newcommand{\II}{I\!I}

\renewcommand{\epsilon}{\varepsilon}

\DeclareMathOperator{\Sym}{\mathfrak{S}}
\DeclareMathOperator{\Alt}{\mathfrak{A}}
\DeclareMathOperator{\Dih}{\mathfrak{D}}

\AtBeginDocument{\addtocontents{toc}{\protect\begin{multicols}{2}}}
\AtEndDocument{\addtocontents{toc}{\protect\end{multicols}}}

\usepackage{xcolor}

\begin{document}

\begin{abstract}
We classify the number of $k$-rational lines and conic fibrations on del Pezzo surfaces over a field $k$ in terms of relatively minimal surfaces and establish rational curve analogues of the inverse Galois problem for del Pezzo surfaces. We completely solve these problems in all degrees over all global, local and finite fields and provide new solutions of the inverse Galois problem in characteristic $2$. Our results generalise well-known theorems on cubic surfaces.
\end{abstract}

\maketitle

\begin{center}
    \scshape Contents
\end{center}
\vspace*{-2.75em}
\def\contentsname{\empty}
\setcounter{tocdepth}{1}
{\setlength{\parskip}{3pt}\tiny\tableofcontents}
\vspace*{-2.5em}

\section{Introduction} \label{sec:intro}

One of the jewels of classical algebraic geometry is the \emph{Cayley--Salmon theorem}: any smooth cubic surface over an algebraically closed field contains exactly $27$ lines. Determining the number of lines becomes more subtle when keeping track of fields of definition. Segre \cite{Seg49} proved\footnote{Segre's claim that his proof fails when $\ch{k} = 2$ is false; see \cite[\S3.1]{McK22} for more.} that the number of $k$-rational lines on a cubic surface over a field $k$ is in $\{0, 1, 2, 3, 5, 7, 9, 15,27\}$ and that all are realised over $\mathbb{Q}$. Schl\"afli \cite{Sch58} showed that, over $\mathbb{R}$, only four counts are realised: 3, 7, 15 and 27.

Cubic surfaces are \emph{del Pezzo surfaces} (smooth projective surfaces $X$ with ample anticanonical class $-K_X$, classified by degree $d := K_X^2 \in \{1,\dots,9\}$). We arrive at the focus of this paper: determining the possible numbers of $k$-rational lines on degree-$d$ del Pezzo surfaces over an arbitrary field $k$. We also consider the analogue for conic fibrations on these surfaces. For definitions of lines and conics, see Definition~\ref{def:lines and conics}; they coincide with the classical ones for degree at least $3$ (Remark~\ref{rem:line is line}).

These problems relate to the \emph{inverse Galois problem for del Pezzo surfaces}. The action of $G_k = \Gal(k_s/k)$ on the geometric Picard group $\Pic(\overline{X})$ of a degree-$d$ del Pezzo surface $X/k$ induces a homomorphism $\rho_X: G_k \rightarrow \Aut(\Pic(\overline{X}))$ to a subgroup non-canonically isomorphic to a group $G_d$ (see Lemma~\ref{lem:Gd}). We thus obtain $[H_X] = [\rho_X(G_k)]$ in the set $C(G_d)$ of subgroup conjugacy classes of $G_d$.

Denote by $\mathcal{S}_d(k)$ the set of isomorphism classes of degree-$d$ del Pezzo surfaces over $k$ and by $C(G_d,k)$ the subset of $C(G_d)$ represented by groups isomorphic to $\Gal(L/k)$ for $L/k$ some Galois extension, i.e.\ ``splitting groups permitted by the Galois theory of $k$''.

\begin{namedthm*}{\bf{Problem} $\boldsymbol{\IGP_{d}(k)}$}[Inverse Galois problem for del Pezzo surfaces] \label{prob:igp}
Determine the image of the map
\[
\gamma_d(k): \mathcal{S}_d(k) \rightarrow C(G_d,k), \quad
X \mapsto [H_X].
\]
\end{namedthm*}

Note that $\rho_X$ determines line/conic counts: there is a map $\delta_{d,1}(k): C(G_d,k) \rightarrow \mathbb{Z}_{\geq 0}$ mapping $[H] \in C(G_d,k)$ to the number of $k$-lines on $X \in \mathcal{S}_d(k)$ with $[H_X] = [H]$, were such $X$ to exist, with image in a finite set of ``counts permitted by $G_k$'' $\mathcal{C}_{d,1}(k)$ in a total set of possible counts $\mathcal{C}_{d,1}$. We have an analogous map and set $\delta_{d,2}(k):C(G_d,k) \rightarrow \mathcal{C}_{d,2}(k) \subset \mathcal{C}_{d,2}$ for conic fibrations.

\begin{namedthm*}{\bf{Problem} $\boldsymbol{\ICP_{d,e}(k)}$}[Inverse curve problems for del Pezzo surfaces] \label{prob:icp}
For $e \in \{1,2\}$, determine the image of the map
\[
\delta_{d,e}(k) \circ \gamma_d(k): \mathcal{S}_d(k) \rightarrow \mathcal{C}_{d,e}(k).
\]
\end{namedthm*}

We say that a solution is \emph{surjective} if the image is surjective onto $\mathcal{C}_{d,e}(k)$. We say that we have a \emph{simultaneous} surjective solution to both problems if $\left(\delta_{d,1}(k),\delta_{d,2}(k)\right) \circ \gamma_d(k)$ is surjective. Note that a solution to $\IGP_d(k)$ gives a solution to $\ICP_{d,e}(k)$.

The results above amount to solutions to $\ICP_{3,1}(k)$ for $k$ algebraically closed (where $\mathcal{C}_{3,1}(k) = \{27\}$), $k = \mathbb{Q}$ (where $\mathcal{C}_{3,1}(\mathbb{Q}) = \mathcal{C}_{3,1} = \{0, 1, 2, 3, 5, 7, 9, 15,27\}$) and $k = \mathbb{R}$ (where $\mathcal{C}_{3,1}(\mathbb{R}) = \{3,7,15,27\}$). Over separably closed fields, the analogue of Cayley--Salmon (describing $\mathcal{C}_{d,e}(k)$ for $k$ algebraically closed) is well-known (see Theorem~\ref{thm:geomcounts}), and the analogue of Segre's result (describing $\mathcal{C}_{d,e}$) is fairly easy to establish (see Theorems~\ref{thm:line} and \ref{thm:conic}). Little is known otherwise about the inverse curve problems. We develop techniques for computing rational curve counts and Galois actions to obtain new results. We completely solve the inverse curve problems over all global, local and finite fields and solve the inverse curve problems when the status of the Galois analogue is unknown.

The spirit of this work is close to that of recent advances in enumerative geometry by Kass and Wickelgren \cite{KW} and Larson and Vogt \cite{LV} on counting lines on cubic surfaces and bitangents to plane quartics (equivalent to lines on the corresponding degree-2 del Pezzo surface double-covering the plane with this branch curve) with arithmetic weights.

\subsection{Results}

\subsubsection{Classification by blowup type}
Any del Pezzo surface is relatively minimal (Definition~\ref{def:relmin}) or the blowup of a relatively minimal surface. Further, line and conic counts are determined by Galois action on the surface and blowup locus. Then it is natural to approach the inverse curve problems by classifying del Pezzo surfaces in terms of their \emph{blowup types} (Definition~\ref{def:blowuptype}), meaning the ways in which they can be constructed via blowups of relatively minimal surfaces. This can be thought of as a minimal model program for del Pezzo surfaces over arbitrary fields. We prove that, together with line and conic counts, the \emph{intersection invariant} (Definition~\ref{def:PicDet}) determines blowup type. 

\begin{thm} \label{thm:lci}
For a del Pezzo surface $X$, denote by $L_X$, $C_X$ and $I_X$ the line count, conic count and intersection invariant of $X$, respectively. Then the triple $(L_X,C_X,I_X)$ determines the blowup type of $X$ uniquely.
\end{thm}

The invariant $I_X$ appears to be novel in the del Pezzo setting and is a powerful tool in distinguishing blowup types, otherwise computationally untenable in low degree.

Using Theorem~\ref{thm:lci}, we build tables of all blowup types with line and conic counts (Algorithm~\ref{alg:class}), found in Appendix~\ref{sec:Tables}. One may view this as an extension of Iskovskikh's classification of relatively minimal del Pezzo surfaces over an arbitrary field (Theorem \ref{thm:isk}). With these tables, we attack the inverse curve problems via Algorithm~\ref{alg:ICPalg}.

\subsubsection{Finite fields}
Over finite fields, much is known about the inverse Galois problem: it was solved for $d \geq 2$ by Loughran and Trepalin \cite{LT20}. However, the methods break down in degree $1$ and one encounters additional difficulties (see Remark~\ref{rem:IGPfordP1s}).

Despite this, Banwait, Fit\'e and Loughran \cite{BFL19} determined the possibilities for the trace of Frobenius on del Pezzo surfaces of degree $1$ over finite fields. Trace of Frobenius and curve counts are both ``shadows'' of the Galois action, but the latter admits more possible values. By generalising their computational and theoretical techniques, we completely solve the inverse curve problems over all finite fields.

In Tables~\ref{table:FiniteFieldLines} and \ref{table:FiniteFieldConics} below, the presence of a count in a column means that it is possible exactly when the condition labelling the column is true.

\begin{thm} \label{thm:FiniteFieldLines}
Over the finite field $\mathbb{F}_q$, the inverse line problem $\ICP_{d,1}(\mathbb{F}_q)$ admits a surjective solution for all $q$ when $d \geq 5$. The solutions for $d \leq 4$ are given in Table~\ref{table:FiniteFieldLines}.
\end{thm}

\begin{longtable}{|l|l|l|l|l|l|l|l|l|}\caption{Line counts over $\mathbb{F}_q$} \label{table:FiniteFieldLines}\\
\hline
Degree & All $q$ & $q \geq 3$ & $q \geq 4$ & $q \geq 5$ & $q \geq 7$ & $q \geq 9$ & $q \geq 11$ & \makecell[l]{$q = 16$ or \\ $q \geq 19$} \\
\hline
$4$ & \makecell[l]{$0$, $1$, $2$, $4$, \\ $8$} & -- & $16$ & -- & -- & -- & -- & -- \\
\hline
$3$ & \makecell[l]{$0$, $1$, $2$, $3$, \\ $5$, $9$, \\ $15$
} & -- & $7$ & -- & $27$ & -- & -- & -- \\
\hline
$2$ & \makecell[l]{$0$, $2$, $4$, $6$, \\ $8$} & $12$, $20$ & -- & $16$, $32$ & -- & $56$ & -- & -- \\
\hline
$1$ & \makecell[l]{$0$, $2$, $4$, $6$, \\ $12$} & $8$, $14$, $20$ & $24$, $30$ & $40$ & $26$, $72$ & -- & $126$ & $240$ \\
\hline
\end{longtable}

\begin{thm} \label{thm:FiniteFieldConics}
Over the finite field $\mathbb{F}_q$, the inverse conic problem $\ICP_{d,2}(\mathbb{F}_q)$ admits a surjective solution for all $q$ when $d \geq 5$. The solutions for $d \leq 4$ are given in Table~\ref{table:FiniteFieldConics}.
\end{thm}

\begin{longtable}{|l|l|l|l|l|l|l|l|l|}\caption{Conic counts over $\mathbb{F}_q$} \label{table:FiniteFieldConics}\\
\hline
Degree & All $q$ & $q \geq 3$ & $q \geq 4$ & $q \geq 5$ & $q \geq 7$ & $q \geq 9$ & $q \geq 11$ & \makecell[l]{$q = 16$ or \\ $q \geq 19$} \\
\hline
$4$ & $0$, $2$, $4$, $6$ & -- & $10$ & -- & -- & -- & -- & -- \\
\hline
$3$ & \makecell[l]{$0$, $1$, $2$, $3$, \\ $5$, $9$, $15$} & -- & $7$ & -- & $27$ & -- & -- & -- \\
\hline
$2$ & $0$, $2$, $4$, $6$ & $8$, $12$, $14$ & -- & $26$, $60$ & $24$ & $126$ & -- & -- \\
\hline
$1$ & $0$, $2$, $4$, $6$ & \makecell[l]{$12$, $18$, $30$, \\ $36$} & $24$, $90$ & -- & \makecell[l]{$72$, $252$, \\$270$} & -- & $756$ & $2160$ \\
\hline
\end{longtable}

\subsubsection{Infinite fields}

Having completely settled the problems for finite fields, we turn to infinite fields, for which we solve the inverse curve problems for $d \geq 4$.
\begin{thm} \label{thm:icp}
For $k$ an infinite field, the problems $\ICP_{d,e}(k)$ admit surjective solutions for all $d \geq 4$.
\end{thm}

The main content of Theorem~\ref{thm:icp} is for $d = 4$: $\IGP_4(k)$, thus $\ICP_{4,e}(k)$, is solved for infinite $k$  with $\Char(k) \neq 2$ by Kunyavski\u{\i}, Skorobogatov and Tsfasman \cite[\S6]{KST89}. However, the representation-theoretic methods and conic bundle constructions of \emph{loc.\ cit.}\ break down in characteristic $2$. By careful construction of Artin--Schreier analogues, we solve new cases of the inverse Galois problem and arrive at our solution.

Under modest assumptions on the ground field, we obtain \emph{simultaneous} surjective solutions to the inverse curve problems in all degrees.

\begin{thm} \label{thm:ModestICP}
Let $k$ be an infinite field with separable extensions of degrees $2$ to $8$. Then $\ICP_{d,e}(k)$ has a surjective solution for all $d$ and $e$. If $\Char(k) \neq 2$ and $k$ admits a Galois extension of degree 3, then the problems admit simultaneous surjective solutions.
\end{thm}

This generalises a result of Kulkarni and Vemulapalli \cite[Thm.~3.1.1]{KulVem24} on bitangents of plane quartics (equivalent to $\ICP_{2,1}(k)$) by removing characteristic and Galois assumptions. Further, it completely solves the inverse curve problems for local fields. Since the archimedean local fields $\mathbb{R}$ and $\mathbb{C}$ are already settled, we restrict to non-archimedean ones (e.g.,\ the field of $p$-adic numbers $\mathbb{Q}_p$).

\begin{cor} \label{cor:LocalICP}
For $k$ a non-archimedean local field, the problems $\ICP_{d,e}(k)$ admit surjective solutions for all $d$ and $e$. If $\Char(k) \neq 2$, these solutions are simultaneous.
\end{cor}

\subsubsection{Hilbertian fields}

We can say even more over Hilbertian fields (Definition~\ref{def:hilbertian}), including global fields and function fields in finitely many variables. We show that all blowup types exist, thus solving the inverse curve problems simultaneously.

\begin{thm} \label{thm:hilb}
For all $1 \leq d \leq 9$ and $k$ a Hilbertian field, all blowup types of degree-$d$ del Pezzo surfaces exist. In particular, the inverse curve problems $\ICP_{d,e}(k)$ admit simultaneous surjective solutions.
\end{thm}

The properties of Hilbertian fields imply the existence of relatively minimal del Pezzo surfaces in each degree, but there remains the problem of determining whether there exist certain closed points in general position (Definition~\ref{def:general position del Pezzo}) on these surfaces. Through consideration of the geometric and arithmetic properties of symmetric powers, we are able to show that such points do in fact exist.

\begin{rem}
The only fields for which the inverse curve problems are not completely solved are infinite non-Hilbertian non-local non-closed fields which fail to have low-degree separable extensions. We also extend Russo's work \cite{Rus02} over $\mathbb{R}$ to real closed fields (Remark~\ref{rem:RealClosedFields}), thus dispensing with virtually all of the fields usually encountered in arithmetic geometry.
\end{rem}

\subsection{Related literature}

We now give a brief summary of the literature.

\textbf{Arbitrary fields:} For $d \geq 5$, Zaitsev \cite[Thms.~1.2,~1.7]{Zai23} showed that $\IGP_d(k)$ has a surjective solution for any field $k$.

\textbf{Finite fields:} Work of Trepalin \cite{Tre20} and Loughran--Trepalin \cite{LT20} solves $\IGP_d(\mathbb{F}_q)$ for all $q$ when $d \in \{2,3,4\}$, following earlier work of Rybakov \cite{Ryb05} and Rybakov--Trepalin \cite{RybTre17}. Banwait, Fit\'e and Loughran \cite[Thm.~1.7]{BFL19} determined the possibilities for the trace of Frobenius acting on the geometric Picard group and proved a quantitative form of $\IGP_d(\mathbb{F}_q)$ for $d \leq 6$ as $q \rightarrow \infty$; later cohomological results of Das \cite{Das20}, Bergvall and Gounelas \cite{BG23} and Bergvall \cite{Ber24} imply quantitative results. For more, we direct the reader to \cite[\S1]{Ber24}.

\textbf{Infinite fields:} A surjective solution to $\text{IGP}_4(k)$ for $k$ infinite with $\ch{k} \neq 2$ is given by Kunyavski\u{\i}, Skorobogatov and Tsfasman \cite[Cor.~6.10]{KST89} (see the final line before \cite[Prop.~6.4]{KST89}). This work also features a surjective solution of the analogous problem for conic bundles. Elsenhans and Jahnel gave a surjective solution to $\IGP_3(\mathbb{Q})$ \cite{EJ15} and partial results for $\IGP_2(k)$ for $k$ infinite with $\Char(k) \neq 2$ \cite{EJ19a,EJ19b}. Kulkarni and Vemulapalli \cite{KulVem24} gave partial results on $\ICP_2(k)$, while Kulkarni \cite{Kul21} established partial results for $\IGP_1(\mathbb{Q})$. Work of Russo \cite{Rus02} implies $\IGP_d(\mathbb{R})$ for all $d \in \{1,\dots,9\}$.

\textbf{Enriched counts:} The last decade has seen a surge of activity in enumerative geometry over non-closed fields, in part due to the \emph{enriched enumerative geometry} program. In enriched enumerative geometry, one uses tools from motivic homotopy theory to give weighted counts of geometric objects, analogous to the signed counts that are characteristic of real enumerative geometry. Kass and Wickelgren gave an enriched count of lines on cubic surfaces \cite{KW}, while Larson and Vogt proved some results on enriching the count of bitangents to plane quartics \cite{LV} (which comes with additional subtleties and limitations due to issues of orientability). 

While our work and enriched enumerative geometry share the goal of counting curves ``arithmetically,'' our methods and results are distinct. For example, the enriched count of lines on a cubic surface $X$ over $\mb{F}_q$ is a sum of quadratic forms over $\mb{F}_q$, with each line $L$ contributing a quadratic form depending on the geometry of the inclusion $L\subset X$ and the trace form of the field of definition of $L$ over $\mb{F}_q$. Over most fields, these trace forms prevent one from deducing anything about the number of $k$-rational lines from the enriched count of lines. A notable exception is $k=\mb{R}$, over which the trace form of any non-trivial extension has signature 0. In this case, the signature of the enriched count recovers the signed count of 3 lines on a real cubic surface (proved by Segre \cite{Seg42} but only first noticed much later by Finashin--Kharlamov and Okonek--Teleman \cite{FK12,OT14}). In the context of our results, this signed count of 3 just implies that a real cubic surface must have at least 3 real lines.

\subsection{Outline}

In Section~\ref{sec:delpezzo} we give background on del Pezzo surfaces. In Section~\ref{sec:classification} we present results for distinguishing blowup types and Algorithm~\ref{alg:class} for constructing tables of blowup types with line and conic counts, also proving Theorem~\ref{thm:lci}. In Section~\ref{sec:inverseproblems} we review previous results and present Algorithm~\ref{alg:ICPalg}. In Section~\ref{sec:finitefields} we prove Theorems~\ref{thm:FiniteFieldLines} and \ref{thm:FiniteFieldConics}. In Section~\ref{sec:infinitefields} we prove Theorems~\ref{thm:icp} and \ref{thm:ModestICP}, and in Section~\ref{sec:hilb} we prove Theorem~\ref{thm:hilb}. Appendix~\ref{sec:Tables} contains our tables for arbitrary fields and special tables for finite fields; Appendix~\ref{sec:general position appendix} contains computations on points in general position, and
Appendix~\ref{sec:Questions} contains open questions.

Our \texttt{Magma}~\cite{Magma} code is available at
\[
\text{\url{https://github.com/HappyUppal/Lines-on-del-Pezzo-surfaces.git}}. 
\] 
We also benefited from \href{https://people.maths.bris.ac.uk/~matyd/GroupNames/index.html}{GroupNames}, Tim Dokchitser's database of groups.

\subsection{Notation and conventions}

\subsubsection*{Algebra}

Unless otherwise specified, we work over an arbitrary field $k$ with separable closure $k_s$ and algebraic closure $\overline{k}$, setting $G_k:=\Gal(k_s/k)$. We denote by $\Sym_n$, $\Alt_n$, $\Dih_n$ and $C_n$ the symmetric and alternating groups on $n$ letters, dihedral group of order $2n$ and cyclic group of order $n \in \mathbb{Z}_{\geq 2}$, respectively. For a prime power $q$, we denote by $\mathbb{F}_q$ the finite field of $q$ elements.

\subsubsection*{Geometry}

We denote by $\Spec R$ the spectrum of a ring $R$ and by $\mathbb{A}^n_R$ and $\mathbb{P}^n_R$ the affine and projective spaces of dimension $n$ over $R$ respectively, omitting the subscript when it is clear from context. A \emph{variety} $X$ over a field $k$ is an integral separated scheme of finite type over $\Spec k$; we denote its $k$-rational points by $X(k)$ and its canonical divisor class by $K_X$. We write $X/k$ when emphasising the base field. Given an extension $L/k$, we write $X_L$ for the base-change $X \times_{\Spec k} \Spec L$, reserving $\overline{X} := X_{\overline{k}}$. \emph{Curves} and \emph{surfaces} are varieties of dimensions $1$ and $2$ respectively. We denote by $C\cdot D$ the intersection product of two divisors on a surface $X$, setting $C^2:= C \cdot C$. We say a variety $X$ is \emph{rational} (resp.\ \emph{unirational}) if there is a birational (resp.\ dominant) map $\PP^{\dim(X)} \dasharrow X$ over $k$. Let $\pi:Y\to X$ be the blowup of a surface $X/k$ in a point $P\in X(k)$. The \emph{exceptional divisor} of $\pi$ is the preimage $E:=\pi^{-1}(P)$. We have $E\cong \mathbb{P}^{1}_{k}$ and $E^2=-1$. The \emph{strict transform} $\tilde{C}\subset Y$ of a curve $C\subset X$ is the closure of $\pi^{-1}(C\setminus\{P\})$, and its \emph{pullback} or \emph{total transform} is $\pi^*C:=\tilde{C}+rE$, where $r$ is the multiplicity of $C$ at $P$ \cite[Prop.~V.3.6]{Har77}.

\subsection*{Acknowledgements}

We thank Julian Demeio, Aaron Landesman, Daniel Loughran, Adam Morgan, Ross Paterson, Alexei Skorobogatov, Tony V\'arilly-Alvarado and Sameera Vemulapalli for helpful discussions and Rosa Winter for sharing code. EK was supported by FWO grant GYN-D9843-G0B1721N.
SM was supported by the Simons Foundation and NSF grants DMS-2202825 and DMS-2502365.
SS and HU were supported by the University of Bristol and the Heilbronn Institute for Mathematical Research.

\section{Del Pezzo surfaces} \label{sec:delpezzo}

This section is devoted to gathering relevant facts about del Pezzo surfaces.

\begin{defn}
A \emph{del Pezzo surface} over a field $k$ is a smooth projective geometrically integral variety $X/k$ of dimension 2 such that the anticanonical class $-K_X$ is ample. The \emph{degree} of $X$ is the intersection product $d:=K_X^2\in\{1,\ldots,9\}$.
\end{defn}

\begin{rem}
For $d\geq 3$, the class $-K_X$ is very ample, inducing an \emph{anticanonical embedding} $X\hookrightarrow\mb{P}^d$. The image is a smooth cubic surface in $\mathbb{P}^3$ for $d=3$ and a smooth intersection of quadrics in $\mathbb{P}^4$ for $d = 4$. For degree $2$ and $1$, the anticanonical morphism is a double cover of $\mathbb{P}^2$ branched over a quartic curve and a rational map to $\mathbb{P}^1$ defined away from one point respectively.
\end{rem}

\begin{rem}
The anticanonical morphism when $d = 2$ and the bianticanonical morphism when $d = 1$ are both double covers: they give rise to the \emph{Geiser and Bertini involutions} respectively, which we will use later.
\end{rem}

The following result classifies del Pezzo surfaces over a separably closed field.

\begin{thm} \cite[Thm.~9.4.4]{Poo17} \label{thm:ksclass} 
Let $X$ be a del Pezzo surface of degree $d$ over a field $k$. Then $X_{k_s}$ is either isomorphic to $\mathbb{P}^1_{k_s} \times \mathbb{P}^1_{k_s}$ (in which case $d = 8$) or $X_{k_s}$ is isomorphic to the blowup of $\mathbb{P}^2_{k_s}$ in $9-d$ points in general position.
\end{thm}

For the definition of general position, see Definition~\ref{def:general position plane}.

\subsection{Rational curves on del Pezzo surfaces}

As described in the introduction, our primary objects of study are lines and conic families on del Pezzo surfaces.

\begin{defn}\label{def:lines and conics}
    Let $X/k$ be a del Pezzo surface. A curve $L\subset X_{k_s}$ is called a \emph{line} (or \emph{exceptional curve}) if $-K_X\cdot L=1$ and $L\cdot L=-1$. A reduced curve $C\subset X_{k_s}$ is called a \emph{conic} if $-K_X\cdot C=2$ and $C^2=0$. If $X$ contains a $k$-rational conic $C/k$, we call it a \emph{conic del Pezzo surface}. 
\end{defn}

\begin{rem} \label{rem:line is line} We observe some consequences of this definition:
\begin{enumerate}
    \item Lines and conics merit their names: they have anticanonical degrees 1 and 2 respectively, thus they correspond, for degree $d\geq 3$, to such curves in $\mb{P}^d$ lying on an anticanonical embedding.
    \item A conic $C$ induces a fibration $\pi:X\to B$ of $X$ into conics over a genus-0 curve $B$ (see e.g.~\cite[Proof of Lem.~3.5]{Str21}), hence a \emph{conic fibration/family} on $X$.
    \item A line $L \subset X_{k_s}$ is $G_k$-invariant if and only if it is defined over $k$ (since it is isolated). Further, a conic $C \subset X_{k_s}$ is $G_k$-invariant if and only if it represents a class in $\Pic(X) \subset \Pic(\overline{X})$. To see this latter point, note that the associated morphism $X \rightarrow \mathbb{P}^1$ is $G_k$-stable as soon as the class of $C$ is.
\end{enumerate}
\end{rem}

As mentioned in the introduction, it is relatively easy to determine the number of ``geometric'' lines and conic families on a del Pezzo surface.

\begin{thm} \cite[Table~3.1]{Der06} \label{thm:geomcounts}
The number of lines and conic families\footnote{Note that the class of a conic is a \textit{ruling}, as defined in \cite[Def.~3.1]{Der06}.} on a degree-$d$ del Pezzo surface over $k_s$ is given in the following table, in which 8q denotes degree-8 del Pezzo surfaces isomorphic to quadrics and 8b denotes blowups of the plane in a point.

\begin{longtable}{ |c|c|c|c|c|c|c|c|c|c|c| }
\caption{Geometric counts on del Pezzo surfaces} \\
\hline
Degree & 9 & 8q & 8b & 7 & 6 & 5 & 4 & 3 & 2 & 1 \\
\hline
Number of lines & 0 & 0 & 1 & 3 & 6 & 10 & 16 & 27 & 56 & 240 \\ 
\hline
Number of conic families & 0 & 2 & 1 & 2 & 3 & 5 & 10 & 27 & 126 & 2160 \\ 
\hline
\end{longtable}
\end{thm}

It is also possible to compute the total set of possibilities for line and conic counts.
\pagebreak
\begin{thm} \label{thm:line}
The complete list of possibilities for the number of $k$-rational lines on a degree-$d$ del Pezzo surface is given in the following table.

\begin{longtable}{ |c|c| }
\caption{Possible line counts on del Pezzo surfaces} \\
\hline
Degree & Possible line counts \\
\hline
9 & 0
\\ 
\hline
8 & {0,1}
\\ 
\hline
7 & {1,3}
\\ 
\hline
6 & {0,2,6}
\\
\hline 
5 & {0,1,2,4,10}
\\
\hline 
4 & {0,1,2,4,8,16}
\\
\hline 
3 & {0,1,2,3,5,7,9,15,27}
\\
\hline
2 & {0,2,4,6,8,12,16,20,32,56}
\\
\hline
1 & {0,2,4,6,8,12,14,20,24,26,30,40,60,72,126,240}
\\
\hline
\end{longtable}
\end{thm}

\begin{thm} \label{thm:conic}
The complete list of possibilities for the number of $k$-rational conic families on a degree-$d$ del Pezzo surface is given in the following table.
\begin{longtable}{ |c|c| }
\caption{Possible conic counts on del Pezzo surfaces} \\
\hline
Degree & Possible conic family counts \\
\hline
9 & 0
\\ 
\hline
8 & {0,2}
\\ 
\hline
7 & {0,2}
\\ 
\hline
6 & {0,1,3}
\\
\hline 
5 & {0,1,2,3,5}
\\
\hline 
4 & {0,2,4,6,10}
\\
\hline 
3 & {0,1,2,3,5,7,9,15,27}
\\
\hline
2 & {0,2,4,6,8,12,14,24,26,60,126}
\\
\hline
1 & {0,2,4,6,12,18,24,30,36,72,90,252,270,756,2160}
\\
\hline
\end{longtable}
\end{thm}

\begin{rem}
In the case $d = 3$, note that classes of lines and conic fibrations are in bijection under the map $D \mapsto -K_X - D$.
\end{rem}

\begin{rem}
It is easy to derive these theorems from our tables in Appendix~\ref{sec:Tables}. It is also feasible to compute the list of all possible line-conic count pairs, but it is quite large.
\end{rem}

\subsection{Relatively minimal del Pezzo surfaces} 

Determining the existence of blowup types relies on two ingredients: existence of the underlying relatively minimal surface, and existence of suitable points to blow up. We discuss the first of these here.

\begin{defn} \label{def:relmin}
A surface $X$ over a field $k$ is \emph{relatively minimal} if $X$ does not contain a $G_k$-orbit of pairwise skew exceptional curves.
\end{defn}

\begin{rem}
Relatively minimal surfaces include \emph{minimal} surfaces, i.e.\ relatively minimal surfaces $X/k$ such that any birational map from a surface $Y/k$ to $X$ is a morphism.
\end{rem}

Iskovskikh classified relatively minimal del Pezzo surfaces over arbitrary fields.

\begin{thm}\cite[Cor.,~\S3]{Isk79} \label{thm:isk}
For $X$ a degree-$d$ del Pezzo surface over a field $k$:
\begin{enumerate}[align=left]
\item[$d = 9$:] $X$ is relatively minimal and $\Pic(X) = \langle L \rangle \cong \mathbb{Z}$ with $L^2 = 1$.
\item [$d=8$:] $X$ is relatively minimal if and only if $X \cong Q \subset \mathbb{P}^3$ for some smooth quadric $Q$. Either $\Pic(X) = \langle H \rangle \cong \mathbb{Z}$ with $H^2 = 2$ (non-split) or $\Pic(X) = \langle H,V\rangle \cong \mathbb{Z} \times \mathbb{Z}$ with $H^2 = V^2 = 0$ and $H \cdot V = 1$ (split).

\item[$d=7$:] $X$ is never relatively minimal.

\item[$d \in \{3,5,6\}$:] $X$ is relatively minimal if and only if $\Pic(X) = \langle -K_X \rangle \cong \mathbb{Z}$.

\item[$d \in \{1,2,4\}$:] $X$ is relatively minimal if and only if either $\Pic(X) = \langle -K_X \rangle \cong \mathbb{Z}$ or $\Pic(X) = \langle -K_X,C\rangle \cong \mathbb{Z} \times \mathbb{Z}$ for $C$ a conic class.
\end{enumerate}
\end{thm}

\begin{notn}
    For $d \in \{3,5,6,9\}$, we denote by $S_d$ a relatively minimal degree-$d$ del Pezzo surface. For $d \in \{1,2,4,8\}$, we denote by $S_{d,I}$ (resp. $S_{d,\II}$) a relatively minimal degree-$d$ del Pezzo surface of Picard rank $1$ (resp. $2$).
\end{notn}

\subsubsection{Degree at least 5}

From Theorem~\ref{thm:isk}, it is easy to see that there exist relatively minimal del Pezzo surfaces of degrees 9 and 8 over any field: $\mathbb{P}^2$ and $\mathbb{P}^1 \times \mathbb{P}^1$ respectively. Further, relatively minimal del Pezzo surfaces of degree $7$ do not exist. The situation is not too much more delicate for degrees 5 and 6, as the following result of Zaitsev shows.

\begin{thm} \cite[Thm.~1.5,~Thm.~1.7]{Zai23} \label{thm:Minimal5and6}
Let $k$ be an arbitrary field. Then there exists a relatively minimal del Pezzo surface of degree $5$ over $k$ if and only if $k$ admits a separable extension of degree $5$. There exists a relatively minimal del Pezzo surface of degree $6$ over $k$ if and only if there exist separable extensions of degrees $2$ and $3$.

\end{thm}

\begin{proof}
Necessity is immediate from consideration of the corresponding Galois actions on lines. We give a sketch of the proof of sufficiency for the sake of detailing explicit constructions. For degree $5$, take a separably closed point of degree $5$ on a smooth conic $C \subset \mathbb{P}^2$ and blow it up to obtain a degree-4 del Pezzo surface: the strict transform of $C$ is a line, which we may blow down to obtain the desired relatively minimal degree-5 del Pezzo surface. For degree 6, take a separably closed point of degree 2 along a line $L\subset \mathbb{P}^2$ and a separable point of degree 3 on a smooth conic $C \subset \mathbb{P}^2$; blowing up, we again obtain a degree-4 del Pezzo surface, now with two skew rational lines (the strict transform of $L$ and the strict transform of the conic $C'$ through all five underlying points). Blowing these lines down, we obtain a relatively minimal degree-6 del Pezzo surface.
\end{proof}

\subsubsection{Degree at most 4}

In degree at most 4, the situation is more complex.

Over a finite field $\mathbb{F}_q$, the existence of relatively minimal del Pezzo surfaces of degree at least $2$ over $\mathbb{F}_q$ for all $q$ is confirmed by the results of Loughran--Trepalin \cite{LT20}. In degree $1$, we are able to establish the existence of all blowup types of del Pezzo surfaces over $\mathbb{F}_q$ for all $q$, including the two relatively minimal kinds, which appears to be novel; see Lemma~\ref{lem:RelMinsOverFq}.

Over a general infinite field, one has the following result.

\begin{thm}
For $k$ any infinite field with $\Char(k) \neq 2$, there exist del Pezzo surfaces of types $S_{4,I}$, $S_{4,\II}$ and $S_{2,\II}$ over $k$ as soon as they are permitted by the Galois theory of $k$.
\end{thm}

\begin{proof}
This is a direct consequence of Corollary~\ref{cor:igpconic2}.
\end{proof}

\begin{rem}
Note that this leaves open the existence of del Pezzo surfaces of types $S_3$ (relatively minimal cubic surfaces), $S_{2,I}$, $S_{1,I}$ and $S_{1,\II}$.
\end{rem}

In degree $3$, we have the following result, which furnishes a relatively minimal del Pezzo surface of degree $3$ for fields which are not closed under the operation of taking cube roots.

\begin{lem} \label{lem:RelMinCubic}
Let $k$ be a field with $\Char(k) \neq 3$. If there exists $\alpha\in k^\times \setminus k^{\times 3}$, then the following cubic surface in $\mathbb{P}^3$ is relatively minimal:
\[
X_0 + X_1^3 + X_2^3 + \alpha X_3^3 = 0.
\]
\end{lem}

\begin{proof}
For $k$ perfect containing a cubic root of unity, see \cite[Example~21.9]{Man86}. In view of Coombes' result that del Pezzo surfaces are separably split \cite{Coo88}, one may drop perfectness. Further, one may drop the assumption of a cubic root of unity, as this only enlarges the Galois orbits of lines on the resulting cubic.
\end{proof}

Over Hilbertian fields, all relatively minimal types exist, as we shall see in Section~\ref{sec:hilb}.

\subsection{Points in general position}

We now turn to the second ingredient in blowup types: the existence of points in general position.

\subsubsection{Defining general position}

We begin by defining general position in the plane.

\begin{defn}\label{def:general position plane}
    We say a $G_k$-stable set $\{P_1,\ldots,P_n\}\subseteq\mb{P}^2(k_s)$ of $n \leq 8$ points are in \emph{general position} if no 3 lie on a line, no 6 lie on a conic, and if $n=8$, any cubic through $P_1,\ldots,P_8$ does not have a double point at $P_i$ for some $i$.
\end{defn}

It is classical that if $P_1,\ldots,P_n\in\mb{P}^2$ are in general position, then $\Bl_{P_1,\ldots,P_n}\mb{P}^2$ is a del Pezzo surface. We mimic this definition for other del Pezzo surfaces (as suggested in \cite[Rem.~26.3]{Man86}).

\begin{defn}\label{def:general position del Pezzo}
    Let $X$ be a degree-$d$ del Pezzo surface over $k$ and $\{P_1,\dots,P_n\}\subseteq X(k_s)$ be a $G_k$-stable set of points with $n \leq d-1$. We say that $P_1,\ldots,P_n$ are in \emph{general position} if $\Bl_{P_1,\dots,P_n}X$ is a del Pezzo surface.
\end{defn}

Using Theorem~\ref{thm:ksclass}, we can derive equivalent conditions for Definition~\ref{def:general position del Pezzo}.

\begin{lem}\label{lem:def of gen position}
Let $X$ be a degree-$d$ del Pezzo surface over $k$. The elements of a $G_k$-stable set of points $\{P_1,\ldots,P_n\}\subseteq X(k_s)$ (with $n\leq d-1$) are in general position if and only if the following conditions are both satisfied:
\begin{enumerate}[(i)]
\item For all $1 \leq r \leq \min\{4,n\}$, no $r$ of the $P_i$ lie on a curve of self-intersection $r-2$ and arithmetic genus 0.
\item For $n=d-1$, the $P_i$ do not all lie on an anticanonical curve with a double point at one of them.
\end{enumerate}
\end{lem}

\begin{proof}
All curves on a del Pezzo surface have self-intersection at least $-1$, and self-intersection of the strict transform decreases by at least $1$ when we blow up along a curve (and by at least $4$ when the blowup is at a double point, with the arithmetic genus decreasing by $1$ in this case), so necessity is immediate. 

For sufficiency when $X$ is not relatively minimal of degree $8$, choose a collection of $9-d$ skew $k_s$-lines in $X_{k_s}$ and blow down to $\mathbb{P}^2_{k_s}$. Denote the blowdowns of the skew lines by $Q_1,\dots,Q_{9-d}$, and identify the points $P_1,\dots,P_n$ with their images under blowing down. Set $A := \{P_i\}_{i=1}^n\cup\{Q_j\}_{j=1}^{9-d}$. The points in $A\subset\mb{P}^2(k_s)$ are in general position (according to Definition~\ref{def:general position plane}) if and only if the $P_i$ are in general position on $X$ (according to Definition~\ref{def:general position del Pezzo}). Note that:
\begin{enumerate}
\item The strict transform of a line through three of the points in $A$ is an $r$-curve of arithmetic genus $0$ through $r-2$ of the $P_i$ for some $r \leq 3$.
\item The strict transform of a conic through six of the points in $A$ is an $r$-curve of arithmetic genus $0$ through $r-2$ of the $P_i$ for some $r \leq 4$.
\item In the case $n = d-1$, the strict transform of a cubic through all of the points in $A$ with a double point at one of them is either:
\begin{enumerate}
\item A $(d-3)$-curve of arithmetic genus $0$ through all $d-1$ of the $P_i$ (if the double point is at some $Q_i$), or
\item An anticanonical curve of arithmetic genus $1$ through all $d-1$ of the $P_i$ with a double point at one of them (if the double point is at some $P_i$).
\end{enumerate}
\end{enumerate}
For sufficiency when $X$ is relatively minimal of degree $8$, we may simply blow up at one of the $P_i$ and then argue as in the case $d = 7$. Alternatively, we may argue by using the isomorphism between the blowup of the plane in two points and the blowup of a quadric in one point (see Section~\ref{sec:coincidences}).
\end{proof}

\begin{notn}
    For $S$ a relatively minimal del Pezzo surface, denote by $\Bl_{d_1,\dots,d_r}S$ the blowup of $S$ in closed points $P_1,\dots,P_r$ in general position with $
    \text{deg}(P_i) = d_i$.
\end{notn}

\begin{defn} \label{def:blowuptype}
    We say that a del Pezzo surface $X$ has \emph{blowup type} $\Bl_{d_1,\dots,d_r}S$ if $X \cong \Bl_{d_1,\dots,d_r}S$ for some relatively minimal del Pezzo surface $S$.
\end{defn}

With the above definitions in hand, we begin to consider the question of existence of closed points in general position.

\subsubsection{General position over infinite fields}

Here we discuss closed points in general position over infinite fields; among these are the Hilbertian fields, for which we can say more (see Proposition~\ref{prop:HilbertianGenPos}). We postpone finite fields until Section~\ref{sec:finitefields}.

The existence of \emph{$k$-rational} points in general position is straightforward in most cases.

\begin{lem}\label{lem:rational points in general position}
Let $k$ be an infinite field, and let $X$ be a del Pezzo surface over $k$ of degree $d \geq 2$. Assume that $X$ contains a $k$-rational point; if $d=2$, further assume that the rational point does not lie on four exceptional curves or the ramification locus of $X$. Then there exists a $(d-1)$-tuple of rational points in general position on $X$. 
\end{lem}
\begin{proof}
Such $X$ is $k$-unirational (by \cite{Man86} and \cite[Thm.~1.1]{Kol02} for $d\geq 3$, and \cite[Cor.~3.3]{STVA14} for $d=2$). Since $k$ is infinite, unirationality implies that $X(k)$ is dense. It remains to show that density of rational points is sufficient for the result, which we now do by inductively building a tuple.

Let $Z$ be the union of all lines on $X$ and set $U:=X\setminus Z$. Since $Z$ is proper closed, there exists $P \in U(k)$. Note that $P$ is a rational point in general position on $X$.

Now assume that there exist $P_1,\ldots,P_n\in X(k)$ in general position for some $n<d-2$. By definition, $P_1,\ldots,P_n\in U_n(k)$, where $U_n=X\setminus Z_n$ for $Z_n$ the proper closed subset of all lines on $X$, all conics on $X$ passing through one of the $P_i$, and, for each $r\leq n$, all curves on $X$ of self-intersection $r-2$ and arithmetic genus 0. Let $Z_{n+1}$ denote the union of all curves of arithmetic genus 0, self-intersection $n-1$, and passing through the points $P_1,\ldots,P_n$. Set $U_{n+1}:=U_n\cap(X\setminus Z_{n+1})$. Since $Z_{n+1}$ is again a proper closed subset of $X$, we have $U_{n+1}\neq\emptyset$ and thus there exists $P_{n+1}\in U_{n+1}(k)$. Note that $P_1,\ldots,P_n\not\in U_{n+1}(k)$, as these points were removed with the locus $Z_{n+1}$. It follows that $P_1,\ldots,P_{n+1}$ is an $(n+1)$-tuple of points in general position.

We can apply the previous inductive step until $n=d-2$, in which case we must additionally ensure that $P_{n+1}$ is chosen so that there is no anticanonical curve through all of the $P_1,\dots,P_{n+1}$ with a double point at one of them. To see that this may be done, blow up at $P_1,\dots,P_{d-2}$ to obtain a del Pezzo surface $Y/k$ of degree $2$ with dense rational points. Note that there exists a point $P_{d-1} \in U_{d-2}(k)$ and not on the ramification divisor. Moreover, note that for any $P_i$ on the ramification divisor, it suffices to choose $P_{d-1}$ to further avoid the unique anticanonical curve with a double point at $P_i$. We thus obtain a general $(d-1)$-tuple of rational points on $X$.
\end{proof}

For our purposes, we will need to blow up del Pezzo surfaces in sets of not-necessarily-rational points in general position. This means that we need a more general result than Lemma~\ref{lem:rational points in general position}. We first use combinatorial arguments to construct closed points in general position on $\mb{P}^2$ over arbitrary infinite fields.

\begin{defn}
Let $1\leq n\leq 8$. Let $\lambda=(n_1,\ldots,n_8)$ denote the partition $n=\sum_{i=1}^8 i\cdot n_i$, and let $|\lambda|=\sum_{i=1}^8 n_i$ denote the length of the partition. We say that a set of closed points $M=\{P_1,\ldots,P_{|\lambda|}\}\subset\mb{P}^2_k$ \emph{corresponds to the partition} $\lambda$ if there exist disjoint subsets $M_1,\ldots,M_8\subseteq M$ such that
\begin{enumerate}[(i)]
\item $\bigsqcup_{i=1}^8 M_i=M$,
\item $|M_i|=n_i$, and
\item every element of $M_i$ is a closed point of degree $i$.
\end{enumerate}
\end{defn}

\begin{thm}\label{thm:gen position P2}
Let $k$ be an infinite field. Let $1\leq n\leq 8$. Let $\lambda=(n_1,\ldots,n_8)$ be a partition of $n$. If $k$ admits a separable extension of degree $i$ for each $n_i>0$, then there exists a set of closed points in general position in $\mb{P}^2_k$ corresponding to $\lambda$.
\end{thm}
\begin{proof}
    See Appendix~\ref{sec:general position appendix}.
\end{proof}

\begin{rem}
Our proof of Theorem~\ref{thm:gen position P2} holds for many finite fields as well. See Appendix~\ref{sec:general position appendix} for sufficient (but not necessarily necessary) lower bounds on $|k|$.
\end{rem}

If $X$ is a rational del Pezzo surface over an infinite field $k$, then we obtain points in general position on $X$ from Theorem~\ref{thm:gen position P2}. This allows us to get points in general position on all blowup types of del Pezzo surfaces of degree $d\geq 5$.

\begin{prop}\label{prop:general position dP}
    Let $k$ be an infinite field. Let $5\leq d\leq 9$ and $1\leq n\leq d-1$. Let $\lambda=(n_1,\ldots,n_8)$ be a partition of $n$, and assume that $k$ admits a separable extension of degree $i$ for each $n_i>0$. Every blowup type of degree-$d$ del Pezzo surface contains a surface $X$ with a set of closed points in general position corresponding to $\lambda$.
\end{prop}
\begin{proof}
    Every del Pezzo surface is birational to some relatively minimal del Pezzo surface, so we may restrict our attention to relatively minimal blowup types. Recall that for $d\geq 5$, any degree-$d$ del Pezzo surface containing a $k$-rational point is rational \cite[Thm.~9.4.8]{Poo17}. Such a surface then contains a set of closed points in general position corresponding to $\lambda$ by Theorem~\ref{thm:gen position P2}. Thus, given a relatively minimal degree-$d$ del Pezzo surface $X$ over $k$, it suffices to find a relatively minimal degree-$d$ del Pezzo surface $Y$ over $k$ with $Y(k)\neq\varnothing$. We will treat each degree in turn.
    
    The case  $d=9$ follows from Theorem~\ref{thm:gen position P2}. The  case $d=8$ follows from the case $d = 9$ as the blowup of a quadric surface in a rational point is isomorphic to the blowup of the plane in two points (see Proposition~\ref{prop: coincidence result}). The case $d=7$ is vacuous, as there are no relatively minimal degree-7 del Pezzo surfaces.
    
    Finally, as discussed in Theorem~\ref{thm:Minimal5and6} and its proof, whenever a relatively minimal del Pezzo surface of degree $d \in \{5,6\}$ exists, one may be constructed via blowup of the plane. Thus the result follows easily from the case $d = 9$.
\end{proof}

For $d\leq 4$, we can no longer rely on rationality for points in general position. We can, however, progress with assumptions on the ground field. We will address Hilbertian fields, finite fields, and local fields in their respective sections.

\section{Classification by blowup type}\label{sec:classification}

\subsection{Intersection invariant}

As we construct new del Pezzo surfaces by blowing up, we must determine whether new blowups are isomorphic to old ones. Where coincidences occur, we will give explicit geometric proofs. It will then remain to prove that no further coincidences occur. By calculating line and conic counts and Picard ranks for certain blowups, we can sometimes distinguish them. We then narrow down to small sets of ``potentially coincident'' surfaces. Given such a pair of surfaces, we may use the additional structure of the intersection form on their Picard groups to distinguish them. For this, we introduce the following invariant.

\begin{defn} \label{def:PicDet}
Let $X$ be a smooth projective surface such that $\Pic(\overline{X})$ is torsion-free and finitely generated (e.g.\ a del Pezzo surface) and $\mathcal{B} = \{B_1,\dots,B_r\}$ be a basis of the Picard group of $X$. Denote by $G_X(\mathcal{B}) \in M_r(\mathbb{Z})$ the Gram matrix associated to $\mathcal{B}$, i.e.\ the matrix with $ij^{\text{th}}$ entry the intersection product $B_i\cdot B_j$, and set $I_X = \det(G_X(\mathcal{B}))$. 
\end{defn}

\begin{rem} In the previous definition, as the notation suggests, $I_X$ is independent of the choice of basis $\mathcal{B}$. Indeed, given a new basis $\mathcal{C} = \{C_1,\dots,C_r\}$ of $\Pic(X)$, there exist integers $a_{ij} \in \mathbb{Z}$, $1 \leq i,j \leq r$, such that $C_i = \sum_{j=1}^r a_{ij}B_j$ for all $i \in \{1,\dots,r\}$ and that $A = (a_{ij}) \in \text{GL}_r(\mathbb{Z})$. Note that $C_i \cdot C_j = \mathbf{v} G_X(\mathcal{B})\mathbf{w}^{\text{T}}$, where $\mathbf{v} = (a_{i1},\dots,a_{ir})$ and $\mathbf{w} = (a_{j1},\dots,a_{jr})$. It readily follows that $G_X(\mathcal{C}) = A G_X(\mathcal{B}) A^{\text{T}}$. Taking determinants, we deduce independence, since $\det(A) = \det(A^{\text{T}}) = \pm 1$. 
\end{rem}

Essentially the same argument shows that $I_X$ is an isomorphism invariant.

\begin{lem}
If $X$ is a smooth projective surface with torsion-free finitely generated geometric Picard group and $Y$ is a surface isomorphic to $X$, then $I_X = I_Y$.
\end{lem}

\begin{proof}
The isomorphism $X \cong Y$ induces an isomorphism $\varphi: \Pic(X) \xrightarrow{\sim} \Pic(Y)$ preserving intersection products. Choosing bases for each Picard group, $\varphi$ is represented by a matrix $A \in \text{GL}_r(\mathbb{Z})$, and one has $v \cdot w = vG_Xw^{\text{T}}$ (surpressing the choice of basis in $\Pic(X)$). By preservation of intersection multiplicities, we deduce that $G_{X} = AG_{Y}A^{\text{T}}$. Since $\det(A) = \pm 1$, we deduce that $I_X = I_Y$.
\end{proof}

Note that $I_X$ takes the following values for $X$ a relatively minimal del Pezzo surface:
\[
I_X =
\begin{cases}
1 \text{ if $d = 9$}, \\
2 \text{ if $d = 8$ and $\rank\Pic(X) = 1$}, \\
d \text{ if $d \leq 6$ and $\rank\Pic(X) = 1$}, \\
-1 \text{ if $d = 8$ and $\rank\Pic(X) = 2$}, \\
-4 \text{ if $d \in \{1,2,4\}$ and $\rank\Pic(X) = 2$}.
\end{cases}
\]

Note also that $I_X$ transforms nicely under blowing up:

\begin{lem}
Let $X$ be a smooth projective surface with torsion-free finitely generated geometric Picard group, and let $Y = \Bl_{d_1,\dots,d_n}X$. 
Then $I_Y = (-1)^n I_X \prod_{i=1}^n d_i$.
\end{lem}

\begin{proof}
Let $\{B_1,\dots,B_r\}$ be a basis of $\Pic(X)$. Note that
\[
\Pic(Y) = \langle B'_1,\dots,B'_r,E_1,\dots,E_n\rangle \cong \mathbb{Z}^{r+n},
\]
where $B_i'$ is the pullback of $B_i$ and the $E_j$ are the (Galois orbits of) exceptional divisors over closed points in the blowup locus. Then the intersection matrix $M_{Y}$ of $\Pic(Y)$ with respect to this presentation is a block diagonal matrix with blocks the intersection matrix $M_X$ of $X$ and $-\text{diag}(d_1,\dots,d_n)$. Then $I_Y = (-1)^nI_X\prod_{i=1}^n d_i$.
\end{proof}

\begin{cor} \label{cor:PicDet}
Let $X$ and $Y$ be two relatively minimal del Pezzo surfaces over a field $k$, and suppose that two blowups $X' = \Bl_{d_1,\dots,d_r}X$, $Y' = \Bl_{e_1,\dots,e_s}Y$ satisfy
\[(\rank\Pic(X'),L_{X'},C_{X'}) = (\rank\Pic(Y'),L_{Y'},C_{Y'})\]
for $L_X$ (respectively, $C_X$) the line (respectively, conic) count for $X$. The surfaces $X'$ and $Y'$ are not isomorphic in the case that
\[(-1)^rI_X\prod_{i=1}^rd_i \neq (-1)^s I_Y\prod_{j=1}^s e_j.\]

\end{cor}

\subsection{Counting lines and conics on blowups} \label{sec:CountingLinesAndConics}

We now give formulae for counting lines and conic fibrations on blowups of del Pezzo surfaces of degrees 8 and 9.

\begin{notn}
For $X/k$ a del Pezzo surface, we denote by $L_X$ and $C_X$ the number of $k$-rational lines and $k$-rational conic families on $X$ respectively.
\end{notn}

In principle we require counts on blowups of lower-degree del Pezzo surfaces, but we will see that these counts offer us enough information to distinguish all blowup types up to isomorphism, at which point we are able to run through a list of possible Galois actions to compute the remaining counts.

\begin{notn}
Let $(P_1,\dots,P_r)$ be a tuple of separable closed points on a del Pezzo surface. Denote by $\mathbf{d} = (d_1,\dots,d_r)$ the tuple of their degrees, i.e.\ $d_i = \deg(P_i)$.
\begin{itemize}
    \item Let $\mathbf{d}(n)$ be the number of sub-multisets of $\{d_1,\dots,d_r\}$ with sum $n$.
    \item Let $\mathbf{d}(m;n)$ be the number of pairs of disjoint sub-multisets of $\{d_1,\dots,d_r\}$ with sums $m$ and $n$.
\end{itemize}
\end{notn}

In Tables~\ref{table:9LineOrigins},~\ref{table:9ConicOrigins},~\ref{table:8LineOrigins}, we assign numerical labels for each source of curves. When positive, these numbers coincide with the degree of the associated complete linear system. We use one or more asterisks to differentiate sources associated to the same linear system.

\begin{lem} \label{lem:lccounts}
Let $\{P_1,\dots,P_r\} \in \mathbb{P}^2(k_s)$ be $r \leq 8$ points in general position and set $X = \Bl_{P_1,\dots,P_r}\mathbb{P}^2_{k_s}$. Denote by $[n \mid m_1,\dots,m_r]$ the system $|\mathcal{O}(n) - \sum_{i=1}^rm_iP_i|$ of curves in $\mathcal{O}(n)$ with a singularity of order at least $m_i$ at each $P_i$, indicating repeated entries via exponents, e.g.\ $[2\mid1,1] = [2 \mid 1^2]$. The origins of the lines and conics on $X_{k_s}$ in $\mathbb{P}^2$, together with $k$-rational counts, are described by Tables~\ref{table:9LineOrigins} and \ref{table:9ConicOrigins} respectively. In both cases, Geiser duals and Bertini duals are identified.
\end{lem}

\begin{longtable}{|l|l|l|l|l|l|} \caption{Origin of lines on the blowup of $S_9$} \label{table:9LineOrigins}  \\
\hline
Number & System & $k$-rational count & Geiser & Bertini \\
\hline
0 & (Exceptional divisors) & $\mathbf{d}(1)$ & 3 & 6 \\
\hline
1 & $[1|1^2]$ & $\mathbf{d}(2)$ & 2 & 5 \\
\hline
2 & $[2|1^5]$ & $\mathbf{d}(5)$ & 1 & 4 \\
\hline
3 & $[3|2^1,1^6]$ & $\mathbf{d}(1;6)$ & 0 & 3 \\
\hline
4 & $[4|2^3,1^5]$ & $\mathbf{d}(8)\mathbf{d}(3)$ & -- & 2 \\
\hline
5 & $[5|2^6,1^2]$ & $\mathbf{d}(8)\mathbf{d}(2)$ & -- & 1 \\
\hline
6 & $[6|3^1,2^7]$ & $\mathbf{d}(8)\mathbf{d}(1)$ & -- & 0 \\
\hline
\end{longtable}

\begin{longtable}{|l|l|l|l|l|l|} \caption{Origin of conics on the blowup of $S_9$} \label{table:9ConicOrigins}  \\
\hline
Number & System & $k$-rational count & Geiser & Bertini \\
\hline
1 & $[1|1^1]$ & $\mathbf{d}(1)$ & 5 & 11 \\
\hline
2 & $[2|1^4]$ & $\mathbf{d}(4)$ & 4 & 10 \\
\hline
3 & $[3|2^1,1^5]$ & $\mathbf{d}(1;5)$ & 3 & 9 \\
\hline
4 & $[4|2^3,1^4]$ & $\mathbf{d}(3;4)$ & 2 & 8 \\
\hline
5 & $[5|2^6,1^1]$ & $\mathbf{d}(1;6)$ & 1 & 7 \\
\hline
7 & $[7|4^1,3^1,2^6]$ & $\mathbf{d}(8)\mathbf{d}(1;1)$ & -- & 5 \\
\hline
8 & $[8|4^1,3^4,2^3]$ & $\mathbf{d}(8)\mathbf{d}(1;3)$ & -- & 4 \\
\hline
9 & $[9|4^2,3^5,2^1]$ & $\mathbf{d}(8)\mathbf{d}(1;2)$ & -- & 3 \\
\hline
10 & $[10|4^4,3^4]$ & $\mathbf{d}(8)\mathbf{d}(4)$ & -- & 2 \\
\hline
11 & $[11|4^7,3^1]$ & $\mathbf{d}(8)\mathbf{d}(1)$ & -- & 1 \\
\hline
4* & $[4|3^1,1^7]$ & $\mathbf{d}(8)\mathbf{d}(1)$ & -- & 8* \\
\hline
5* & $[5|3^1,2^3,1^4]$ & $\mathbf{d}(8)\mathbf{d}(1;3)$ & -- & 7* \\
\hline
6* & $[6|3^2,2^4,1^2]$ & $\mathbf{d}(8)\mathbf{d}(2;2)$ & -- & 6* \\
\hline
7* & $[7|3^4,2^3,1^1]$ & $\mathbf{d}(8)\mathbf{d}(1;3)$ & -- & 5* \\
\hline
8* & $[8|3^7,1^1]$ & $\mathbf{d}(8)\mathbf{d}(1)$ & -- & 4* \\
\hline
\end{longtable}

\begin{proof}
We have the following equalities:

\begin{enumerate}
\item For the dimension of the linear system, we have
\[
\dim\left|\mathcal{O}(n) - \sum_{i=1}^r m_i P_i\right| = \frac{n(n+3)}{2} - \sum_{i=1}^r\binom{m_i+1}{2}.
\]
\item For $C \sim \mathcal{O}(n)$ and $\widetilde{C}$ the strict transform of $C$ on $\Bl_{P_1,\dots,P_r}
\mathbb{P}^2$,
\[
\begin{aligned}
p_a(\widetilde{C}) = \binom{n-1}{2} - \sum_{i=1}^r\binom{m_i}{2}, \quad
\widetilde{C}^2 = n^2 - \sum_{i=1}^r m_i^2. 
\end{aligned}
\]
\end{enumerate}
Then lines are given by solutions of
\[
\frac{n(n+3)}{2} - \sum_{i=1}^r\binom{m_i+1}{2} = \binom{n-1}{2} - \sum_{i=1}^r\binom{m_i}{2} = n^2 - \sum_{i=1}^r m_i^2 + 1 = 0,
\]
while conic families are given by solutions to
\[\frac{n(n+3)}{2} - \sum_{i=1}^r\binom{m_i+1}{2} -1  = \binom{n-1}{2} - \sum_{i=1}^r\binom{m_i}{2} = n^2 - \sum_{i=1}^r m_i^2 = 0.
\]
Relevant computations can be found in Derenthal's thesis \cite[Ch.~3]{Der06}.
\end{proof}

\begin{cor} \label{cor:PlaneCounts}
We have the following formulas for $X = \Bl_{d_1,\dots,d_r}S_9$:
\[
\begin{aligned}
L_X & = \mathbf{d}(1) + \mathbf{d}(2) + \mathbf{d}(5) + \mathbf{d}(1;6) + \mathbf{d}(8)(\mathbf{d}(1) + \mathbf{d}(2) + \mathbf{d}(3)), \\
C_X & = \mathbf{d}(1) + \mathbf{d}(4) + \mathbf{d}(1;5) + \mathbf{d}(1;6) + \mathbf{d}(3;4) \\
& + \mathbf{d}(8)(3\mathbf{d}(1) + \mathbf{d}(4) + \mathbf{d}(1;1) + \mathbf{d}(1;2) +  3\mathbf{d}(1;3) +  \mathbf{d}(2;2)).
\end{aligned}
\]
\end{cor}

\begin{lem} \label{lem:QuadricCounts}
Let $S_8/k$ be a degree-$8$ del Pezzo surface of quadric type, and set $\delta = \rank \Pic (S_8) - 1$. Let $X/k$ be the blowup of $S_8$ in points $P_1,\dots,P_r \in S_8(k_s)$ in general position. Let us denote by $[a,b \mid m_1,\dots,m_r]$ the system $|\mathcal{O}(a,b) - \sum_{i=1}^r m_iP_i|$. The origins of the lines on $X_{k_s}$ in $S_8$, together with their $k$-rational counts, are described by Table~\ref{table:8LineOrigins}, and the origins of the conic families are described in Table~\ref{table:8ConicOrigins}. In both cases, Geiser duals and Bertini duals are identified.
\end{lem}

\begin{longtable}{|l|l|l|l|l|} \caption{Origin of lines on the blowup of $S_8$} \label{table:8LineOrigins}  \\
\hline
Number & System & $k$-rational count & Geiser & Bertini \\
\hline
0 & (Exceptional divisors) & $\mathbf{d}(1)$ & 4 & 8 \\
\hline
1 & $[1,0|1^1]$, $[0,1|1^1]$ & $2\delta\mathbf{d}(1)$ & 3 & 7 \\
\hline
2 & $[1,1|1^3]$ & $\mathbf{d}(3)$ & 2 & 6 \\
\hline
3 & $[2,1|1^5]$, $[1,2|1^5]$ & $2\delta\mathbf{d}(5)$ & 1 & 5 \\
\hline
4 & $[2,2|2^1,1^5]$ & $\mathbf{d}(1;5)$ & 0 & 4 \\
\hline
4* & $[3,1|1^7]$, $[1,3|1^7]$ & $2\delta
\mathbf{d}(7)$ & -- & 4* \\
\hline
5 & $[3,2|2^2,1^5]$, $[2,3|2^2,1^5]$ & $2\delta\mathbf{d}(7)\mathbf{d}(2)$ & -- & 3 \\
\hline
6 & $[3,3|1^3,2^4]$ & $\mathbf{d}(7)\mathbf{d}(3)$ & -- & 2 \\
\hline
7 & $[4,3|2^6,1^1]$, $[3,4|2^6,1^1]$ & $2\delta\mathbf{d}(7)\mathbf{d}(1)$ & -- & 1 \\
\hline
8 & $[4,4|3^1,2^6]$ & $\mathbf{d}(7)\mathbf{d}(1)$ & -- & 0 \\
\hline
\end{longtable}
\pagebreak
\begin{longtable}{|l|l|l|l|l|} \caption{Origin of conics on the blowup of $S_8$} \label{table:8ConicOrigins}  \\
\hline
Number & System & $k$-rational count & Geiser & Bertini \\
\hline
1 & $[1,0]$, $[0,1]$ & $2\delta$ & 7 & 15 \\
\hline
2 & $[1,1|1^2]$ & $\mathbf{d}(2)$ & 6 & 14 \\
\hline
3 & $[2,1|1^4]$, $[1,2|1^4]$ & $2\delta\mathbf{d}(4)$ & 5 & 13 \\
\hline
4 & $[2,2|2^1,1^4]$ & $\mathbf{d}(1;4)$ & 4 & 12 \\
\hline
5 & $[3,2|2^2,1^4]$, $[2,3|2^2,1^4]$ & $2\delta\mathbf{d}(2;4)$ & 3 & 11 \\
\hline
6 & $[3,3|1^2,2^4]$ & $\mathbf{d}(2;4)$ & 2 & 10 \\
\hline
7 & $[4,3|2^6]$, $[3,4|2^6]$ & $2\delta\mathbf{d}(6)$ & 1 & 9 \\
\hline
8 & $[4,4|3^2,2^3,1^2]$ & $\mathbf{d}(7)\mathbf{d}(2;2)$ & -- & 8 \\
\hline
9 & $[5,4|4^1,2^6]$, $[4,5|4^1,2^6]$ & $2\delta\mathbf{d}(7)\mathbf{d}(1)$ & -- & 7 \\
\hline
10 & $[5,5|4^1,3^2,2^4]$ & $\mathbf{d}(7)\mathbf{d}(1;2)$ & -- & 6 \\
\hline
11 & $[6,5|4^1,3^4,2^2]$, $[5,6|4^1,3^4.2^2]$ & $2\delta\mathbf{d}(7)\mathbf{d}(2;4)$ & -- & 5 \\
\hline
12 & $[6,6|4^1,3^4,2^1]$ & $\mathbf{d}(7)\mathbf{d}(1;2)$ & -- & 4 \\
\hline
13 & $[2,1|1^4]$, $[1,2|1^4]$ & $2\delta\mathbf{d}(4)$ & -- & 3 \\
\hline
14 & $[7,7|4^5,3^2]$ & $\mathbf{d}(7)\mathbf{d}(2)$ & -- & 2 \\
\hline
15 & $[8,7|4^7]$, $[7,8|4^7]$ & $2\delta\mathbf{d}(7)$ & -- & 1 \\
\hline
6* & $[3,3|3^1,2^1,1^5]$ & $\mathbf{d}(7)\mathbf{d}(1;1)$ & -- & 10* \\
\hline
6** & $[4,2|2^3,1^4]$, $[2,4|2^3,1^4]$ & $2\delta\mathbf{d}(7)\mathbf{d}(3)$ & -- & 10** \\
\hline
7* & $[4,3|3^1,2^3,1^3]$, $[3,4|3^1,2^3,1^3]$ & $2\delta\mathbf{d}(7)\mathbf{d}(1;3)$ & -- & 9* \\
\hline
8* & $[5,3|3^1,2^5,1^1]$, $[3,5|3^1,2^5,1^1]$ & $\mathbf{d}(7)\mathbf{d}(1;1)$ & -- & 8* \\
\hline
9* & $[5,4|3^3,2^3,1^1]$, $[4,5|3^3,2^3,1^1]$ & $2\delta\mathbf{d}(7)\mathbf{d}(1;3)$ & -- & 9* \\
\hline
10* & $[5,5|3^5,2^1,1^1]$ & $\mathbf{d}(7)\mathbf{d}(1;1)$ & -- & 6* \\
\hline
10** & $[6,4|3^4,2^3]$, $[4,6|3^4,2^3]$ & $2\delta\mathbf{d}(7)\mathbf{d}(3)$ & -- & 6** \\
\hline
\end{longtable}

\begin{proof}
We replicate our proof when blowing up from the plane. For a quadric $Q$ over a separably closed ground field, we have $\Pic(Q) = \Pic(\mathbb{P}^1 \times \mathbb{P}^1) = \langle \mathcal{O}(1,0), \mathcal{O}(0,1) \rangle \cong \mathbb{Z} \times \mathbb{Z}$. We have the following formulas:
\[
\begin{aligned}
D & := \dim\left|\mathcal{O}(a,b) - \sum_{i=1}^r m_i P_i\right| = ab + a + b - \sum_{i=1}^r \binom{m_i + 1}{2}, \\
I & :=\widetilde{C}^2 = 2ab - \sum_{i=1}^r m_i^2,
\end{aligned}
\]
where $\widetilde{C}$ is the strict transform of a curve in $|\mathcal{O}(a,b) - \sum_{i=1}^r m_i P_i|$.
Solving for $D = 0$ and $I = -1$ gives the lines, while solving for $D = 1$ and $I = 0$ gives the conics.
\end{proof}

\begin{cor}
We have the following formulas for $X = \Bl_{d_1,\dots,d_r}S_8$:
\[
\begin{aligned}
L_X & = \mathbf{d}(1) + \mathbf{d}(3) + \mathbf{d}(1;5) + 2\delta(\mathbf{d}(1) + \mathbf{d}(5)) \\
    & + \mathbf{d}(7)(\mathbf{d}(1) + \mathbf{d}(3) + 2\delta(1 + \mathbf{d}(1) + \mathbf{d}(2))), \\
C_X & = \mathbf{d}(2) + \mathbf{d}(1;4) + \mathbf{d}(2;4) + 2\delta
    (1 + \mathbf{d}(4) + 2\mathbf{d}(6) + \mathbf{d}(2;4)) \\
    & + \mathbf{d}(7)
    (\mathbf{d}(2) + 2\mathbf{d}(1;1)+ 2\mathbf{d}(1;2) + \mathbf{d}(2;2))\\
    & + 2\delta\mathbf{d}(7)
    (1 + 2\mathbf{d}(1) + 3\mathbf{d}(3) + \mathbf{d}(1;1) + \mathbf{d}(1;2) + 2\mathbf{d}(1;3)).
\end{aligned}
\]
\end{cor}

\subsection{Coincidence of blowups} \label{sec:coincidences}
In this section, we note the following isomorphisms of blowups between blowups of relatively minimal del Pezzo surfaces.

\begin{prop}\label{prop: coincidence result}
We have the following coincidences of blowups:
\begin{enumerate}
    \item $\Bl_{2n-1} S_{8,\II} \cong \Bl_{2n-1,1}S_9$ for $n \in \{1,2,3,4\}$.
    \item $\Bl_1 S_{8,I} \cong \Bl_2 S_9$.
    \item $\Bl_1 S_6 \cong \Bl_3 S_{8,I}$.
    \item $\Bl_1 S_5 \cong \Bl_5 S_9$.
    \item $\Bl_2 S_5 \cong \Bl_5 S_{8,I}$.
\end{enumerate}
\end{prop}

\begin{proof}
We deal with each case in turn.
\begin{enumerate}
    \item
    Blowing up the projective plane in a rational point $P$ and a closed point $Q$ of degree $2n-1$ with underlying geometric points $Q_1,\dots,Q_{2n-1}$ such that the set of points $\{P,Q_1,\dots,Q_{2n-1}\}$ is in general position, note that the strict transforms of the lines joining $P$ and each $Q_i$ form an orbit of $2n-1$ pairwise skew lines; contracting these, one obtains $S_{8,\II}$. The converse direction follows from the description of the origins of lines given in Lemma~\ref{lem:QuadricCounts} readily imply the converse.
    \item
    The proof is essentially the same as the previous case: blowing up a conjugate pair of points in the plane, the strict transform of the line between them is an exceptional curve which may be blown down; the two exceptional divisors of this blowup map down to the (conjugate) fibres of the geometric rulings through the point contracted onto.
    \item
    Blowing up $S_{8,I}$ in a general-position cubic point, the strict transform of the unique plane section containing this point is an exceptional curve which one may contract to obtain $S_6$; the six lines on the blowdown correspond to the three pairs of conic fibres through the blowup points, and so the converse and relative minimality are easily seen.
    \item
    As discussed in the proof of Theorem~\ref{thm:Minimal5and6}, blowing up the plane in a general-position quintic point, the strict transform of the unique conic containing this point is a line; contracting it, we obtain $S_5$.
    \item
    Suppose that we have a general-position quintic point on $S_{8,I}$. It follows from Lemma~\ref{lem:QuadricCounts} that the strict transforms of the unique curves with classes $\mathcal{O}(2,1)$ and $\mathcal{O}(1,2)$ through the blowup locus on the base-change to $k_s$ form a conjugate pair of lines; blowing them down, we obtain $S_5$ with a general-position quadratic point, and the converse is easily seen. \qedhere
\end{enumerate}
\end{proof}

We also have a result of Manin and Iskovskikh, which implies that there are no blowup coincidences involving relatively minimal del Pezzo surfaces of degree at most $4$.

\begin{prop} \label{prop:ManinIskovskikh} Let $k$ be a field. Consider relatively minimal del Pezzo surfaces $X/k$ of degree $1,2,3$, and also $4$ under the condition that $X(k)$ is empty. If two such surfaces of the same degree are birational, then they are isomorphic.
\end{prop}

\begin{proof}
If the degree is $1,2$ or $3$, this follows from \cite[Thm.~5.8]{Man67}. The remaining case follows from \cite[Thm.~3.3]{Isk70}). See also the introduction of \cite{Sko86}. Perfectness in all cases is unnecessary due to \cite{Coo88}.
\end{proof}

\subsection{Algorithm 1}
We now detail an algorithm for classifying del Pezzo surfaces up to blowup type.

\begin{algorithm}[H] 
\label{alg:class}
\caption{\bf Classifying del Pezzo surfaces up to blowup type}

\KwIn{A degree $d \in \{1,\dots,9\}$}
\KwOut{Tables $T$ in Appendix~\ref{sec:GenFieldTables}}

\begin{enumerate}
\item Create an empty list $C$, an empty multiset $S$ and an empty table $T$ with columns labelled ``Type'', ``$(L_X,C_X)$'',``$I_X$'',``$G_X$'' and ``$\#C$''.
\item Add to $C$ each conjugacy class of $G_d$ with a representative.
\item For each blowup type of a del Pezzo surface of degree $d$, add to $S$ the triple consisting of the blowup type and its intersection invariant and Picard rank.
\item Sort the multiset $S$ into sub-multisets consisting of types $S'$ with equal intersection invariant and Picard rank.
\item For each sub-multiset $S'$ in $S$:
    \begin{enumerate}
    \item If $S'$ has size $1$, then remove $S'$ from $S$ and add to $T$ a corresponding row for the blowup type.
    \item Otherwise, pass to the next sub-multiset.
    \end{enumerate}
\item For each remaining sub-multiset $S'$ in $S$:
    \begin{enumerate}
        \item If all elements in $S'$ are blow-ups of a del Pezzo surface of degree $8$ or $9$, compute line and conic counts; use these to distinguish and use the coincidence result (Proposition~\ref{prop: coincidence result}) to check for isomorphisms. Add corresponding rows to $T$ with line and conic counts and remove the sub-multiset.
        \item Otherwise, distinguish by comparing possible conic counts. Remove $S'$ and add corresponding rows to $T$.
    \end{enumerate}
    \item For each class $c$ in $C$, compute the intersection invariant and Picard rank.
        \begin{enumerate}
            \item If this pair corresponds to only one row, add $1$ to that row's $\#C$ entry and add teh representative to the $G_X$ column. If that row does not have line and conic counts, compute them via Galois action.
            \item If this pair corresponds to multiple rows, compute the line and conic count to determine the correct row, and add the line and conic counts,  and $1$ to the $\#C$ entry, in that row and the representative to the $G_X$ entry.
        \end{enumerate}
\end{enumerate}
\end{algorithm}

\begin{rem}
We make some observations on Algorithm~\ref{alg:class}.
\begin{enumerate}
    \item Note that Algorithm~\ref{alg:class} can easily be adjusted to store all subgroup conjugacy classes for each blowup type, not just the number of them; however, this necessitates notation for these classes and is infeasible to represent for low degree, when the number of conjugacy classes is incredibly large.
    \item That the above algorithm runs successfully, in particular at Step (6)(b), is not obvious, but it amounts to distinguishing the two types $X = \Bl_{6,2}S_9$ and $Y = \Bl_{3}S_{4,\II}$: by Corollary~\ref{cor:PlaneCounts}, we have $C_X = 0$, while $C_Y > 0$ as $Y$ is the blowup of a surface with a rational conic family.
    \item We deduce that line and conic count is determined by blowup type. This can be seen more directly via an adaptation of the arguments of Section~\ref{sec:CountingLinesAndConics}.
\end{enumerate}
\end{rem}

\section{Inverse problems}\label{sec:inverseproblems}

We now discuss the inverse problems. Recall the following Galois representation.

\begin{defn} \label{def:galrep}
For $X$ a smooth projective surface over a field $k$, we denote by $\rho_X: G_k \rightarrow \Aut(\Pic(\overline{X}))$ the associated Galois representation on the geometric Picard group. The image $H_X := \rho_X(G_k)$ of $\rho_X$ is called the \emph{splitting group} of $X$.
\end{defn}

For $X/k$ a del Pezzo surface, $H_X$ is contained in the following subgroup $G_d$.

\begin{lem} \label{lem:Gd}
For $1 \leq d \leq 9$ and $X/k$ a degree-$d$ del Pezzo surface, let $G_d \leq \Aut(\Pic(\overline{X}))$ denote the subgroup of automorphisms fixing the canonical class $K_{\overline{X}}$ and preserving intersection products. We then have
\[
G_d \cong
\begin{cases}
0, & d=9, \\
C_2, & d=8, \\
W(R_{9-d}), & d \leq 7,
\end{cases}
\]
where $W(R_{9-d})$ is the Weyl group of a root system $R_{9-d}$ as in Table~\ref{table:root systems}.
\end{lem}

\begin{proof}
For $d \geq 8$, the result is trivial; for $d \leq 7$, see \cite[Thm.~23.9(ii)]{Man86}. 
\end{proof}

\begin{longtable}{|l|l|l|l|l|l|l|l|}
    \caption{Root systems and their Weyl groups}\label{table:root systems}
    \\ \hline
    $d$ & 7 & 6 & 5 & 4 & 3 & 2 & 1 \\
    \hline
    $R_{9-d}$ & $A_1$ & $A_2+A_1$ & $A_4$ & $D_5$ & $E_6$ & $E_7$ & $E_8$\\
    \hline
    $W(R_{9-d})$ & $C_2$ & $\Dih_6$ & $\Sym_5$ & $C_2^4\rtimes\Sym_5$ & $W(E_6)$ & $W(E_7)$ & $W(E_8)$ \\
    \hline
\end{longtable}

\subsubsection{Galois action on blowups}

Let $X$ be a degree-$d$ del Pezzo surface over $k$. By the fundamental theorem of Galois theory, $H_X \cong\Gal(\ell_X/k)$ for some field extension $\ell_X/k$.

Now suppose that $Z=\{P_1,\ldots,P_n\}\subseteq X(k_s)$ is a $G_k$-stable set of points in general position, so that $Y:=\Bl_Z X$ is a del Pezzo surface over $k$; form $\rho_Y$ and $\ell_Y/k$ as before. In addition to the homomorphisms $\rho_X$ and $\rho_Y$, we have another homomorphism
\[\rho_Z:G_k\to\Sym_n\]
via the action of $G_k$ on $P_1,\ldots,P_n$. Again, $\rho_Z(G_k) \cong\Gal(\ell_Z/k)$ for some extension $\ell_Z/k$.

\begin{prop}\label{prop:galois action on blowups}
    In the setting above, we have isomorphisms
    \[\Gal(\ell_Y/k)\cong\Gal(\ell_X\ell_Z/k)\cong\Gal(\ell_X/k)\times_{\Gal(\ell_X\cap\ell_Z/k)}\Gal(\ell_Z/k),\]
    where $\ell_X\ell_Z$ denotes the compositum of the fields $\ell_X$ and $\ell_Z$ as subfields of $k_s$.
\end{prop}
\begin{proof}
    The second isomorphism is a standard result in Galois theory, so it suffices to address the first. Let $E_i$ be the exceptional divisor associated to $P_i$. Then
    \[\Pic{\overline{Y}}\cong\Pic{\overline{X}}\oplus\langle E_1,\ldots,E_n\rangle.\]
    Further, for the canonical inclusions $\iota_X:\Aut(\Pic{\overline{X}})\hookrightarrow\Aut(\Pic{\overline{Y}})$ and $\iota_Z:\Sym_n\hookrightarrow\Aut(\Pic{\overline{Y}})$, the homomorphism $\rho_Y$ factors as
    \[\rho_Y:G_k\xrightarrow{(\rho_X,\rho_Z)}\Aut(\Pic{\overline{X}})\oplus\Sym_n\xrightarrow{\iota_X \iota_Z}\Aut(\Pic{\overline{Y}}).\]
    Note that the product $\iota_X \iota_Z$ is injective, since the images of $\iota_X$ and $\iota_Z$ act on disjoint subsets of a basis of $\Pic{\overline{Y}}$. Together, this implies that
    \[
        \ker(\rho_Y)=\ker(\rho_X)\cap\ker(\rho_Z)
        \cong\Gal(k_s/\ell_X)\cap\Gal(k_s/\ell_Z)
        \cong\Gal(k_s/\ell_X\ell_Z),
    \]
    and it follows that $\im(\rho_Y)\cong\Gal(\ell_X\ell_Z/k)$.
\end{proof}

\subsection{Inverse Galois problem for del Pezzo surfaces} \label{sec:IGP}

We henceforth identify $G_d$ with $W(R_{9-d})$ for $d \leq 7$. The isomorphism of Lemma~\ref{lem:Gd} is well-defined up to conjugation. Consequently, we obtain a map
\[
    \gamma_d(k): \mathcal{S}_d(k)  \rightarrow C(G_d,k), \quad X \mapsto [H_X],
\]
where $\mathcal{S}_d(k)$ denotes the set of isomorphism classes of degree-$d$ del Pezzo surfaces over $k$ and $C(G_d,k)$ denotes the subset of the set of conjugacy classes of subgroups $C(G_d)$ of $G_d$ represented by Galois groups over $k$. Given a class $[H] \in C(G_d,k)$, we say that $X/k \in \mathcal{S}_d(k)$ \emph{realises $[H]$} if $[H_X] = [H]$.

Recall that $\IGP_d(k)$ asks whether the map $\gamma_d(k)$ is surjective.

\begin{rem}
While conjugate subgroups of $G_d$ are isomorphic, the converse does not hold in general, and so multiple conjugacy classes may be represented by the same group.
\end{rem}

\subsection{Inverse curve problems}

Since the Galois action sends lines to lines, knowledge of the Galois action determines the line count on a del Pezzo surface. There is a partial converse: given full knowledge of the Galois action on lines (i.e.\ not just the number of size-$1$ orbits, but all orbits), we can recover the Galois action, in particular the Galois action on higher-degree rational curves, from the following result.

\begin{lem}\label{lem:basis for pic}
    \cite[Prop.~25.1]{Man86} Let $X/k$ be a del Pezzo surface of degree $d$. If $d=8$, assume that $X_{k_s} \not\cong \mathbb{P}^1_{k_s} \times \mathbb{P}^1_{k_s}$. Then there exists a basis $\{L_0,\ldots,L_{9-d}\}$ of $\Pic{\overline{X}}$, with $L_1,\ldots,L_{9-d}$ exceptional curves and $-K_X=3L_0-\sum_{i=1}^{9-d}L_i$.
\end{lem}

An explicit instance of this determination is given in the following proposition.

\begin{prop}\label{prop:detecting rational conics}
    \cite[Prop.~5.2]{FLS18} Let $X/k$ be a del Pezzo surface of degree $d\leq 7$. Then $X$ admits a $k$-rational conic fibration if and only if there exist two (possibly identical) non-empty Galois orbits of lines $\mf{L}_1$ and $\mf{L}_2$ satisfying $(\mf{L}_1+\mf{L}_2)^2=0$.
\end{prop}

\begin{rem}
    Our definition of conic fibration differs from the one in \cite{FLS18} in that we do not require the base curve of the associated projective morphism to contain a rational point. This is why \cite[Prop.~5.2]{FLS18} gives an equivalence in our setting but only a one-way implication in their setting.
\end{rem}

In particular, for each $e \in \mathbb{Z}_{\geq 1}$, we have a well-defined map
\[
\delta_{d,e}(k): C(G_d,k) \rightarrow \mathcal{C}_{d,e}(k) \subset \mathbb{Z}_{\geq 0}
\]
which associates to a class $[H] \in C_d(G,k)$ the number of families of degree-$e$ rational curves on $X \in \mathcal{S}_d(k)$ with $[H_X] = [H]$ when such $X$ exists. That $\mathcal{C}_{d,e}(k)$ is a finite set follows from the fact that $\Pic(\overline{X})$ is finitely generated and $-K_X$ is ample.

Recall that $\ICP_{d,e}(k)$ asks whether $\delta_{d,e}(k)$ is surjective. In view of Lemma~\ref{lem:basis for pic}, a surjective solution of $\text{IGP}_{d}(k)$ over $k$ implies a surjective solution of $\text{ICP}_{d,e}$ for all $e \geq 1$.

\subsection{Previous results} \label{sec:igpresults}

We now review what is known about these inverse problems.

\subsubsection{Degree at least 5}

In high degree, the inverse Galois problem is settled.
\begin{thm} \label{thm:dP5IGP}
For $d \geq 5$ and any field $k$, $\IGP_d(k)$ admits a surjective solution.
\end{thm}
\begin{proof}
Let $X$ be a degree-$d$ del Pezzo surface over $k$. We separate into degrees.
\begin{itemize}
\item[$d=9$:] Here, $\Pic(\overline{X}) \cong \mathbb{Z}$ has trivial $G_k$-action. This is always realised by $\mathbb{P}^2$.
\item[$d=8$:] First, assume that $\overline{X} \cong \Bl_1 S_9$. Then $\Pic(\overline{X})$ has trivial action. This is realised by $\Bl_1\mathbb{P}^2$.
Now assume that $\overline{X} \cong \mathbb{P}^1_{k_s} \times \mathbb{P}^1_{k_s}$. Then $\Pic(\overline{X})$ has Galois action by $C_2$ if and only if $X$ is non-split and trivial action otherwise. The split quadric surface $\mathbb{P}^1 \times \mathbb{P}^1$ exists over any field. The existence of a non-split quadric is equivalent to the existence of a degree-$2$ separable extension $L/k$; given such an extension, a non-split quadric is realised by the Weil restriction $R_{L/k}(\mathbb{P}^1_L)$.
\item[$d=7$:] The possible Galois actions are trivial and $C_2$: the trivial action is realised by $\Bl_{1,1}\mathbb{P}^2$, while the $C_2$-action requires a degree-$2$ separable extension $L/k$ and is realised by $\Bl_2 \mathbb{P}^2$, where the blown-up degree-$2$ point is defined over $L$.
\end{itemize}
For the cases $d = 5,6$, see the proof of Theorem~\ref{thm:Minimal5and6}, which reduces them to $d=9$. \qedhere
\end{proof}

\subsubsection{Degree 4}

The next result is due to Kunyavski\u{\i}, Skorobogatov and Tsfasman and was later proved independently by Elsenhans and Jahnel \cite[Cor.~7.2]{EJ19b}.

\begin{thm}\cite[Cor.~6.10]{KST89} \label{thm:dP4IGP}
$\IGP_4(k)$ has a surjective solution for any infinite field $k$ with $\Char(k) \neq 2$.
\end{thm}

The proof goes via a related result on conic bundles (Theorem~\ref{thm:igpconic}).

\begin{rem}
A possible alternative strategy to prove Theorem~\ref{thm:dP4IGP} for all infinite fields is to adapt a description of degree-4 del Pezzo surfaces over fields of characteristic zero developed by Flynn \cite{Fly09}, later explained by Skorobogatov \cite{Sko10} and employed recently by Morgan and Skorobogatov \cite{MS24}, which says that any del Pezzo surface of degree $4$ is the Galois twist of one with a line, which is then rational.
\end{rem}

\subsubsection{Degree 3}

For degree $3$, we have the following results of Elsenhans and Jahnel.

\begin{thm}\cite[Thm.~0.1]{EJ15} \label{thm:dP3IGPQ}
$\IGP_3(\mathbb{Q})$ has a surjective solution.
\end{thm}

The proof of Theorem~\ref{thm:dP3IGPQ} is specific to $\mathbb{Q}$ since it involves a $\mathbb{Q}$-rational point search on certain moduli spaces.

\begin{thm} \cite[Thms.~6.1 \& 7.1]{EJ19b}
Let $k$ be an infinite field with $\ch{k}\neq 2$.
Denote the (maximal) subgroups of $W(E_6)$ of indices $36$ and $27$ by $U_{\text{ds}}$ and $U_{1}$, respectively. Then for any class $[H]$ lying in $U_{\text{ds}} \cap C(W(E_6),k)$ or $U_{1} \cap C(W(E_6),k)$, there exists a degree-3 del Pezzo surface $X/k$ with $[H_X] = [H]$.
\end{thm}

The subscripts ds and 1 indicate that surfaces with splitting group in these classes are exactly those containing a Galois-invariant double-six and rational line respectively.

\subsubsection{Degree 2}

In \cite{EJ19a,EJ19b}, Elsenhans and Jahnel described a construction of plane quartics with prescribed Galois action on the $28$ bitangents in certain cases. Using the relationship between plane quartics and del Pezzo surfaces of degree $2$ (see e.g.\ \cite[Thm.~III.3.5.2]{Kol96}), they established partial results for the inverse Galois problem for degree-$2$ del Pezzo surfaces, as summarised in the following theorem.

\begin{thm} \label{thm:CayleySteiner} \cite[Thm.~3.2]{EJ19a}, \cite[Cor.~5.3]{EJ19b},
Let $k$ be an infinite field with $\ch{k}\neq 2$. Consider the quotient map
\[\pi\colon W(E_7)\twoheadrightarrow W(E_7)/Z(W(E_7))\cong\Sp_6(\F_2),\]
and denote the (maximal) subgroups of $\Sp_6(\F_2)$ of indices $63$ and $36$ by $U_{63}$ and $U_{36}$, respectively. Then for any class $[H]$ in $\pi^{-1}(U_{63}) \cap C(W(E_7),k)$ or $\pi^{-1}(U_{36}) \cap C(W(E_7),k)$, there exists a degree-$2$ del Pezzo surface $X/k$ with $[H_X] = [H]$.
\end{thm}

We obtain the following important corollary.

\begin{cor} \label{cor:igpconic2}
For $k$ an infinite field with $\Char(k) \neq 2$, the inverse Galois problem for conic degree-2 del Pezzo surfaces admits a surjective solution.
\end{cor}

\begin{proof}
Theorem~\ref{thm:CayleySteiner} solves the inverse Galois problem for degree-2 del Pezzo surfaces whose branch curve has a Galois-invariant \emph{Steiner hexad} among its 28 bitangents; according to \cite[Prop.~2.3(d)]{EJ19b} and \cite[Lem.~2.11]{EJ19b}, any Galois action resulting in a $k$-rational conic bundle implies a Steiner hexad for the branch curve.
\end{proof}

Related work of Kulkarni and Vemulapalli \cite{KulVem24} leads to the following result.

\begin{thm} \cite[Thm.~3.1.1]{KulVem24}
Let $k$ be an infinite field with $\ch{k} \not\in \{2,3\}$. If $k$ admits degree-$d$ Galois extensions for $2 \leq d \leq 8$, then $\ICP_{2,1}(k)$ has a surjective solution.
\end{thm}

\subsubsection{Degree 1}

Kulkarni \cite{Kul21} provided a partial solution to the inverse Galois problem for del Pezzo surfaces of degree one over $\mathbb{Q}$.

\begin{thm} \label{thm:kulkarni} \cite[Thm.~1.0.1]{Kul21}
    For all $[H] \in C(W(D_8),\mathbb{Q}) \subset C(W(E_8),\mathbb{Q})$, there exists a degree-$1$ del Pezzo surface $X/\Q$ with $[H_X] =[H]$.
\end{thm}

In particular, we deduce the following result.

\begin{cor} \label{cor:ConicdP1IGPQ}
Over $\mathbb{Q}$, the inverse Galois problem for conic del Pezzo surfaces has a surjective solution. 
\end{cor}

\begin{proof}
For $X/k$ a del Pezzo surface of degree $1$ with a conic bundle, we have $H_X \leq W(D_7) \leq W(D_8) \leq W(E_8)$ by Lemma~\ref{lem:WDn}.
\end{proof}

\begin{lem} \label{lem:KulkarniExtensionZero}
Theorem~\ref{thm:kulkarni} extends to any infinite field $k$ with $\Char(k) = 0$.
\end{lem}

\begin{proof}
Following the arguments in the proof of \cite[Thm.~1.0.1]{Kul21}, it is easily seen that the properties of the ground field $k$ are invoked in precisely the following ways:
\begin{enumerate}
\item Construction 2.1.1 makes use of the fact that $\Char(k) \neq 2$ in order to characterise quadratic extensions as radical.
\item Lemma 2.2.1 uses that $k$ is infinite in order to ensure Zariski-density of rational points in affine space.
\item Lemma 3.1.3 is the point at which $k$ must be assumed infinite with $\Char(k) = 0$: Construction 3.1.1 gives a procedure for generating genus-4 curves over an affine base, and the locus of smooth curves is open, but one must show that it is non-empty. Example 3.1.2 does this by lifting a smooth curve from $\mathbb{F}_{101}$.
\item The remainder of the proof goes through for $k$ infinite with $\Char(k) \neq 2$. \qedhere
\end{enumerate}
\end{proof}

\begin{cor} \label{cor:KulkarniExtensionZero}
Over any infinite field $k$ with $\Char(k) = 0$, the inverse Galois problem for conic del Pezzo surfaces has a surjective solution. 
\end{cor}

\begin{rem}
Corollary~\ref{cor:KulkarniExtensionZero} can be extended to infinite fields $k$ with $\Char(k) \not\in \{2,3\}$ if one can show that \cite[Construction~3.1.1]{Kul21} yields at least one smooth curve.
\end{rem}

\subsubsection{Conic bundles}

We briefly discuss a related inverse problem solved by Kunyavski\u{\i}, Skorobogatov and Tsfasman mentioned above.

\begin{lem}\cite[Cor.~2.13]{KST89} \label{lem:WDn}
Let $\phi: Y \rightarrow C$ be a (standard) conic bundle with $n \leq 7$ singular fibres. Then the subgroup of $\Aut(\Pic(\overline{Y}))$ fixing $K_Y$ and the class of fibres of $\phi$ is isomorphic to $W(D_n)$.
\end{lem}

One may then define the \emph{splitting group} of $Y$ to be the image of $G_k$ in $W(D_n)$.

\begin{thm}\cite[Thm.~6.3]{KST89} \label{thm:igpconic}
The inverse Galois problem for conic bundles with $n \leq 7$ singular fibres over any infinite field $k$ with $\ch{k} \neq 2$ has a surjective solution: given any subgroup $G \leq W(D_n)$, there exists a conic bundle $Y \rightarrow C$ with splitting group conjugate to $G$ if and only if $G$ is a Galois group over $k$.
\end{thm}

\begin{rem}
It is not automatic that the conic bundles with prescribed Galois action obtained from Theorem~\ref{thm:igpconic} are del Pezzo surfaces. Indeed, the proof of \emph{loc.\ cit.}\ notes that, in degree $4$, the obtained surface may be an Iskovskikh surface, but that one can make an additional transformation to obtain a del Pezzo surface.
\end{rem}

Despite the previous remark, the inverse Galois problem for conic del Pezzo surfaces of degree $ d \geq $ over $k$ infinite with $\Char(k) \neq 2$ has a surjective solution (Corollary~\ref{cor:igpconic2}).

\subsubsection{Finite fields}

We now summarise previous work over finite fields.

\begin{rem}
Over finite fields, the inverse Galois problem for del Pezzo surfaces in the literature takes on a superficially different form: since $\Gal(\overline{\mathbb{F}}_q/\mathbb{F}_q)$ is generated by the Frobenius automorphism, $\rho_X(G_{\mathbb{F}_q})$ from Definition~\ref{def:galrep} is determined by the image of Frobenius, which is well-defined in the associated Weyl group up to conjugacy. Thus, the inverse Galois problem for del Pezzo surfaces over finite fields is traditionally nominally concerned with which conjugacy classes of \emph{elements} can be obtained as the image of Frobenius. However, since conjugacy classes of elements are in bijection with conjugacy classes of cyclic subgroups, this is equivalent to the version of the problem which we state.
\end{rem}

Work of Trepalin \cite{Tre17,Tre20} and Loughran--Trepalin \cite{LT20} gives the next theorem.

\begin{thm}
$\IGP_d(\mathbb{F}_q)$ admits a surjective solution for all finite fields $\mathbb{F}_q$ and all $d \geq 2$.
\end{thm}

We do not state specific results here, which necessitates a welter of notation for conjugacy classes; rather, we direct the interested reader to consult these papers or  the tables in Section~\ref{sec:TablesFinite} for finite fields. We will say more about the methods behind these results in Section~\ref{sec:finitefields}. A generalisation of these results to degree $1$, as stressed in \cite[\S7]{Tre20}, is a great challenge.

\subsubsection{Real closed fields}

Work of Russo \cite{Rus02} solves $\IGP_d(\mathbb{R})$.

\begin{thm} \label{thm:Russo}
For all degrees $d$, $\IGP_d(\mathbb{R})$ admits a surjective solution.
\end{thm}

Relevant tables are found in \cite{Rus02}: see,  in \emph{loc.\ cit.}, Proposition~1.2 for $d = 8$ relatively minimal, Corollary~2.4 for $d \geq 5$ except the previous case, Corollary~3.2 for $d = 4$, Corollary~3.3 for $d = 2$ and Corollary~5.3 for $d =1$. While Russo does not directly show existence of all listed surfaces, it follows from his descriptions of their twists (see the next remark). 

\begin{rem} \label{rem:RealClosedFields}
One may also consider \emph{real closed fields}, i.e.\ fields $k$ with $ 1\neq G_k$ finite; by the Artin--Schreier theorem \cite[Cor.~9.3]{Lang}, we then have $G_k \cong C_2$. Examples aside from $\mathbb{R}$ include the field of definable numbers and the Levi--Civita field. It is easily seen that Russo's result generalises into the real closed setting. Indeed, by combining Theorem~\ref{thm:dP5IGP} and Theorem~\ref{thm:dP4IGP}, it is immediate for $d \geq 4$. Further, since $G_k \cong C_2$ and a degree-3 del Pezzo surface has 27 geometric lines, at least one is $G_k$-stable and may be blown down, so $\IGP_3(k)$ reduces to $\IGP_4(k)$. For $d \leq 2$, Russo observes (see \cite[Cors.~4.3,~4.4,~5.3]{Rus02}) that any surface with no rational lines has a quadratic twist with a rational line, which ultimately gives the same reduction and relies only on the formal properties of such twists (see Section~\ref{sec:twists} for these twists in the finite field setting).
\end{rem}

\subsection{Algorithm 2} \label{sec:ICPalg}

We now detail our approach to the inverse curve problems. Specifically, we explain how to handle each individual possible count.

\begin{algorithm}[H] \label{alg:ICPalg}
\label{}
\caption{\bf Solving the inverse curve problems}

\KwIn{Integers $d \in \{1,\dots,9\}$ and $e \in \{1,2\}$ and an integer $n \in \mathcal{C}_{d,e}$.}
\KwOut{Either a surjective solution for the count $n$ or a list of ``ugly'' groups.}

\begin{enumerate}
\item Create an empty list $L$ and empty lists $G$, $B$ and $U$ of ``good'', ``bad'' and ``ugly'' groups.
\item By consulting the degree-$d$ table, add all types $T$ associated to the count $n$ to $L$.
\item For each blowup type $T$ in $L$, find representatives for all of its associated conjugacy classes:
    \begin{enumerate}
        \item If $T$ is associated to a relatively minimal surface of degree $5$ or higher, add these representatives to the list $G$ of good groups.
        \item Otherwise, add them to $B$.
    \end{enumerate}
\item If $B$ is empty, then declare that the subproblem admits a surjective solution for $n$. Otherwise, continue.
\item For each $H \in B$, compute all quotient groups of $H$:
    \begin{enumerate}
        \item If a quotient of $H$ lies in $G$, then remove $H$ from $B$.
        \item Otherwise, remove $H$ from $B$ and add $H$ to $U$.
    \end{enumerate}
\item If $U$ is empty, then declare that the subproblem admits a surjective solution for $n$. Otherwise, output $U$.
\end{enumerate}
\end{algorithm}

At the termination of Algorithm~\ref{alg:ICPalg}, we have either reached a surjective solution for the count $n$ or else $U$ contains groups corresponding to the count $n$ which are associated to blowing up a relatively minimal surface of degree at most $4$ and do not imply the existence of a blowup from degree at least $5$ achieving the same count.

We may then employ additional strategies. For example, if $\ch{k} \neq 2$ and $d \geq 2$ and each group in $U$ is ``potentially conic'' (in the sense that the associated surfaces either possess conic fibrations or do after blowing up), then we can attempt to apply Corollary~\ref{cor:igpconic2}; if $\ch{k} = 2$, then we can attempt an explicit Artin--Schreier version of their construction. We can similarly attempt to employ any of the results in Section~\ref{sec:igpresults} under additional hypotheses on the field $k$.

\section{Finite fields} \label{sec:finitefields}
In this section we explain the construction of more detailed versions of our tables for the finite field case and prove Theorem~\ref{thm:FiniteFieldLines} and Theorem~\ref{thm:FiniteFieldConics}.

\subsection{Frobenius}
Any finite extension $k/\mathbb{F}_q$ satisfies $k \cong \mathbb{F}_{q^n}$ for some $n \in \mathbb{Z}_{\geq 1}$, and we have $\Gal(\mathbb{F}_{q^n}/\mathbb{F}_q) = \langle \Frob_q \rangle \cong C_n$, where $\Frob_q$ is the \emph{Frobenius automorphism} which is defined by $\Frob_q(x) = x^q$. In fact, $\Frob_q$ is a topological generator of $G_{\mathbb{F}_q}$.

Given a del Pezzo surface $X/\mathbb{F}_q$, the \emph{trace of Frobenius} $a(X)$ on $X$ is the trace of $\Frob_q$ acting on $\Pic(\overline{X}) \cong \mathbb{Z}^{10 - d}$. One has
\[
\#X(\mathbb{F}_q) = q^2 + a(X)q + 1
\]
(see, e.g.\ \cite[Thm.~27.1]{Man67}).
It is also relatively easy to see that, for a relatively minimal conic bundle $X \rightarrow \mathbb{P}^1_{\mathbb{F}_q}$ with $r$ singular fibres, we have
\[
a(X) = 2 - r.
\]
Further, by an appeal to class field theory, it can be shown that such a conic bundle has singular fibres over an even number of closed points (see \cite[Lem.~2.6]{BFL19}).

One may capture more data via the \emph{Frame symbol} \cite[Eq.~3.2]{Frame51} (see also \cite[\S2]{Urabe96}).

\begin{defn}
Let $X/\mathbb{F}_q$ be a del Pezzo surface. The \emph{Frame symbol} associated to $X$ is a symbol $\prod_{n \geq 1} n^{a_n}$, where $\psi(t) = \prod_{n \geq 1} (t^n - 1)^{a_n} \in \mathbb{F}_q[t]$ is the characteristic polynomial of $\Frob_q$ acting on $\Pic(\overline{X})$.  The exponent $a_1$ of $1$ in the Frame symbol for $X$ is thus $a(X)$.
\end{defn}

\subsection{Cyclic conjugacy classes}

Since $G_{\mathbb{F}_q}$ acts on $\Pic(\overline{X})$ via a cyclic subgroup of the  group $G_d$, there are drastically fewer conjugacy classes to consider when it comes to the inverse Galois problem. The number $\mathfrak{C}_d$ of cyclic conjugacy classes (equivalently, conjugacy classes of elements) is given in the following table.

\begin{longtable}{|l|l|l|l|l|l|l|l|l|l|}\caption{Number of cyclic conjugacy classes in $G_d$}\\
\hline
d & 9 & 8 & 7 & 6 & 5 & 4 & 3 & 2 & 1 \\
\hline
$\mathfrak{C}_d$ & 1 & 2 & 2 & 6 & 7 & 18 & 25 & 60 & 112 \\
\hline
\end{longtable}

These relatively manageable numbers make it feasible to produce more detailed versions of our tables with a row for each conjugacy class and more attendant data.

We are aided by the long-standing interest in the conjugacy classes of Weyl groups, going back to work of Frame on $W(E_n)$ for $n \in \{6,7,8\}$ \cite{Frame51}, followed by Carter \cite{Carter72}. In \cite[Table~1]{Man86}, Manin, building on earlier work of Swinnerton-Dyer \cite{SD67}, produced a version for $W(E_6)$ with data pertinent to the arithmetic of the associated cubic surfaces such as the trace of Frobenius. Urabe \cite{Urabe96} produced analogous tables for $W(E_7)$ and $W(E_8)$. A revised version of Manin's table was constructed by Banwait, Fité and Loughran \cite[Table~7.1]{BFL19}.

By calculation of invariants, we identify each row in Urabe's tables with a blowup type and so produce our tables, which we use to prove Theorem~\ref{thm:FiniteFieldLines} and Theorem~\ref{thm:FiniteFieldConics}. Our proofs will occasionally make use of the \emph{orbit type}, meaning the number of Galois orbits of lines partitioned by configurations, as given in Urabe's tables.

\begin{rem}
During the construction of our tables, we noticed an error in Urabe's table for $W(E_7)$: the first row is not in fact of index $0$, i.e.\ it does not correspond to a relatively minimal del Pezzo surface of degree $2$. It instead corresponds to $\Bl_{2^3}S_{8,I}$, which has index $6$.
\end{rem}

\subsection{Geiser and Bertini twists} \label{sec:twists}
Let $X/\mathbb{F}_q$ be a del Pezzo surface of degree $d$. When $d = 2$, the Geiser involution on $X$ (swapping the two sheets of $X$ as a double cover of $\mathbb{P}^2$) gives rise to a quadratic twist known as its \emph{Geiser twist} (see \cite[\S4.1.2]{BFL19}). Similarly, when $d = 1$, we may associate to $X$ a \emph{Bertini twist}. Geiser twists were used by Trepalin \cite{Tre20} and Loughran and Trepalin \cite{LT20} in order to verify the existence of certain types of del Pezzo surfaces of degree $2$. These twists satisfy the following properties.

\begin{lem}
Let $X/\mathbb{F}_q$ be a del Pezzo surface of degree $d \leq 2$, and denote by $X'$ its Geiser/Bertini twist. 
\begin{enumerate}
    \item \cite[Lem.~4.4]{Tre20} If $d = 2$ and $X$ has an $\mathbb{F}_q$-rational line, then $X'$ does not.  
    \item \cite[Rem.~3.3]{Tre17} The Frame symbol of $X'$ is obtained from the Frame symbol of $X$ via the following sequence of operations on the Frame symbol $\prod_{n \geq 1} n^{a_n}$:
    \begin{enumerate}
        \item Replace $1^{a_1}$ by $1^{a_1 - 1}$.
        \item Replace each $n^{a_n}$ with $n$ odd by $n^{-a_n} (2n)^{a_n}$ and simplify to a new symbol $\prod_{n \geq 1} n^{b_n}$.
        \item Replace $1^{b_1}$ by $1^{b_1 + 1}$.
    \end{enumerate}
    \item We have $a(X') = 2 - a(X)$.
\end{enumerate}
\end{lem}

\begin{proof}
(1) is proved by Trepalin in \cite[Lem.~4.4]{Tre20}. For (2), note that the Geiser/Bertini involution acts by $-1$ on $K_{\overline{X}}^{\perp} \subset \Pic\left(\overline{X}\right)$, hence, for $\psi_X$ and $\psi_{X'}$ the characteristic polynomials of Frobenius for $X$ and $X'$ respectively, we have $\psi_{X'}(t) = \psi_X(-t)$, from which the result follows. Note that (3) follows immediately from (2) since $a_1 = a(X)$.
\end{proof}

\begin{rem} \label{rem:IGPfordP1s}
We note, as Trepalin does in \cite[\S7]{Tre20}, that the Bertini twist exhibits pathologies which the Geiser twist does not. In particular:
\begin{enumerate}
\item The $60$ classes of degree-$2$ del Pezzo surfaces split into $30$ pairs of twists, while some degree-$1$ del Pezzo surfaces are in the same class as their Bertini twist.
\item The twist of any degree-$2$ del Pezzo surface of Picard rank $1$ is of Picard rank at least $2$, while the twists of the $30$ classes of degree-$1$ del Pezzo surfaces of Picard rank $1$ include $5$ pairs of twists and $7$ classes with twist in the same class. It is the existence of these ``persistently rank-$1$'' del Pezzo surfaces of degree $1$ that puts $\IGP_1(\mathbb{F}_q)$ out of current reach.
\end{enumerate}
\end{rem}
\subsection{Proofs of Theorems~\ref{thm:FiniteFieldLines} and \ref{thm:FiniteFieldConics}} \label{sec:FiniteFieldsProofs}

In the following proves, we will make use of the Hasse--Weil bound for the number of $\mathbb{F}_q$-rational points on the ramification divisor of the anticanonical morphism of a degree-2 del Pezzo surface over $\mathbb{F}_q$: as such a curve is of genus $3$, it has at most $F(q)$ such points, where
\[
F(q):=\begin{cases}
q+1+6\sqrt{q} \quad \text{if $q$ is odd,} \\
2(q+1) \qquad \, \, \, \, \text{if $q$ is even.}
\end{cases}
\]
We also employ the following additional notation:
\begin{itemize}
    \item For $\sum_{i=1}^r a_i n_i = 8-d$ and $d \in \{1,2,4\}$, we denote by $S_{d,\II}^{(n_1^{a_1},\dots,n_r^{a_r})}$ a degree-$d$ relatively minimal del Pezzo surface which is a conic bundle with exactly $a_i$ singular fibres over closed points of degree $n_i$ for all $1 \leq i \leq r$.
    
    \item For $\sum_{i=1}^r a_i n_i = 5$, we denote by $S_{4,\II}^{(n_1^{a_1},\dots,n_r^{a_r})}$ a degree-4 minimal del Pezzo surface of Picard rank $1$ whose blowup at a rational point in general position (when one exists) is  a conic bundle with exactly $a_i$ singular fibres over closed points of degree $n_i$ for all $1 \leq i \leq r$. In particular, among the 18 types I--XVIII of minimal del Pezzo surfaces of degree $4$, the 6 which correspond to cyclic subgroups are given in the following table:
    \begin{center}
    \begin{tabular}{ | m{2cm} | m{1cm} | m{1cm} |m{1cm} |m{1cm} |m{1cm} |} 
      \hline
      \ \ \ \textbf{Type} & I & II & IV & V & XVIII \\ 
      \hline
      \textbf{del Pezzo} & $S_{4,I}^{(3,2)}$ & $S_{4,\II}^{(3,1)}$ & $S_{4,\II}^{(1^4)}$ & $S_{4,\II}^{(2^2)}$ & $S_{4,I}^{(4,1)}$ \\ 
      \hline
    \end{tabular}
    \end{center}
\end{itemize}

\begin{proof}[Proof of Theorem~\ref{thm:FiniteFieldLines}]

It suffices to consider the case $d = 1$. We consider each of the possible line counts from Theorem~\ref{thm:line}:
\[
0,2,4,6,8,12,14,20,24,26,30,40,60,72,126,240.
\]
\begin{enumerate}
    \item[0:] This count is possible for all $q$. Indeed, $\Bl_8 S_9$ always exists by \cite[\S5.2,~$a=1$]{BFL19} and has $(L_X,C_X) = (0,0)$. 
    
    \item[2:] This count is possible for all $q$. Indeed, $\Bl_{7,1} S_9$ always exists by \cite[\S5.2,~$a=2$]{BFL19} and has $(L_X,C_X) = (0,0)$. 
    
    \item[4:] This count is possible for all $q$. Consider type 44, which is $\Bl_1 S_{2,\II}^{(4,2)}$. It is clearly enough to show that $X = S_{2,\II}^{(4,2)}$ (a type 6 degree-2 del Pezzo surface) has a rational point not on any line and not on the ramification divisor. Note that $\#X(\mathbb{F}_q) = q^2 + 2q + 1$ and $(L_X,C_X) = (0,2)$. The orbit type of $X$ is $4^2 8^2 8^2 8^2$. Note that the $4^2$ corresponds to the singular fibre above the closed point of degree $2$ and its Geiser dual; in particular, it contributes no rational points, and the other orbits are too large to contribute rational points. Then there are no rational points on the lines of $X$. Then it suffices to note that $q^2 + 2q + 1 > F(q)$ for all $q$. 
    
    \item[6:] This count is possible for all $q$. It is enough to construct $\Bl_{5,2,1} S_9$ explicitly. First, choose a Galois orbit of $5$ points $P_1,\dots,P_5$ on a smooth plane conic $C/\mathbb{F}_q \subset \mathbb{P}^2_{\mathbb{F}_q}$, which always exists (take $xz = y^2$ when $2 \nmid q$ and $xz = y^2 + y$ when $q = 2$). Now choose a line $L/\mathbb{F}_q \subset \mathbb{P}^2_{\mathbb{F}_q}$ and a Galois orbit of two points $Q_1,Q_2$ on $L$ (this always exists). Finally, choose a rational point $R$ in $\mathbb{P}^2_{\mathbb{F}_q} \setminus L \cup C$. This is possible since $\#\mathbb{P}^2(\mathbb{F}_q) = q^2 + q + 1$, while $\#L(\mathbb{F}_q) = \#C(\mathbb{F}_q) = q + 1$ and $q^2 + q + 1 - 2(q+1) = q^2 - q - 1 > 0$ for all prime powers $q$. The blowup of $\mathbb{P}^2_{\mathbb{F}_q}$ furnishes the sought surface. Indeed, no three are collinear, no six are on a common conic, and the points are not all on a singular cubic, as can be seen by considering intersection numbers of the containing curves.  
    
    \item[8:] This count is possible exactly for $q \geq 3$. It can only arise from $\Bl_{2^3,1} S_{8,I}$ (type 98), which has $(L_X,C_X) = (8,24)$ and $\Bl_{6,1^2} S_9$ (type 107), which has $(L_X,C_X) = (8,12)$. The former is the blowup of a type-1 degree-2 del Pezzo surface, which only exists for $q \geq 3$ by \cite[Thm.~1.2]{LT20}. The latter only exists for $q \geq 3$ by \cite[\S5.2,~$a=3$]{BFL19}. Thus we conclude. 
    
    \item[12:] This count is possible for all $q$. It can only arise from $\Bl_{1^3} S_{4,I}$ (types 70, 73 and 74), $\Bl_{3,2^2,1}S_9$ (type 99) or $\Bl_{3^2,1^2} S_9$ (type 100). First, we show that it is realisable via $\Bl_{3^2,1^2} S_9$ for $q \geq 3$ (combined with the bound for the blowdown to a type-54 degree-2 del Pezzo surface, this shows that type 100 exists exactly when $q \geq 3$). Note that this surface is the blowup of a type-54 degree-2 del Pezzo surface $X = \Bl_{3^2,1} S_9$. We have $\#X(\mathbb{F}_q) = q^2 + 2q + 1$ and $(L_X,C_X) = (2,6)$. Further, the orbit type is $1^2 3^6 3^{12}$. By consideration of the origin of the $56$ lines in the plane, it is clear that the orbits $3^{12}$ represent pairwise disjoint lines, so there are at most $2(q+1) + 6 + F(q)$ points in the bad locus. Note that $q^2 - 7 > F(q)$ for $q \geq 7$. On the other hand, it can be checked computationally that it occurs for $q \in \{3,4,5\}$. Computer search reveals that the following degree-1 del Pezzo surface over $\mathbb{F}_2$ has exactly $12$ lines:
    \[
    w^2 + w(x^3 + y^3 + xy^2) = x^4z + xyz^2 + z^3 + y^2 z^2.
    \]
    Then we conclude. 

    \item[14:] This count is possible exactly when $q \geq 3$. It can only arise from $\Bl_{4,2,1^2} S_9$ (type 101). We first show that this exists for $q \geq 7$. It is equivalent to show the same for its Bertini twist, $\Bl_{2,1} S_{4,I}^{(2,1^3)}$. In turn, it is equivalent to show that a type-34 degree-2 del Pezzo surface $X = \Bl_{2}S_{4,I}^{(2,1^3)}$ has a rational point in general position for $q \geq 7$. 
    Such $X$ exists for $q \geq 3$. Note that $\#X(\mathbb{F}_q) = q^2 -2q + 1$ and $(L_X,C_X) = (0,0)$. Further, the orbit type is $2^2 2^6 4^{10}$. Since $C_X = 0$, the conjugate pairs of lines contain no rational points. Then there are at most $F(q)$ points in the bad locus. Note that $q^2 - 2q + 1 > F(q)$ for $q \geq 7$. Computation reveals that this also exists for $q \in \{3,4,5\}$, but not $q = 2$. 
    
    \item[20:] This line count is possible exactly when $q \geq 3$. The only way it can arise is as $\Bl_{5,1^3} S_9$ (type 102); the existence bound follows from \cite[\S5.2,~$a=4$]{BFL19}. 
    
    \item[24:] This line count is possible exactly when $q \geq 4$. It can only arise from $\Bl_{1^3} S_{4,\II}$. We first show that this exists for $q \geq 8$. It suffices to show that the type-27 degree-2 del Pezzo surface $X = \Bl_{1^2}S_{4,\II}^{(2^2)}$ has a rational point in general position for $q \geq 8$. Such $X$ only exists for $q \geq 3$. We have $\#X(\mathbb{F}_q) = q^2 + 4q + 1$ and $(L_X,C_X) = (8,6)$.  The orbit type is $1^8 4^{12}$. The $8$ rational lines contain at most $8q$ rational points: let us now show this. he two exceptional divisors contain exactly $2q + 2$ rational points. There are two fibrations on $S_{4,\II}$: denote the fibre of the $i$th fibration through a point $R$ by $C_i(R)$, and let $P$, $Q$ be the two blowup points. The (strict transforms of) the curves $C_1(P), C_2(Q)$ form a Geiser-dual pair and contribute at most $2q$ more rational points; the same is true for $C_2(P)$ and $C_1(Q)$. The final two rational lines come from $T_P$, an anticanonical curve with a double point at $P$ through $Q$, and $T_Q$, the same but with $P$ and $Q$ playing the opposite roles. Each of these contributes at most $2q$ more points (one can slightly improve this bound to $8q - 2$, but it will not change the outcome). The orbits of size $4$ all have the configuration of pairs of singular fibres from the conic fibration, hence they contain no rational points. Then there are at most $8q + F(q)$ points in the bad locus. Note that $q^2 - 4q + 1 > F(q)$ for $q \geq 8$. Now we investigate $q \leq 8$. Since $\Bl_{1^3}S_{4,\II}$ only exists for $q \geq 3$, we know that a line count of 24 is not possible over $\mathbb{F}_2$. Moreover, the only type of $\Bl_{1^3}S_{4,\II}$ that could possibly exist over $\mathbb{F}_3$ is type 68, which has trace 5. It is shown in \cite[Thm.~3.1.3]{Li} that there is just one degree-1 del Pezzo surface with trace 5 over $\mathbb{F}_3$:
    \[
    w^2 = z^3 + (2x^4 + x^2y^2 + 2y^4)z + (x^6 + 2x^4y^2 + y^6).
    \]
    Computer search shows that this surface has no rational lines, hence there is no degree-1 del Pezzo surface with 24 lines over $\mathbb{F}_3$. It remains to search over $\mathbb{F}_4$, $\mathbb{F}_5$ and $\mathbb{F}_7$.

Over $\mathbb{F}_4$, the following surface has exactly $24$ lines:
\[
w^2 + wy^3 = x^5y + xy^5 + z^3.
\]  
Over $\mathbb{F}_5$, the following surface has $24$ lines:
\[
x^6 + 4x^2y^4 + z^3 + 3zy^4 + w^2 + wy^3 + 4y^6.
\]
Over $\mathbb{F}_7$, the following surface has $24$ lines:
\[
x^5y + 5x^3y^3 + 2xy^5 + z^3 + w^2 + wy^3 = 0.
\]

    \item[26:] This count can only arise from $\Bl_{2^3,1^2} S_9$ (type 95). We claim that it exists exactly when $q \geq 7$. First, we show its existence for $q \geq 7$. It is enough to show existence of the Bertini twist, $\Bl_{2,1} S_{4,\II}^{(1^4)}$ (type 69), for $q \geq 7$. In turn, it is enough to determine when the type-28 degree-2 del Pezzo surface $X = \Bl_{2} S_{4,\II}^{(1^4)}$ has a rational point in general position. Such $X$ only exists for $q \geq 5$. We have $\#X(\mathbb{F}_q) = q^2 -2q + 1$ and $(L_X,C_X) = (0,4)$. The orbit type is $2^4 2^8 2^{16}$. The singular fibres of a chosen conic bundle structure lie over four rational points and a double point. Then we see that the $2^8$ corresponds to $8$ intersecting pairs of conjugate lines. The remaining size-2 orbits consist either of disjoint lines or Geiser duals; in either case, we get no more rational points off the ramification divisor, so there are at most $8 + F(q)$ points in the bad locus. Then note that $q^2 - 2q - 7 > F(q)$ for $q \geq 7$. Computational search reveals that it does not exist for $q \leq 5$. 
    
    \item[30:] This count is possible exactly when $q \geq 4$. It can only arise from $\Bl_{3,2,1^3} S_9$ (type 96). We first show that this exists for $q \geq 7$. It is enough to verify existence of its Bertini twist $\Bl_1 S_{2,I}^5$ (type 48) for $q \geq 7$. In turn, enough to show that the type-12 degree-2 del Pezzo surface $X = S_{2,I}^5$ has a rational point in general position for $q \geq 7$. This surface only exists for $q \geq 3$. We have $\#X(\mathbb{F}_q) = q^2 - 3q + 1$ and $(L_X,C_X) = (0,0)$. The orbit type is $2^{10} 6^6$. Since there are no rational lines and no conjugate pairs of lines which meet away from the ramification curve, hence there are at most $F(q)$ rational points in the bad locus. We see that $q^2 - 3q+ 1 > F(q)$ for $q \geq 7$. A computational search reveals existence for $q = 4$ and $q = 5$ and non-existence for $q = 2$ and $q = 3$. 
    
    \item[40:] This count is only possible for $q \geq 5$. The only way it can arise is from $\Bl_{4,1^4} S_9$ (type 97), which only exists for $q \geq 5$ by \cite[\S5.2,~$a=5$]{BFL19}. 
    
    \item[60:] This count is possible exactly when $q \geq 7$. It can only arise from $\Bl_{2^2,1^4} S_9$ (type 93). First we show that this exists when $q \geq 8$. It is enough to show the same for the Bertini twist $\Bl_1 S^{(1^6)}_{2,\II}$ (type 38). In turn, it is enough to show that the type-2 degree-2 del Pezzo surface $X = S^{(1^6)}_{2,\II}$ has a rational point in general position for $q \geq 8$. This exists only for $q \geq 5$. We have $\#X(\mathbb{F}_q) = q^2 - 4q + 1$ and $(L_X,C_X) = (0,2)$. The orbit type is $2^{12} 2^{16}$. It is clear that the $2^12$ corresponds to the $6$ singular fibres of the conic fibration and its dual; these contribute at most $12$ rational points to the bad locus. Since the other conjugate pairs are not of the same configuration, they cannot meet in a point away from the ramification curve. Then there are at most $12 + F(q)$ points in the bad locus. Then we note that $q^2 - 4q - 11 > F(q)$ for $q \geq 8$. Computer search reveals existence when $q = 7$ and non-existence when $q \leq 5$. 
    
    \item[72:] This count is possible exactly when $q \geq 7$: it arises only from $\Bl_{3,1^5}S_9$ (type 94), which exists exactly when $q \geq 7$ by \cite[\S5.2,~$a=6$]{BFL19}.
    
    \item[126:] This count is possible exactly when $q \geq 11$: it arises only from $\Bl_{2,1^6}S_9$ (type 92), which exists exactly when $q \geq 11$ by \cite[\S5.2,~$a=7$]{BFL19}.
    
    \item[240:] This count is possible exactly when $q = 16$ or $q \geq 19$: it arises only from $\Bl_{1^8}S_9$, which exists exactly in this range by \cite[\S5.2,~$a=9$]{BFL19}. \qedhere
\end{enumerate}

\end{proof}

\begin{proof}[Proof of Theorem~\ref{thm:FiniteFieldConics}]
It suffices to consider the case $d=1$. We consider each of the possible conic counts from Theorem~\ref{thm:conic}:
\[
0,2,4,6,12,18,24,30,36,72,90,252,270,756,2160.
\]
\begin{enumerate}
    \item[0:] This count is possible for all $q$. Indeed, $\Bl_8 S_9$ always exists by \cite[\S5,2,~$a=1$]{BFL19} and has $(L_X,C_X) = (0,0)$. 
    
    \item[2:] This conic count is possible for all $q$. First we show that it is possible for $q \geq 3$. Note that $\Bl_{2,1}S_{4,I}^{(2,1^3)}$ (type 76) has conic count 2, so it suffices to show that its Bertini twist $\Bl_{4,2,1^2} S_9$ (type 101) exists for $q \geq 3$. Consider its blowdown, the type-55 degree-2 del Pezzo surface $X=\Bl_{4,2,1}S_9$, which exists for $q \geq 3$. We have $\#X(\mathbb{F}_q) = q^2 + 2q + 1$ and $(L_X,C_X) = (4,4)$. Here, we are blowing up the plane in a rational point $P$, a pair of conjugate points $Q_1,Q_2$ and a $4$-tuple of conjugate points $R_1,\dots,R_4$. The $56$ geometric lines on $X$ correspond to the 7 exceptional divisors, the strict transforms of the $21$ lines joining these points and the Geiser duals. It is easily seen that, among the 21 lines, there is only one orbit of lines which all contain a rational point: namely, letting $R_{i+1}$ be the image of $R_i$ under Frobenius, the lines $L_{R_1R_3}$ and $L_{R_2R_4}$. Altogether, we deduce that there are at most $4(q+1) + 2 + F(q)$ points in the bad locus on $X$. Since $q^2 -2q - 1 > F(q)$ for $q \geq 7$, we verify existence of $X$ for $q \geq 7$. Computational search reveals that $\Bl_{4,2,1^2} S_9$ also exists for $q \in \{3,4,5\}$. It remains to show that it is possible over $\mathbb{F}_2$. Computer search reveals that the following surface is of type 5 over $\mathbb{F}_2$:
    \[
    x^5y + x^3yz + xyz^2 + xy^2w + z^3 + w^2 + y^6 = 0.
    \]
    
    \item[4:] This count is possible for all $q$. Indeed, $\Bl_{7,1} S_9$ always exists by \cite[\S5,2,~$a=2$]{BFL19} and has $(L_X,C_X) = (2,4)$. 
    
    \item[6:] This count is possible for all $q$. Let us first show that it is possible for $q \geq 7$. It is enough to show that $\Bl_{1^3}S_{4,I}$ (type 74) exists for $q \geq 7$. In turn, it is enough to show that the type-33 degree-2 del Pezzo surface $X = \Bl_{1^2}S_{4,I}^{(3,2)}$ has a rational point in general position for $q \geq 7$. Note that $\#X(\mathbb{F}_q) = q^2 + 3q + 1$ and $(L_X,C_X) = (4,2)$. The orbit type is $1^4 4^2 4^2 6^2 12^2$. There are at most $4(q+1)$ points on the $4$ rational lines. Further, at most one $4^2$ contributes points, hence there are at most $4q + 6$ points on lines, hence at most $4q + 6 + F(q)$ points in the bad locus. Note that $q^2 - q - 5 > F(q)$ for $q \geq 7$. It now suffices to note, via computer search, that $\Bl_{5,2,1}S_9$, which has conic count $6$, exists for $2 \leq q \leq 7$. 
    
    \item[12:] This count is possible exactly when $q \geq 3$. Note that $\Bl_{6,1^2}S_9$ (type 107) has $(L_X,C_X) = (8,12)$ and exists for $q \geq 3$ by \cite[\S5,2,~$a=3$]{BFL19}. It remains to show that conic count 12 is impossible over $\mathbb{F}_2$. Considering restrictions from blowing up degree-2 del Pezzo surfaces, we see that it can only arise from type 71, which has 6 lines and trace 2. Further, the twist contains 2 lines over $\mathbb{F}_2$ and $240$ lines over $\mathbb{F}_{64}$, and these properties distinguish type 71. Computational search reveals that these properties are not simultaneously satisfied, hence we conclude. 
    
    \item[18:] This count is possible exactly when $q \geq 3$. It can only arise from $\Bl_{3,2^2,1}S_9$ (type 99) which has Bertini twist $\Bl_1 S_{2,\II}^{(3,1^3)}$ (type 40). First we show that a type-4 degree-2 del Pezzo surface $X = S_{2,\II}^{(3,1^3)}$ has a rational point in general position for $q \geq 7$. Such $X$ only exists for $q \geq 3$. We have $\#X(\mathbb{F}_q) = q^2-q+1$ and $(L_X,C_X) = (0,2)$. Further, the orbit type is $2^4 2^6 6^2 6^4$. The $2^6$ corresponds to singular points on the three singular fibres over rational points and their Geiser duals. Then there are at most $6 + F(q)$ points in the bad locus. Note that $q^2 - q - 5 > F(q)$ for $q \geq 7$. Computer search reveals that $\Bl_{3,2^2,1}S_9$ exists for $q \in \{3,4,5\}$ and does not exist when $q = 2$. 
    
    \item[24:] This count is possible exactly when $q \geq 4$. First we show that it is possible for $q \geq 4$. It is enough to show that $\Bl_{2^4}S_9$ (type 98) exists in this range (computer search reveals that it does not exist for $q \leq 3$). Since $\Bl_{2^4}S_9 \cong \Bl_{2^3,1}S_{8,I}$, it is enough to show that the type-1 degree-2 del Pezzo surface $X=\Bl_{2^3}S_{8,I}$ has a rational point in general position when $q \geq 4$. We have $\#X(\mathbb{F}_q) = q^2 + 1$ and $(L_X,C_X) = (0,6)$. Further, the orbit type is $2^4 2^{24}$. Since there are more than 4 orbits of disjoint pairs of lines, it is clear that the lines contain at most $4$ rational points. Note that $q^2 - 3 > F(q)$ when $q \geq 4$. On the other hand, computer search shows that $\Bl_{2^4}S_9$ does not exist for $q \leq 3$. Now we show non-existence of the other types (types 66, 67 and 68, all $\Bl_{1^3}S_{4,\II}$) with 24 conics when $q \leq 3$. This is dealt with in the argument for line count 24. 
    
    \item[30:] This count is possible exactly when $q \geq 3$. It can only arise from $\Bl_{4,2,1^2} S_9$ (type 101) and $\Bl_{5,1^3} S_9$ (type 102). The former is the blowup of a type-55 degree-2 del Pezzo surface, which only exists for $q \geq 3$, while the latter only exists for $q \geq 3$ by \cite[\S5,2,~$a=4$]{BFL19}. 
    
    \item[36:] This count is possible exactly when $q \geq 3$. It can only arise from $\Bl_{3^2,1^2}S_9$ (type 100), which has twist type 16, which is the blowup of a type-54 degree-2 del Pezzo surface $X = \Bl_{3^2,1}S_9$, which only exists for $q \geq 3$. We have $\#X(\mathbb{F}_q) = q^2 + 2q + 1$ and $(L_X,C_X) = (2,6)$. The orbit type is $1^2 3^6 3^{12}$. There are at most $2(q+1)$ points on the rational lines. It is not too hard to see that the singular fibres of one of the conic fibrations lie above two cubic points, accounting for $3^2$ in each case, ultimately accounting for the $3^{12}$. It is also easy to see that the remaining $3^6$ corresponds to conjugate triples of disjoint lines, hence there are at most $2q+2 + F(q)$ points in the bad locus. Note that $q^2 - 1 > F(q)$ for $q \geq 4$. A computer search reveals existence for $q = 3$. 
    
    \item[72:] This count is possible exactly when $q \geq 7$. It can only arise from type 95. We showed in the argument for $L_X = 26$ that this exists exactly when $q \geq 7$. 
    
    \item[90:] This count is possibly exactly when $q \geq 4$. It can only arise from $\Bl_{3,2,1^3}S_9$ (type 96) and $\Bl_{4,1^4}S_9$ (type 97). We showed above in the argument for $L_X = 30$ that the former exists exactly when $q \geq 4$, and the latter exists exactly when $q \geq 5$ by \cite[\S5,2,~$a=5$]{BFL19}. 
    
    \item[252:] This count is possible exactly when $q \geq 7$. It can only arise from $\Bl_{2^2,1^4}S_9$ (type 93), whose Bertini twist has type 38, which is the blowup of a type-2 degree-2 del Pezzo surface $X = S_{2,\II}^{(1^6)}$, which only exists for $q \geq 5$. We have $\#X(\mathbb{F}_q) = q^2 - 4q + 1$ and $(L_X,C_X) = (0,2)$. The orbit type is $2^{12} 2^{16}$. The $2^{12}$ clearly arises from the singular fibres of the conic fibration and its dual, and so the bad locus contains at most $12 + F(q)$ points. Note that $q^2 -4q - 11 > F(q)$ for $q \geq 8$. A computational search reveals existence for $q = 7$ and non-existence for $q \leq 5$. 
    
    \item[270:] This count is possible exactly when $q \geq 7$. It can only arise from $\Bl_{3,1^5}S_9$ (type 94), which exists exactly when $q \geq 7$ by \cite[\S5,2,~$a=6$]{BFL19}. 
    
    \item[756:] This count is possible exactly when $q \geq 11$. It can only arise from $\Bl_{2,1^6}S_9$ (type 92), which exists exactly when $q \geq 11$ by \cite[\S5,2,~$a=7$]{BFL19}. 
    
    \item[2160:] This count is possible exactly when $q \geq 16$ or $q \geq 19$. It can only arise from $\Bl_{1^8} S_9$ (type 91), which exists exactly when $q = 16$ or $q \geq 19$ by \cite[\S5,2,~$a=9$]{BFL19}. \qedhere
\end{enumerate}
\end{proof}

\subsection{Existence of relatively minimal del Pezzo surfaces over finite fields}

From the table for degree-1 del Pezzo surfaces and the calculations above, we deduce the following result.

\begin{lem} \label{lem:RelMinsOverFq}
Let $\mathbb{F}_q$ be any finite field.
\begin{enumerate}
    \item For all $d \in \{1,\dots,9\} \setminus \{7\}$, a surface of blowup type $S_{d,I}$ exists over $\mathbb{F}_q$.
    \item For $d \in \{1,2,4\}$, a surface of blowup type $S_{d,\II}$ exists over $\mathbb{F}_q$.
\end{enumerate} 
\end{lem}

\begin{proof}
For $d \geq 2$, both results follow from \cite{LT20}. Further, since $S_{1,I}^{13}$ (a type-20 degree-1 del Pezzo surface) is the twist of $\Bl_{7,1}S_9$ (type 111) and the latter exists for all $q$ by \cite[\S5.2,~$a=2$]{BFL19}, so does the former, hence $S_{1,I}$ exists for all $q$. Then it only remains to verify the existence of $S_{1,\II}$. In the $C_X = 2$ case of the proof of Theorem~\ref{thm:FiniteFieldConics}, we saw that $S_{1,\II}^{(6,1)}$ exists over $\mathbb{F}_2$; since the twist of $S_{1,\II}^{(3,2,1^2)}$ (type 2), namely $\Bl_{4,3,1}S_9$ (type 105) exists for $q \geq 3$, we deduce that $S_{1,\II}$ always exists.
\end{proof}

\section{Infinite fields}\label{sec:infinitefields}

In this section we use our algorithms to prove Theorems~\ref{thm:icp} and \ref{thm:ModestICP}.

We will see that this comes down to constructing conic bundles with three prescribed Galois groups in characteristic $2$. Accordingly, we begin with the following well-known result from Artin--Schreier theory (see \cite[\href{https://stacks.math.columbia.edu/tag/09I7}{Tag 09I7}]{stacks-project} for more).

\begin{lem} \label{lem:quadchar2}
Every separable extension of degree 2 of a field $k$ of characteristic $2$ is of the form $k(\phi(\alpha))$ for some $\alpha \in k^\times$, where $\phi(\alpha)$ is a root of the polynomial $x^2 + x + \alpha$. 
\end{lem}

\begin{proof}
Any separable quadratic extension of $k$ arises from adjoining a root of some irreducible quadratic polynomial $ax^2 + bx + c$ where $a \in k^\times$ and, for separability, $b \in k^\times$. Without loss of generality, we may assume $a = 1$; dividing by $b^2$ and considering the resulting expression as a polynomial in $x/b$, we obtain the result.
\end{proof}

\subsection{Proof of Theorem~\ref{thm:icp}}

As a preliminary for the proof of Theorem~\ref{thm:icp}, we will solve the inverse Galois problem for degree-4 del Pezzo surfaces for two specific groups over infinite fields of characteristic 2, beginning with the following proposition.

\begin{prop} \label{prop:char2IGP}
Let $G \leq W(D_5)$ be isomorphic to $C_2^4 \rtimes G_5$, where $G_5 \in \{C_5,A_5\}$, and suppose that $k$ is an infinite field of characteristic $2$ with a Galois extension $L/k$ such that $\Gal(L/k) \cong G$. Then there exists a quintic polynomial $f(x) \in k[x]$ and a Galois extension $L'/k$ such that $\Gal(L'/k) \cong G$ and $L' = k(\phi(\alpha_1),\dots,\phi(\alpha_4))$, where the $\alpha_i$ are four distinct roots of $f(x)$ and $\phi(\alpha)$ is a root of the polynomial $x^2 + x + \alpha$.
\end{prop}

\begin{proof}
Let $L/k$ be a Galois extension with $\Gal(L/k) \cong C_2^4 \rtimes G_5$. Since $G$ has a normal subgroup $H = C_2^4 \rtimes \{1\} \cong C_2^4$ with quotient isomorphic to $G_5$, there exists a subextension $L/M/k$ with $\Gal(L/M) \cong C_2^4$ and $\Gal(M/k) \cong G_5$.

Further, note that $H$ is contained in a subgroup $H'$ of the form $C_2^4 \rtimes H_5$, where $H_5 \leq G_5$ is index-5 (so $H_5 = \{1\}$ for $G_5 = C_5$ and $H_5 = A_4$ for $G_5 = A_5$), and $H'$ corresponds to an extension $k(\alpha_1)/k$, where $\alpha_1$ is one of the five distinct roots $\alpha_1,\dots,\alpha_5$ of a quintic polynomial $f(x) \in k[x]$ (the distinctness of the $\alpha_i$ follows from separability of $k(\alpha_1)/k$ as a subextension of a separable extension).

Note that $L' := k(\phi(\alpha_1),\dots,\phi(\alpha_5))$ is the splitting field of $f(x^2 + x)$, hence it is Galois over $k$. We will show that, in either case, $\Gal(L'/k) \cong G$.

First, note that $H' = C_2^4 \rtimes H_5$ contains an index-2 subgroup $H''$ but $H_5$ does not: in the case $G_5 = C_5$, we may take $H'' = C_2^3 \rtimes \{1\}$, and in the case $G_5 = A_5$, we may take $H'' = C_2^3 \rtimes A_4$, and the absence of index-2 subgroups of $H_5$ itself is clear since $H_5$ is simple. Then $k(\alpha_1)/k$ permits a quadratic extension $k''/k(\alpha_1)$ which is a subextension of $L/k$ but not of $M/k$. By Lemma~\ref{lem:quadchar2}, we have $k'' = k(\alpha_1)\left(\phi\left(\sum_{i=0}^4 a_i \alpha_1^i\right)\right)$ for some $a_0,\dots,a_4 \in k$. Note that $k\left(\sum_{i=0}^4 a_i \alpha_1^i\right)$ is a subextension of a degree-5 extension, hence it is either trivial or of degree 5, but it cannot be trivial since $k''$ is a degree-10 extension of $k$, hence it is of degree 5. Then, without loss of generality, we have $k'' = k(\phi(\alpha_1))$. Then $L'$ is a subfield of $L$ (as the Galois closure of $k(\phi(\alpha_1)) \subset L$), and the Galois group of $L'$ is isomorphic to a normal subgroup of $G$ containing $H''$, but the only such subgroup is $G$ itself, as may be verified in each case, hence we are done.
\end{proof}

\begin{thm} \label{thm:ExplicitChar2IGP}
Let $k$ be an infinite field of characteristic $2$ and let $G \in \{C_2^4 \rtimes C_5,C_2^4 \rtimes A_5\}$. Then a $G$-extension of $k$ exists if and only if there exists a degree-4 del Pezzo surface with splitting group $G$.
\end{thm}

\begin{proof}
It suffices to prove sufficiency. By Proposition~\ref{prop:char2IGP}, we see that there exists a degree-5 polynomial $f(x) \in k[x]$ such that $f(x)$ has Galois group $C_5$ or $A_5$ and $f(x^2+x)$ has Galois group $G$. 

Consider the surface 
\[
Z_0=\{x^2 + xy + ty^2 + f(t)z^2 = 0\} \subset \mathbb{P}^2_{x,y,z} \times \mathbb{A}^1_t.
\]

One can easily check that $Z_0$ is smooth. Indeed, smoothness in the patches $x\neq 0$, $y\neq 0$ is trivial, and any singular point $([x:y:z],t)$ in the patch $z\neq 0$ would imply $f(t) = f'(t) = 0$, which contradicts the separability of $f$.

We now seek to patch this affine conic bundle with another to produce a conic bundle over $\mathbb{P}^1$. First, consider the assignments
\[
u = \frac{1}{t}, \ \ \  X = x, \ \ \  Y = \frac{y}{u}, \ \ \  Z = \frac{z}{u^3}.
\]
Substituting these into the equation for $Z_0$, we arrive at the surface
\[
V=\{X^2 + uXY + uY^2 + ug(u)Z^2 = 0 \}\subset \mathbb{P}^2_{X,Y,Z} \times \mathbb{A}^1_u,
\]
where $g(T):=T^5f(1/T)$. That is, if $f(t) = t^5 + c_4t^4 + \dots + c_1 t + c_0$, then $g(t) = c_0 t^5 + c_1 t^4 + \dots + c_4 t + 1$.

This surface patches with $Z_0$, but it is singular at the point $([0:1:1],0)$. We are thus motivated to transform the equation by moving this singular point to the origin in the affine patch $z \neq 0$, blowing up and taking the strict transform. This amounts to the following additional change of variables:
\[
X'= \frac{X}{u}, \ \ \  Y' = \frac{Y+Z}{u}, \ \ \  Z' = Z.
\]
Substituting these new variables in for the previous ones, we arrive at the surface
\[
Z_\infty=\{X'^2 + uX'Y' + X'Z' + uY'^2 + \left(c_4 + \dots + c_0 u^4\right)Z'^2 = 0\} \subset \mathbb{P}^2_{X',Y',Z'} \times \mathbb{A}^1_u.
\]

It is easily seen from the above that $Z_0 \cap \{t \neq 0\} \cong Z_\infty \cap \{u \neq 0\}$; we claim that $Z_\infty$ is, moreover, smooth. Again, we check in affine coordinate patches.

First consider the case $X' \neq 0$, and normalise so that $X' = 1$. Then we have
\[
\begin{aligned}
    F & = 1 + uY' + Z' + uY'^2 + \left(c_4 + \dots + c_0 u^4\right)Z'^2, \\
    J & = \left(u, 1, Y' + Y'^2 + (c_3 + c_1u^2)Z'^2\right).
\end{aligned}
\]
Since the second entry of $J$ is never zero, there are no singularities in this patch.

Next consider the case $Y' \neq 0$, and normalise so that $Y' = 1$. Then we have
\[
\begin{aligned}
    F & = X'^2 + uX' + X'Z' + u + \left(c_4 + \dots + c_0 u^4\right)Z'^2, \\
    J & = \left(u + Z', X', X' + 1 + (c_3 + c_1 u^2)Z'^2\right).
\end{aligned}
\]
Any singularity must have $X' = 0$, but then the first entry of $J$ gives $u = Z'$, and then the equation gives $u + \left(c_4 + \dots + c_0 u^4\right)u^2 = 0$ while the third entry of the Jacobian gives $1 + c_3u^2 + c_1u^4 = 0$. These imply $g(u) = g'(u) = 0$, but it is easily seen that $g$ is separable. Then there are no singularities in this patch.

Finally consider the case $Z' \neq 0$, and normalise so that $Z' = 1$. Then we have
\[
\begin{aligned}
    F & = X'^2 + uX'Y'  + X' + uY'^2 + \left(c_4 + \dots + c_0 u^4\right), \\
    J & = \left(uY' + 1, uX', X'Y' + Y'^2 + c_3 + c_1u^2\right).
\end{aligned}
\]
The second entry of $J$ gives $u = 0$ or $X' = 0$. If $u = 0$, then the first entry of $J$ gives $1 = 0$, a contradiction, so we must have $X' = 0$. Then the first entry of $J$ gives $Y' = \frac{1}{u}$, and the third entry of $J$ and the equation then give $g(u) = g'(u) = 0$, which is impossible by separability of $g$, hence there are no singularities in this patch. We deduce that $Z_\infty$ is smooth.

Lastly, note that the fibre of $Z_\infty$ over $u = 0$ is given by
\[
X'^2 + X'Z' + c_4 Z'^2 = 0.
\]
We claim that this fibre is split over $k$, i.e.\ that $\phi(c_4) \in k$. To see this, note that $\phi(\alpha_5) = \phi(c_4) + \sum_{i=1}^4 \phi(\alpha_i)$, but $\phi(\alpha_5) \in k(\phi(\alpha_1),\dots,\phi(\alpha_4))$, hence $k(\phi(c_4)) \subset L$, but $L$ contains no index-2 subgroups, hence $k(\phi(c_4)) = k$.

Denote by $Z$ the surface obtained by gluing $Z_0$ and $Z_\infty$ along $\{t \neq 0\}$ and $\{u \neq 0\}$ and contracting the fibre over $u = 0$: we see that $Z$ is a conic bundle over $\mathbb{P}^1$ with five singular fibres and splitting group $G$, and we are done. \qedhere

\end{proof}

Before we prove Theorem~\ref{thm:icp}, we need one more auxiliary result.

\begin{prop} \label{prop:ExplicitChar2ConicICP}
Let $k$ be an infinite field with $\Char(k) = 2$ admitting a separable extension of degree $2$. Then there exists a degree-4 del Pezzo surface $X/k$ of type $S_{4,\II}$ with splitting group $C_2$.
\end{prop}

\begin{proof}
Let $L = k(\phi(\alpha))$ be a separable extension of degree $2$ for some $\alpha \in k^\times$. Choose $\beta \in k^\times \setminus \{\alpha\}$, and consider
\[
Z_0:=\{x^2 + xy + ty^2 + (t^2+t+\alpha)(t^2 + t+ \alpha + \beta)z^2 = 0\} \subset \mathbb{P}^2_{x,y,z} \times \mathbb{A}^1_t.
\]
Set $f(t):=(t^2+t+\alpha)(t^2 + t+ \alpha + \beta)$ and $g(t):=t^4f(1/t)$. Analogously to the proof of Theorem~\ref{thm:ExplicitChar2IGP}, we consider first the change of variables
\[
u:=\frac{1}{t}, \quad X:=x, \quad Y:=\frac{y}{u}, \quad Z:=z,
\]
and upon finding singularities on the associated surface, further consider
\[
X':=\frac{Y}{u}, \quad Y':=\frac{X+Z}{u}, \quad Z':=Z,
\]
obtaining the surface
\[
Z_\infty:=\{Y'^2 + uX'Y' + X'Z' + h(u)Z'^2 = 0\} \subset \mathbb{P}^2_{X',Y',Z'} \times \mathbb{A}^1_u.
\]
It is easily seen that this surface is smooth and glues with $Z_0$. Further, note that the singular fibres of $Z_0$ as a conic bundle under projection to $\mathbb{A}^1_t$ are over the rational points $\alpha,\alpha+1,\beta,\beta+1$ and are non-split over $k$, each with splitting field $k(\alpha)$. Further, $Z_\infty$ contributes no new singular fibres, hence the gluing $Z$ is a relatively minimal conic bundle with splitting group $C_2$. As in the proof of Theorem~\ref{thm:ExplicitChar2IGP}, we deduce the result.
\end{proof}

\begin{cor}[Theorem~\ref{thm:icp}]
The inverse line and conic problems $\ICP_{4,e}(k)$ admit surjective solutions for all infinite fields $k$.
\end{cor}

\begin{proof}
We first consider $\ICP_{4,1}(k)$. It is easily seen from the tables in Section~\ref{sec:Tables} that a degree-4 del Pezzo surface with $n \in \{1,2,4,8,16\}$ lines exists as soon as one is permitted by Galois theory, since these counts correspond exclusively to blowups of del Pezzo surfaces of degree $5$ or greater, for which the inverse Galois problem is known. Then the only count left to deal with is $0$.

Assume that the Galois theory of $k$ permits the existence of a degree-4 del Pezzo surface with zero lines. Since the blowup $\Bl_{2^2}S_{8,I}$ is a degree-4 del Pezzo surface with $0$ lines, we may assume without loss of generality that $k$ admits no separable extensions of degree 2.

By explicitly analysing the list of possible splitting groups for a degree-4 del Pezzo surface with zero lines, we see that all except the two groups in Proposition~\ref{prop:char2IGP} contain subgroups of index 2 and so imply a separable quadratic extension. Then we reduce to assuming that we know only that $k$ admits a Galois extension with one of these two groups as its Galois group. The result now follows from Theorem~\ref{thm:ExplicitChar2IGP}.

It remains to consider $\ICP_{4,2}(k)$. Here, the only problematic conic counts are $0$ and $2$. For a conic count of $2$, we see from Proposition~\ref{prop:ExplicitChar2ConicICP} that a quadratic extension is sufficient, as is a quartic extension; all minimal subgroups of $W(D_4)$ admit normal subgroups of index 2 or 4, hence this count can be dealt with. For the count $0$, we see that a separable extension of degree $4$ or $5$ is sufficient; all minimal subgroups of $W(D_5)$ contain a subgroup of index either $4$ or $5$, so we are done.
\end{proof}

\subsubsection{Cubic surfaces}

Applying Algorithm~\ref{alg:ICPalg} for $d = 3$ (recalling $\ICP_{3,1}(k)$ and $\ICP_{3,2}(k)$ are equivalent), we obtain a surjective solution for each positive count. For the count $0$, our algorithm terminates in $U = \{G\}$ for $G$ the unique index-2 subgroup in $W(E_6)$, and $B$ contains $C_3$, $\Dih_6$ and $\Sym_3$. Since $G$ is simple, our reduction strategy cannot be employed; further, none of the results in Section~\ref{sec:igpresults} are able to handle this case. We thus obtain the following result.

\begin{prop}
For $k$ an infinite field with a separable cubic extension, $\ICP_{3,1}(k)$ and $\ICP_{3,2}(k)$ admit positive solutions.
\end{prop}

\begin{proof}
As discussed above, it suffices to consider the count $0$; a cubic extension permits the existence of $\Bl_{3^2}S_9$, which has no lines.
\end{proof}

\subsection{Proof of Theorem~\ref{thm:ModestICP}}

\begin{proof}[Proof of Theorem~\ref{thm:ModestICP}]
Let $k$ be an infinite field admitting separable extensions of degrees $1$ to $8$. Proposition~\ref{prop:general position dP} implies any blowup type coming from a relatively minimal del Pezzo surface of degree at least $5$ exists over $k$. Inspection of the tables in Section~\ref{sec:Tables} reveals that the only line or conic count unobtainable from such blowups is $24$ in degree $1$, both as a line count and a conic count, which can only be obtained from the type $\Bl_{1^3}S_{4,\II}$. Then one deduces $\ICP_d(k)$ for all $d$ immediately from Theorem~\ref{thm:dP4IGP} for $\Char(k) \neq 2$ and Proposition~\ref{prop:ExplicitChar2ConicICP} otherwise.

Now assume that $\Char(k) \neq 2$ and $k$ admits a Galois extension of degree $3$. Inspection of the tables in Section~\ref{sec:Tables} reveals that the only pairs of line and conic counts $(L_X,C_X)$ unobtainable from these blowup types are as follows:
\begin{enumerate}
    \item $d \geq 3$: None.
    \item $d = 2$: $(0,2)$ (from $S_{2,\II}$), $(0,4)$ (from $\Bl_2 S_{4,\II}$) and $(8,6)$ (from $\Bl_{1^2}S_{4,\II}$).
    \item $d = 1$: $(2,2)$ (from $\Bl_{2,1}S_{4,I}$), $(12,6)$ (from $\Bl_{1^3}S_{4,I}$), $(0,6)$ (from $\Bl_3 S_{4,\II}$), $(6,12)$ (from $\Bl_{2,1}S_{4,\II}$), $(24,24)$ (from $\Bl_{1^3}S_{4,\II}$) and $(4,4)$ (from $\Bl_1 S_{2,\II}$).
\end{enumerate}
We immediately deduce the result for $d \geq 3$. In the case $d = 2$, note that all of the blowup types have rational conic fibrations, among which some have splitting group $C_2$ in each degree, hence the result follows from Corollary~\ref{cor:igpconic2}. Finally we consider $d = 1$. Of the six listed blowup types, note that all except $\Bl_3 S_{4,\II}$ arise as the blowup of a degree-2 del Pezzo surface $Y$ with a rational conic fibration. Further, it is easily seen that these surfaces can be assumed unirational by Theorem~\ref{thm:dP4IGP}, hence $Y$ has a rational point in general position, hence the blowup type exists. It only remains to deal with $\Bl_3 S_{4,\II}$. Note that here, one may construct such a surface with splitting group $C_6$ (see type 75 in the degree-1 table for finite fields). Note then that one may Bertini twist to obtain a surface of the form $\Bl_{1^3}S_{4,\II}$, hence the blowup type exists.
\end{proof}

This result applies in particular to non-archimedean local fields.

\begin{defn}
A field $k$ is a \emph{local field} if it is isomorphic to one of:
\begin{enumerate}
    \item $\mathbb{R}$, $\mathbb{C}$ ($k$ is \emph{archimedean}).
    \item A finite extension of $\mathbb{Q}_p$, $\mathbb{F}_q((u))$ for some prime $p$, prime power $q$ ($k$ is \emph{non-archimedean}).
\end{enumerate}
\end{defn}

Since a non-archimedean local field has a cyclic (unramified) Galois extension of every degree, we obtain Corollary~\ref{cor:LocalICP}.

\section{Hilbertian fields} \label{sec:hilb}
In this section, we give background on Hilbertian fields and prove Theorem~\ref{thm:hilb}.

\begin{defn} \label{def:hilbertian}
A field $k$ is \emph{Hilbertian} if $\Char(k) = 0$ and, for any finite collection of finite morphisms $\phi_i: V_i \rightarrow \mathbb{P}^1$, $i =1 ,\dots,n$, the set $\mathbb{P}^1(k) \setminus \bigcup_{i=1}^n \phi_i(V_i(k))$ is Zariski-dense.
\end{defn}

\begin{ex}
It essentially follows from Hilbert's Irreducibility Theorem that any number field $K$ is Hilbertian. Further, fields finitely generated over their prime field and function fields are Hilbertian; see \cite[\S13]{FriedJarden} for more on this. On the other hand, algebraically closed fields, finite fields and local fields are not Hilbertian.
\end{ex}

\begin{lem}
Let $k$ be a Hilbertian field and let $d \in \{1,\dots,9\}$.
\begin{enumerate}
    \item A relatively minimal degree-$d$ del Pezzo surface with splitting group $G_d$ exists over $k$.
    \item If $d \in \{2,4\}$, relatively minimal conic and non-conic degree-$d$ del Pezzo surfaces with maximal Galois action exist over $k$.
\end{enumerate}
\end{lem}

\begin{proof}
Part (i) follows from \cite[Prop.~4.6]{JL15} (while the cited result is given for number fields, its proof holds for arbitrary Hilbertian fields). The extension to conic del Pezzo surfaces follows in the same way, noting that the existence of such surfaces with maximal action over number fields follows from Theorem~\ref{thm:dP4IGP} and Corollary~\ref{cor:KulkarniExtensionZero}.
\end{proof}

\begin{lem}
A relatively minimal conic del Pezzo surface of degree $1$ with maximal action (splitting group $W(D_7)$) exists over $\mathbb{Q}$.
\end{lem}

\begin{proof}
Follows directly from Corollary~\ref{cor:ConicdP1IGPQ}.
\end{proof}

\begin{lem}
For each $n \in \mathbb{Z}_{\geq 1}$ and $e \in \mathbb{Z}_{\geq 1}$, there exists a general position degree-$e$ closed point in $\mathbb{P}^2_k$ for $k$ a Hilbertian field.
\end{lem}

\begin{proof}
First note that, for every abelian group $A$, there exists a Galois extension $l/k$ with $\Gal(l/k) \cong A$ \cite[Prop.~16.3.5]{FriedJarden}. In particular, $\mathbb{P}^1_k$ admits closed points of degree $e$ for all $e \in \mathbb{Z}_{\geq 1}$, thus so does $\mathbb{P}^2_k$. Then observe that the condition that the underlying points are in general position can be phrased in terms of non-linear polynomial equations in the coefficients of the underlying geometric points as in \cite[Lem.~2.5]{BFL19}. The result then follows upon noting that $\text{Sym}^e(\mathbb{P}^2)$ is a rational variety, thus has the Hilbert property.
\end{proof}

\begin{prop}
Let $k$ be a Hilbertian field and let $d \in \{1,\dots,9\} \setminus \{7\}$.
\begin{enumerate}
    \item A relatively minimal degree-$d$ del Pezzo surface with splitting group $G_d$ and a $k$-rational point exists over $k$.
    \item For $d \geq 3$, a surface as in part (i) exists with a rational point in general position for all $e \geq 1$.
\end{enumerate}
\end{prop}

\begin{proof}

First we prove (i). As the result is trivial for degree $d \geq 8$, we assume henceforth that $d \leq 6$. Consider the Hilbert scheme $\mathcal{H}_{d,k}$ parametrising anticanonically embedded degree-$d$ del Pezzo surfaces over $k$ as in \cite[\S4]{JL15}; by \cite[Proof~of~Lem.~4.1]{JL15}, this scheme is a rational quasi-projective variety, thus $\mathcal{H}_{d,k}(k)$ is non-thin. Meanwhile, the scheme $\mathcal{L}_{d,k}$ parametrising the lines on these surfaces is a finite \'etale cover of $\mathcal{H}_{d,k}$ of degree at least $6$; the Galois action on the finite fibres corresponds to the Galois action on lines of the associated surfaces, and an application of Hilbert's Irreducibility Theorem ensures that this action is by the full $G_d$ outside of a thin locus $T$.

Next we prove (ii). First we show that a non-thin collection of surfaces contain a rational point. Note that we may form a fibration $\pi_d: X_d \rightarrow \mathcal{H}_{d,k}$ with the fibre above $P \in \mathcal{H}_{d,k}$ isomorphic to the associated del Pezzo surface. Note further that $\pi_d$ has geometrically integral generic fibre, thus the image $\pi_d(A) \subset \mathcal{H}_{d,K}$ of a non-thin set $A \subset X_d(k)$ is non-thin. We first observe that $\pi_d(X_d(k))$ is itself non-thin. To see this, note that $X_d$ has a second fibration (the ``coordinate fibration'') with base a projective space and fibres isomorphic to linear hypersurfaces. Then by a standard fibration argument for the Hilbert property \cite[Thm.~1.1]{BarySorokerFehmPetersen}, $X_d(k)$ is non-thin, thus $\pi_d(X_d(k))$ is non-thin, so $\pi_d(X_d(k)) \setminus T$ is non-thin. It remains only to consider the general position requirement, which is non-vacuous only for $d = 2$. It is enough to observe that all rational points not in general position are contained in a one-dimensional subspace of each fibre; by density of rational points in the total space, it follows that there exists at least one fibre, hence surface, with a rational point in general position. \qedhere

\end{proof}

We now set our sights on proving that closed points in general position exist on del Pezzo surfaces with a rational point over Hilbertian fields.

Our approach hinges on working along \emph{tangent sections}, which we now define. In what follows, we denote birational equivalence by $\sim$.

\begin{defn}
Let $X$ be a del Pezzo surface of degree $3$ or $4$ and $P \in X(k)$ not on a line of $X$. A \emph{tangent section} $C_P \subset X$ is the intersection of a hyperplane $T_P X$ tangent to $X$ at $P$ with $X$, which is a singular curve. This is an irreducible curve of geometric genus zero with a rational point, hence $C_P \sim \mathbb{P}^1$.
\end{defn}

Note that $C_P$ is unique if $d = 3$ and a member of a one-dimensional family if $d = 4$.

\begin{lem}
For $P \in X(k)$ a rational point on a del Pezzo surface $X$ of degree $3$ or $4$, we have $\text{Sym}^2(C_P) \sim \mathbb{P}^2_k$ and $\text{Sym}^3(C_P) \sim \mathbb{P}^3$.     
\end{lem}

\begin{proof}
Since $C_P$ is birational to $\mathbb{P}^1$, it suffices to show that $\text{Sym}^2(\mathbb{P}^1_k) \cong \mathbb{P}^2_k$ and \newline $\text{Sym}^3(\mathbb{P}^1_k) \cong \mathbb{P}^3_k$.
Note that we have the maps
\[
\begin{aligned}
\text{Sym}^2(\mathbb{A}^1_k) & \xrightarrow{\sim} \mathbb{A}^2_k, \quad \{x,y\} & \mapsto & \quad (x+y,xy), \\
\text{Sym}^3(\mathbb{A}^1_k) & \xrightarrow{\sim} \mathbb{A}^3_k, \quad \{x,y,z\} & \mapsto & \quad (x+y+z,xy + yz + zx,xyz). \\
\end{aligned}
\]
By interpretation of the coefficients of a univariate polynomial in terms of symmetric functions on its roots, it is clear that these two maps are isomorphisms. These isomorphisms induce the desired ones.
\end{proof}

We are now ready to prove the existence of closed points in general position.

\begin{prop} \label{prop:HilbertianGenPos}
Suppose given a positive integer $1 \leq d \leq 9$, positive integers $e_1,\dots,e_r$ such that $\sum_{i=1}^r e_i \leq d-1$ and a Hilbertian field $k$. Then on any degree-$d$ del Pezzo surface $X/k$ with $X(k) \neq \emptyset$ (and a rational point in general position for $d = 2$), there closed points $P_1,\dots,P_r$ on $X$ with $\deg(P_i) = e_i$ in general position. 
\end{prop}

\begin{proof}
First assume that $d \geq 5$. Then such $X$ is $k$-rational, and the proof is immediate from the existence of points in general position in the plane.

Now assume that $d \leq 4$. We separate into cases depending on $d,E:=\{e_i\}$. Note that, by Lemma~\ref{lem:rational points in general position}, we can (and do) assume that at least one $e_i$ is greater than $1$.

\begin{enumerate}
    \item If $d=4,E = \{2\}$, take $P \in X(k)$ and consider a tangent section $C_P \subset X$, which is birational to $\mathbb{P}^1$. Note that $\text{Sym}^2(C_P) \sim \mathbb{P}^2$. Note that each of the $16$ lines on $X$ intersects $C_P$ in 1 point. Then there are at most $16$ points on $C_P$ lying on a line on $X$, thus at most $16$ curves in $\text{Sym}^2(C_P) \sim \mathbb{P}^2$ corresponding to a pair of points such that one of them is on a line. Denote the $16$ associated curves in $\mathbb{P}^2$ by $L_1,\dots,L_{16}$. Next, for each of the $10$ families of conics $\pi_i: X \rightarrow \mathbb{P}^1$, a general member intersects $C_P$ in $2$ points; we thus obtain morphisms $\phi_i: \mathbb{P}^1 \rightarrow \text{Sym}^2(C_P)$, $i = 1,\dots,10$, sending $Q \in \mathbb{P}^1$ to the set of points $\pi_i^{-1}(Q) \cap C_P$. We obtain $10$ associated curves $C_1,\dots,C_{10} \subset \mathbb{P}^2$. Lastly, denote by $T$ the locus of points in $\text{Sym}^2(C_P)(k) \sim \mathbb{P}^2(k)$ corresponding to pairs of rational points, which is thin, coming from the degree-2 cover $C_P \times C_P \rightarrow \text{Sym}^2(C_P)$. Choosing $R \in \mathbb{P}^2(k) \setminus \left(\bigcup_{i=1}^{16} L_i(k) \cup \bigcup_{j=1}^{10} C_j(k) \cup T\right)$ (which we may do since $\mathbb{P}^2(k)$ is non-thin), the associated quadratic point on $X$ is in general position.
    \item If $d=4,E = \{2,1\}$, choose a conjugate pair of points $Q_1,Q_2$ as in case (1) and blow up to a degree-2 del Pezzo surface. Again, by density of rational points, we may choose a rational point $R'$ not on any line and not on the ramification divisor. Denoting by $R \in X(k)$ the associated point, $\{Q_1,Q_2,R\}$ is in general position.
    \item If $d=4,E = \{3\}$, take $P \in X(k)$ and a tangent section $C_P \subset X$, and consider $Z:=\text{Sym}^3(C_P) \sim \mathbb{P}^3$. There are 16 curves in $Z$ corresponding to a point on a line and two more points, 10 surfaces from two points on a conic in a particular family and one other point, and one surface corresponding to intersection with a tangent section of a point on $C_P$ (the intersection multiplicity is $4$, but the underlying intersection is always at most $3$ points away from $P$). Choose a rational point away from these subvarieties in $\mathbb{P}^3$ and away from the thin locus corresponding to reducible cubic points; then the associated triple of conjugate points $\{Q_1,Q_2,Q_3\}$ are in general position.
    \item If $d =3,E=\{2\}$, choose $P \in X(k)$ and construct the tangent section $C_P$. Again, consider $Y:= \text{Sym}^2(C_P) \sim \mathbb{P}^2$. Each of the 27 lines intersects $C_P$ in  $1$ points, so pairs with a point on a line are parametrised by $27$ curves. Further, pairs of points both contained in a singular anticanonical curve with a singularity at one of them are parametrised by the curve $C_P$, hence by a curve in $\text{Sym}^2 C_P \sim \mathbb{P}^2$. Choose a rational point away from these curves and away from the thin locus corresponding to pairs of rational points. \qedhere
\end{enumerate}
\end{proof}

\begin{cor}[Theorem~\ref{thm:hilb}]
Over a Hilbertian field, all blowup types exist in all degrees. Thus the inverse line and conic problems admit simultaneous surjective solutions.
\end{cor}

\appendix
\section{Tables} \label{sec:Tables}

\subsection{General fields} \label{sec:GenFieldTables}
In this section, we give tables detailing properties of del Pezzo surfaces of each degree.

Denoting by $X$ a del Pezzo surface:
\begin{itemize}
    \item The ``Type'' column gives its description as blowups of relatively minimal surfaces.
    \item The ``$(L_X,C_X)$'' column gives its line and conic counts respectively.
    \item The ``$I_X$'' column gives the intersection invariant.
    \item The ``$G_X$'' column gives its possible splitting groups.
    \item The ``$\#C$'' column gives the number of associated conjugacy classes in $G_{\deg(X)}$.
\end{itemize}
We denote by $T_n \leq \Sym_n$ a transitive subgroup of $\Sym_n$. We denote by $M_3 \leq W(E_6)$ and $M_6 \leq W(A_2 \times A_1)$ the splitting groups of  relatively minimal degree-$3$ and degree-$6$ del Pezzo surfaces respectively. For $d \in \{1,2,4\}$, we denote by $M_{d,I} \leq W(R_d)$ and $M_{d,\II} \leq W(D_{8-d})$ the splitting groups of relatively minimal degree-$d$ del Pezzo surfaces of Picard ranks $1$ and $2$ respectively. Note that:
\begin{enumerate}
\item $M_6 \in \{C_6,\Sym_3,\Dih_6\}$.
\item $T_3 \in \{C_3,\Sym_3\}$.
\item $T_4 \in \{C_2^2,C_4,\Dih_8,\Alt_4,\Sym_4\}$.
\item $T_5 \in \{C_5,\Dih_{10},\text{GA}(1,5),\Alt_5,\Sym_5\}$.
\end{enumerate}

\subsubsection{Degree 9}

1 relatively minimal class, 0 non-relatively minimal classes; 1 class in total.

\begin{longtable}{|l|l|l|l|l|l|}\caption{Relatively minimal del Pezzo surfaces of degree 9}\\
\hline
Type & $(L_X,C_X)$ & $I_X$ & $G_X$ & $\#C$ \\ 
\hline
$S_9$ & $(0,0)$ & 1 & 1 & 1 \\
\hline
\end{longtable}

\subsubsection{Degree 8}

2 relatively minimal classes, 1 non-relatively minimal class; 3 classes in total.

\begin{longtable}{|l|l|l|l|l|l|}\caption{Relatively minimal del Pezzo surfaces of degree 8}\\
\hline
Type & $(L_X,C_X)$ & $I_X$ & $G_X$ & $\#C$ \\ 
\hline
$S_{8,\II}$ & (0,2) & $-1$ & $1$ & $1$ \\
\hline
$S_{8,I}$ & (0,0) & 2 & $C_2$ & $1$ \\
\hline
\end{longtable}
\pagebreak
\begin{longtable}{|l|l|l|l|l|l|}\caption{Non-relatively minimal del Pezzo surfaces of degree 8}\\
\hline
Type & $(L_X,C_X)$ & $I_X$ & $G_X$ & $\#C$ \\ 
\hline
$\Bl_1S_9$ & $(1,1)$ & $-1$ & $1$ & $1$ \\
\hline
\end{longtable}

\subsubsection{Degree 7}

0 relatively minimal classes, 2 non-relatively minimal classes; 2 classes in total.

\begin{longtable}{|l|l|l|l|l|l|}\caption{Non-relatively minimal del Pezzo surfaces of degree 7}\\
\hline
Type & $(L_X,C_X)$ & $I_X$ & $G_X$ & $\#C$ \\ 
\hline
$\Bl_1S_{8,\II}$, $\Bl_{1,1}S_9$ & $(3,2)$ & 1 & 1 & 1  \\
\hline
$\Bl_1S_{8,I}$, $\Bl_2 S_9$ & $(1,0)$ & $-2$ & $C_2$ & 1 \\
\hline
\end{longtable}

\subsubsection{Degree 6}

3 relatively minimal classes, 7 non-relatively minimal classes; 10 classes in total.

\begin{longtable}{|l|l|l|l|l|l|}\caption{Relatively minimal del Pezzo surfaces of degree 6}\\
\hline
Type & $(L_X,C_X)$ & $I_X$ & $G_X$ & $\#C$ \\ 
\hline
$S_{6}$ & $(0,0)$ & 6 & $M_6$ &  $3$ \\
\hline
\end{longtable}

\begin{longtable}{|l|l|l|l|l|l|}\caption{Non-relatively minimal del Pezzo surfaces of degree 6}\\
\hline
Type & $(L_X,C_X)$ & $I_X$ & $G_X$ & $\#C$ \\ 
\hline
$\Bl_3S_{9}$ & $(0,0)$ & $-3$ & $T_3$ & $2$ \\
\hline
\makecell[l]{$\Bl_{2,1}S_9$ \\
$ \Bl_{1^2}S_{8,I}$
}
& $(2,1)$ & $2$ & $C_2$ & 1 \\
\hline
\makecell[l]{$\Bl_{1^3}S_9 $ \\
$ \Bl_{1^2}S_{8,\II}$} & $(6,3)$ & $-1$ & 1 & 1 \\
\hline
$\Bl_2S_{8,I}$ & $(0,1)$ & $-4$ & $C_2$ or $C_2^2$ & $2$ \\
\hline
$\Bl_2S_{8,\II}$ & $(0,3)$ & $2$ & $C_2$ & 1 \\
\hline
\end{longtable}

\subsubsection{Degree 5}

5 relatively minimal classes, 14 non-relatively minimal classes; 19 classes in total.

\begin{longtable}{|l|l|l|l|l|l|}\caption{Relatively minimal del Pezzo surfaces of degree 5}\\
\hline
Type & $(L_X,C_X)$ & $I_X$ & $G_X$ & $\#C$ \\ 
\hline
$S_5$ & $(0,0)$ & 5 & $T_5$ & 5 \\
\hline
\end{longtable}

\begin{longtable}{|l|l|l|l|l|}\caption{Non-relatively minimal del Pezzo surfaces of degree 5}\\
\hline
Type & $(L_X,C_X)$ & $I_X$ & $G_X$ & $\#C$ \\ 
\hline
$\Bl_4 S_9$ & $(0,1)$ & $-4$ & $T_4$ & 5 \\
\hline
\makecell[l]{$\Bl_{3,1}S_9 $ \\
$ \Bl_3 S_{8,\II}$} & $(1,2)$ & $3$ & $T_3$ & 2 \\
\hline
\makecell[l]{$\Bl_{2^2}S_9 $ \\
$\Bl_{2,1} S_{8,I}$} & $(2,1)$ & 4 & $C_2$ or $C_2^2$ & 2 \\
\hline
\makecell[l]{$\Bl_{2,1^2}S_9 $ \\
$\Bl_{1^3} S_{8,I}$ \\
$\Bl_{2,1}S_{8,\II}$} & $(4,3)$ & $-2$ & $C_2$ & 1 \\
\hline
\makecell[l]{$\Bl_{1^4}S_9 $ \\
$ \Bl_{1^3} S_{8,\II}$} & $(10,5)$ & 1 & 1 & 1 \\
\hline
\makecell[l]{$\Bl_1 S_6 $ \\
$\Bl_3S_{8,I}$} & $(1,0)$ & $-6$ & $M_6$ & 3 \\
\hline
\end{longtable}

\subsubsection{Degree 4}

148 relatively minimal classes, 49 non-relatively minimal classes; 197 classes in total.

\begin{longtable}{|l|l|l|l|l|}\caption{Relatively minimal del Pezzo surfaces of degree 4}\\
\hline
Type & $(L_X,C_X)$ & $I_X$ & $G_X$ & $\#C$ \\
\hline
$S_{4,I}$ & $(0,0)$ & $4$ & $M_{4,I} \leq W(D_5)$ & $98$ \\
\hline
$S_{4,\II}$ & $(0,2)$ & $-4$ & $M_{4,\II} \leq W(D_4)$ & $50$ \\
\hline
\end{longtable}

\begin{longtable}{|l|l|l|l|l|l|}\caption{Non-relatively minimal del Pezzo surfaces of degree 4}\\
\hline
Type & $(L_X,C_X)$ & $I_X$ & $G_X$ & $\#C$ \\ 
\hline
\makecell[l]{$\Bl_5 S_9 $ \\
$ \Bl_1 S_5$} & $(1,0)$ & $-5$ & $T_5$ & 5 \\
\hline
$\Bl_{4,1} S_9$ & $(2,2)$ & $4$ & $T_4$ & 5 \\
\hline
\makecell[l]{$\Bl_{3,2}S_9 $ \\
$ \Bl_{3,1}S_{8,I} $ \\
$ \Bl_{1,1}S_6$} & $(2,0)$ & $6$ & $M_6$ & 3  \\
\hline
\makecell[l]{$\Bl_{3,1^2}S_9 $ \\
$\Bl_{3,1}S_{8,\II}$} & $(4,4)$ & $-3$ & $T_3$ & 2 \\
\hline
\makecell[l]{$\Bl_{2^2,1}S_9 $ \\
$\Bl_{2,1^2}S_{8,I}$} & $(4,2)$& $-4$ & $C_2$ or $C_2^2$ &  2 \\
\hline
\makecell[l]{$\Bl_{2,1^3}S_9$ \\
$\Bl_{1^4}S_{8,I}$ \\
$\Bl_{2,1^2}S_{8,\II}$} & $(8,6)$ & 2 & $C_2$ & 1 \\
\hline
\makecell[l]{$\Bl_{1^5}S_9$ \\
$\Bl_{1^4}S_{8,\II}$} & $(16,10)$ & $-1$ & 1 &  1\\
\hline
$\Bl_4 S_{8,I}$ & $(0,0)$ & $-8$ & $T_4$ or $C_2\times T_4$ & 11 \\
\hline
$\Bl_{2^2}S_{8,I}$ & $(0,2)$ & 8 & $C_2$, $C_2^2$ or $C_2^3$ & 5 \\
\hline
$\Bl_{4}S_{8,\II}$ & $(0,4)$ & 4 & $T_4$ & 5 \\
\hline
$\Bl_{2^2}S_{8,\II}$ & $(0,6)$ & $-4$ & $C_2$ or $C_2^2$ & 2 \\
\hline
$\Bl_2 S_6$ & $(0,0)$ & $-12$ & $M_6$ or $M_6 \times C_2$ & 7 \\
\hline
\end{longtable}

\subsubsection{Degree 3}

137 relatively minimal classes, 213 non-relatively minimal classes; 350 classes in total. 

\begin{longtable}{|l|l|l|l|l|l|}\caption{Relatively minimal del Pezzo surfaces of degree 3}\\
\hline
Type & $L_X = C_X$ & $I_X$ & $G_X$ & $\#C$ \\ 
\hline
$S_3$ & 0 & 3 & $M_3 \leq W(E_6)$ & 137 \\ 
\hline
\end{longtable}
\pagebreak
\begin{longtable}{|l|l|l|l|l|l|}\caption{Non-relatively minimal del Pezzo surfaces of degree 3}\\
\hline
Type & $L_X = C_X$ & $I_X$ & $G_X$ & $\#C$ \\ 
\hline
$\Bl_6 S_9$ & $0$ & $-6$ & $T_6$ &  $16$ \\
\hline
\makecell[l]{$\Bl_{5,1} S_9$ \\
$\Bl_5 S_{8,\II}$ \\
$\Bl_{1^2}S_5$} & $2$ & $5$ & $T_5$ & $5$ \\
\hline
\makecell[l]{$\Bl_{4,2}S_9$ \\
$ \Bl_{4,1}S_{8,I}$} & $1$ & 8 & $T_4$ or $C_2 \times T_4$ & $11$ \\
\hline
\makecell[l]{$\Bl_{4,1^2}S_9 $ \\
$\Bl_{4,1}S_{8,\II}$} & $5$ & $-4$ & $T_4$ & $5$ \\
\hline
$\Bl_{3^2} S_9$ & $0$ & 9 & \makecell[l]{$T_3$, $T_3 \times C_3$ \\ or $C_3 \ltimes \Sym_3$} & $6$ \\
\hline
\makecell[l]{$\Bl_{3,2,1}S_9 $ \\
$\Bl_{3,1^2}S_{8,I} $ \\
$\Bl_{3,2}S_{8,\II}$ \\
$\Bl_{1^3}S_6$} & $3$ & $-6$ & $M_6$ & $3$ \\
\hline
\makecell[l]{$\Bl_{3,1^3}S_9 $ \\
$\Bl_{3,1^2}S_{8,\II}$} & $9$ & 3 & $T_3$ & $2$ \\
\hline
\makecell[l]{$\Bl_{2^3}S_9 $ \\
$\Bl_{2^2,1}S_{8,I}$} & $3$ & $-8$ & $C_2$, $C_2^2$ or $C_2^3$ & $4$ \\
\hline
\makecell[l]{$\Bl_{2^2,1^2}S_9 $ \\
$\Bl_{2,1^3}S_{8,I}$ \\
$ \Bl_{2^2,1}S_{8,\II}$} & $7$ & 4 & $C_2$ or $C_2^2$ & $2$ \\
\hline
\makecell[l]{$\Bl_{2,1^4}S_9$ \\
$\Bl_{1^5}S_{8,I}$ \\
$ \Bl_{2,1^3}S_{8,\II}$} & $15$ & $-2$ & $C_2$ & 1 \\
\hline
\makecell[l]{$\Bl_{1^6}S_9$ \\
$\Bl_{1^5}S_{8,\II}$} & $27$ & 1 & 1 & 1 \\
\hline
\makecell[l]{$\Bl_5S_{8,I}$ \\
$\Bl_2 S_5$} & $0$ & $-10$ & \makecell[l]{$T_5$ or $T_5 \times C_2$ \\ (except $C_5$)} & $8$ \\
\hline
\makecell[l]{$\Bl_{3,2}S_{8,I}$ \\
$\Bl_{2,1}S_6$} & $1$ & $12$ & $M_6$ or $M_6 \times C_2$ & $7$ \\
\hline
$\Bl_3 S_6$ & $0$ & $-18$ & \makecell[l]{$M_6$, $M_6 \times C_3$ \\ or $M_6 \times \Sym_3$} & $11$ \\
\hline
$\Bl_1 S_{4,I}$ & $1$ & $-4$ & $M_{4,I}$ & $98$  \\
\hline
$\Bl_1 S_{4,\II}$ & $3$ & $4$ & $M_{4,\II}$ & $33$ \\
\hline
\end{longtable}

\subsubsection{Degree 2}

6877 relatively minimal classes, 1197 non-relatively minimal classes; 8074 classes in total.

\begin{longtable}{|l|l|l|l|l|l|}\caption{Relatively minimal del Pezzo surfaces of degree 2}\\
\hline
Type & $(L_X,C_X)$ & $I_X$ & $G_X$ & $\#C$ \\ 
\hline
$S_{2,I}$ & $(0,0)$ &  $2$ & $M_{2,I} \leq W(E_7)$ & 5565 \\
\hline
$S_{2,\II}$ & $(0,2)$ & $-4$ & $M_{2,\II} \leq W(D_7)$ & 1312 \\
\hline
\end{longtable}
\pagebreak
Number of relatively minimal classes: 6877.

\begin{longtable}{|l|l|l|l|l|l|}\caption{Non-relatively minimal del Pezzo surfaces of degree 2}\\
\hline
Type & $(L_X,C_X)$ & $I_X$ & $G_X$ & $\#C$ \\ 
\hline
$\Bl_7 S_9$ & $(0,0)$ & $-7$ & $T_7$ & $7$ \\
\hline
$\Bl_{6,1} S_9$ & $(2,2)$ & $6$ & $T_6$ & $16$ \\
\hline
\makecell[l]{$\Bl_{5,2} S_9$ \\
$\Bl_{5,1}S_{8,I}$ \\
$\Bl_{2,1}S_5$} & $(2,0)$ & $10$ & \makecell[l]{$T_5$ or $T_5 \times C_2$ \\ (except $C_5$)} & $8$ \\
\hline
\makecell[l]{$\Bl_{5,1^2} S_9$ \\
$\Bl_{5,1}S_{8,\II}$ \\
$\Bl_{1^3}S_5$} & $(6,6)$ & $-5$ & $T_5$ & $5$ \\
\hline
$\Bl_{4,3} S_9$ & $(0,2)$ & $12$ & \makecell[l]{$T_4 \times T_3$, \\ $T_4 \times_{C_3} T_3$ or \\ $T_4 \times_{\Sym_3} T_3$} & $18$ \\
\hline
\makecell[l]{$\Bl_{4,2,1} S_9$ \\
$\Bl_{4,1^2}S_{8,I}$} & $(4,4)$ & $-8$ & $T_4$ or $T_4 \times C_2$ & $11$ \\
\hline
\makecell[l]{$\Bl_{4,1^3} S_9$ \\
$\Bl_{4,1^2}S_{8,\II}$} & $(12,14)$ & 4 & $T_4$ & $5$ \\
\hline
\makecell[l]{$\Bl_{3^2,1} S_9$ \\
$\Bl_{3^2}S_{8,\II}$} & $(2,6)$ & $-9$ & \makecell[l]{$T_3$, $T_3 \times C_2$ \\ or $C_3 \ltimes \Sym_3$} & $6$ \\
\hline
\makecell[l]{$\Bl_{3,2^2} S_9$ \\
$\Bl_{3,2,1}S_{8,I}$ \\
$\Bl_{2,1^2}S_6$} & $(4,2)$ & $-12$ & $M_6$ or $M_6 \times C_2$ & $7$ \\
\hline
\makecell[l]{$\Bl_{3,2,1^2} S_9$ \\
$\Bl_{3,1^3}S_{8,I}$ \\
$\Bl_{3,2,1}S_{8,\II}$ \\
$\Bl_{1^4}S_6$} & $(8,12)$ & $6$ & $M_6$ & $3$\\
\hline
\makecell[l]{$\Bl_{3,1^4} S_9$ \\
$\Bl_{3,1^3}S_{8,\II}$} & $(20,30)$ & $-3$ & $T_3$ & $2$ \\
\hline
\makecell[l]{$\Bl_{2^3,1} S_9$ \\
$\Bl_{2^2,1^2}S_{8,I}$} & $(8,8)$ & $8$ & $C_2$, $C_2^2$ or $C_2^3$ & $4$ \\
\hline
\makecell[l]{$\Bl_{2^2,1^3} S_9$ \\
$\Bl_{2,1^4}S_{8,I}$ \\
$\Bl_{2^2,1^2}S_{8,\II}$} & $(16,26)$ & $-4$ & $C_2$ or $C_2^2$ & $2$ \\
\hline
\makecell[l]{$\Bl_{2,1^5} S_9$ \\
$\Bl_{1^6}S_{8,I}$ \\
$\Bl_{2,1^4}S_{8,\II}$} & $(32,60)$ & $2$ & $C_2$ & $1$ \\
\hline
\makecell[l]{$\Bl_{1^7} S_9$ \\
$\Bl_{1^6}S_{8,\II}$} & $(56, 126)$ & $-1$ & $1$ & $1$ \\
\hline
$\Bl_6 S_{8,I}$ & $(0,0)$ & $-12$ & \makecell[l]{$T_6 \times C_2$ or \\ $T_6 \times_{C_2} C_2$} & $36$ \\
\hline
$\Bl_6 S_{8,\II}$ & $(0,6)$ & $6$ & $T_6$ & $16$ \\
\hline
$\Bl_{4,2} S_{8,I}$ & $(0,2)$ & $16$ & \makecell[l]{$T_4$, $C_2\times T_4$ \\ or $(T_4 \times_{C_2} C_2) \times_{C_2} C_2$ } & $41$ \\
\hline
$\Bl_{4,2} S_{8,\II}$ & $(0,12)$ & $-8$ & \makecell[l]{$T_4 \times C_2$ or \\ $T_4 \times_{C_2} C_2$} & $11$ \\
\hline
\makecell[l]{$\Bl_{3^2} S_{8,I}$ \\
$\Bl_{3,1}S_6$} & $(2,0)$ & $18$ & \makecell[l]{$M_6$, $M_6 \times C_3$ \\ or $M_6 \times \Sym_3$} & $11$ \\
\hline
$\Bl_{2^3} S_{8,I}$ & $(0,6)$ & $-16$ & \makecell[l]{$C_2$, $C_2^2$, \\ $C_2^3$ or $C_2^4$} & $12$ \\
\hline
$\Bl_{2^3} S_{8,\II}$ & $(0,24)$ & $8$ & \makecell[l]{$C_2$, $C_2^2$ \\ or $C_2^3$} & $4$ \\
\hline
$\Bl_4 S_6$ & $(0,0)$ & $-24$ & \makecell[l]{$T_4 \times M_6$ or \\ $T_4 \times_{Q_6} M_6$, \\ $Q_6 \in \{C_6,\Sym_3,\Dih_6\}$} & $55$ \\
\hline
$\Bl_{2^2} S_6$ & $(0,0)$ & $24$ & \makecell[l]{$C_6$, $\Sym_3$, $\Dih_6$, \\ $C_6 \times C_2$, $\Dih_6 \times C_2$, \\
$C_2^2 \times C_6$, $\Dih_6 \times C_2^2$} & $16$ \\
\hline
$\Bl_3 S_5$ & $(0,0)$ & $-15$ & \makecell[l]{$T_5 \times T_3$, \\ $T_5 \times_{C_3} T_3$ or \\ $T_5 \times_{\Sym_3} \Sym_3$} & 13 \\
\hline
$\Bl_2 S_{4,I}$ & $(0,0)$ & $-8$ & \makecell[l]{$M_{4,I} \times C_2$ or \\ $M_{4,I} \times_{C_2} C_2$} & 435 \\
\hline
$\Bl_2 S_{4,\II}$ & $(0,4)$ & $8$ & \makecell[l]{$M_{4,\II} \times C_2$ or \\ $M_{4,\II} \times_{C_2} C_2$} & $183$ \\
\hline
$\Bl_{1^2} S_{4,I}$ & $(4,2)$ & $4$ & $M_{4,I}$ & $98$ \\
\hline
$\Bl_{1^2} S_{4,\II}$ & $(8,6)$ & $-4$ & $M_{4,\II}$ & $33$ \\
\hline
$\Bl_1 S_{3,M_3}$ & $(2,0)$ & $-3$ & $M_3$ & $137$ \\
\hline
\end{longtable}

\subsubsection{Degree 1}
53016 relatively minimal classes, 9076 non-relatively minimal classes; 62092 classes in total.

\begin{longtable}{|l|l|l|l|l|l|}\caption{Relatively minimal del Pezzo surfaces of degree 1}\\
\hline
Type & $(L_X,C_X)$ & $I_X$ & $G_X$ & $\#C$ \\ 
\hline
$S_{1,I}$ & $(0,0)$ & $1$ & $M_{1,I} \leq W(E_8)$ & $48797$ \\
\hline
$S_{1,\II}$ & $(0,2)$ & $-4$ & $M_{1,\II} \leq W(D_8)$ & $4219$ \\
\hline
\end{longtable}

Number of relatively minimal classes: $53016$.

\begin{longtable}{|l|l|l|l|l|l|}\caption{Non-relatively minimal del Pezzo surfaces of degree 1}\\
\hline
Type & $(L_X,C_X)$ & $I_X$ & $G_X$ & $\#C$ \\ 
\hline
$\Bl_8S_9$ & $(0,0)$ & $-8$ & $T_8$ & $50$ \\
\hline
\makecell[l]{$\Bl_{7,1} S_9$ \\
$ \Bl_7 S_{8,\II}$} & $(2,4)$ & 7 & $T_7$ & $7$ \\
\hline
\makecell[l]{$\Bl_{6,2} S_9$ \\
$ \Bl_{6,1}S_{8,I}$} & $(2,0)$ & 12 & \makecell[l]{$T_6 \times C_2$ or \\
$T_6 \times_{C_2} C_2$} & $36$ \\
\hline
\makecell[l]{$\Bl_{6,1^2} S_9$ \\
$\Bl_{6,1}S_{8,\II}$} & $(8,12)$ & $-6$ & $T_6$ & $16$ \\
\hline
\makecell[l]{$\Bl_{5,3} S_9$ \\
$\Bl_{3,1}S_5$} & $(2,0)$ & $15$ & \makecell[l]{$T_5 \times T_3$, \\ $T_5 \times_{C_3} T_3$ or \\ $T_5 \times_{\Sym_3} \Sym_3$} & $13$ \\
\hline
\makecell[l]{$\Bl_{5,2,1} S_9$ \\
$\Bl_{5,1^2}S_{8,I}$ \\
$\Bl_{2,1^2}S_5$} & $(6,6)$ & $-10$ & \makecell[l]{$T_5$ or $T_5 \times C_2$ \\ (except $C_5$)} & $8$ \\
\hline
\makecell[l]{$\Bl_{5,1^3}S_9$ \\
$\Bl_{5,1^2}S_{8,\II}$ \\
$\Bl_{1^4}S_5$} & $(20,30)$ & $5$ & $T_5$ & $5$ \\
\hline
$\Bl_{4^2} S_9$ & $(0,4)$ & $16$ & \makecell[l]{$T_4 \times_{Q_4} T_4'$, \\
$A_4,T_4' \in \{C_2^2,C_4,\Dih_8,\Alt_4,\Sym_4\}$} & $50$ \\
\hline
\makecell[l]{$\Bl_{4,3,1}S_9$ \\
$\Bl_{4,3}S_{8,\II}$} & $(4,12)$ & $-12$ & \makecell[l]{$T_4 \times T_3$, \\ $T_4 \times_{C_3} T_3$ or \\ $T_4 \times_{\Sym_3} \Sym_3$} & $18$ \\
\hline
$\Bl_{4,2^2} S_9$ & $(4,6)$ & $-16$ & \makecell[l]{$T_4$, $C_2\times T_4$ \\ or $(T_4 \times_{C_2} C_2) \times_{C_2} C_2$ } & 32 \\
\hline
\makecell[l]{$\Bl_{4,2,1^2}S_9$ \\
$\Bl_{4,1^3}S_{8,I}$} & $(14,30)$ & $8$ & $T_4$ or $T_4 \times C_2$ & $11$ \\
\hline
\makecell[l]{$\Bl_{4,1^4}S_9$ \\
$\Bl_{4,1^3}S_{8,\II}$} & $(40,90)$ & $-4$ & $T_4$ & $5$ \\
\hline
\makecell[l]{$\Bl_{3^2,2}S_9$ \\
$\Bl_{3^2,1}S_{8,I}$ \\
$\Bl_{3,1^2}S_6$} & $(6,0)$ & $-18$ & \makecell[l]{$M_6$, $M_6 \times C_3$ \\ or $M_6 \times \Sym_3$} & $11$ \\
\hline
\makecell[l]{$\Bl_{3^2,1^2}S_9$ \\
$\Bl_{3^2,1}S_{8,\II}$} & $(12,36)$ & $9$ & \makecell[l]{$T_3$, $T_3 \times C_2$ \\ or $C_3 \ltimes \Sym_3$} & $6$ \\
\hline
\makecell[l]{$\Bl_{3,2^2,1}S_9$ \\
$\Bl_{3,2,1^2}S_{8,I}$ \\
$\Bl_{2,1^3}S_6$} & $(12,18)$ & $12$ & $M_6$ or $M_6 \times C_2$ & $7$ \\
\hline
\makecell[l]{$\Bl_{3,2,1^3}S_9$ \\
$\Bl_{3,1^4}S_{8,I}$ \\
$\Bl_{3,2,1^2}S_{8,\II}$ \\
$ \Bl_{1^5}S_6$} & $(30,90)$ & $-6$ & $M_6$ & $3$ \\
\hline
\makecell[l]{$\Bl_{3,1^5}S_9$ \\
$\Bl_{3,1^4}S_{8,\II}$} & $(72,270)$ & $3$ & $T_3$ & $2$ \\
\hline
\makecell[l]{$\Bl_{2^4}S_9$ \\
$\Bl_{2^3,1}S_{8,I}$} & $(8,24)$ & $16$ & \makecell[l]{$C_2$, $C_2^2$, \\ $C_2^3$ or $C_2^4$} & $8$ \\
\hline
\makecell[l]{$\Bl_{2^3,1^2}S_9$ \\
$\Bl_{2^2,1^3}S_{8,I}$ \\
$\Bl_{2^3,1}S_{8,\II}$} & $(26,72)$ & $-8$ & $C_2$, $C_2^2$ or $C_2^3$ & $4$ \\
\hline
\makecell[l]{$\Bl_{2^2,1^4}S_9$ \\
$\Bl_{2,1^5}S_{8,I}$ \\
$\Bl_{2^2,1^3}S_{8,\II}$} & $(60,252)$ & $4$ & $C_2$ or $C_2^2$ & $2$ \\
\hline
\makecell[l]{$\Bl_{2,1^6}S_9$ \\
$\Bl_{1^7}S_{8,I}$ \\
$\Bl_{2,1^5}S_{8,\II}$} & $(126,756)$ & $-2$ & $C_2$ & $1$ \\
\hline
\makecell[l]{$\Bl_{1^8}S_9$ \\
$\Bl_{1^7}S_{8,\II}$} & $(240, 2160)$ & $1$ & $1$ & $1$ \\
\hline
$\Bl_{7} S_{8,I}$ & $(0,0)$ & $-14$ & \makecell[l]{$T_7 \times C_2$ or \\ $T_7 \times_{C_2} C_2$} & $10$  \\
\hline
\makecell[l]{$\Bl_{5,2}S_{8,I}$ \\
$\Bl_{2^2} S_5$} &  $(0,2)$ & $20$ & \makecell[l]{ $G \times C_2$ or $G \times_{C_2} C_2$, \\
$G = T_5$ (not $C_5$) or $G = T_5 \times C_2$ } & $19$ \\
\hline
\makecell[l]{$\Bl_{4,3}S_{8,I}$ \\
$\Bl_{4,1}S_6$} &  $(2,0)$ & $24$ & \makecell[l]{$T_4 \times M_6$ or \\ $T_4 \times_{Q_6} M_6$, \\ $Q_6 \in \{C_6,\Sym_3,\Dih_6\}$} & $55$ \\
\hline
\makecell[l]{$\Bl_{3,2^2}S_{8,I}$ \\
$\Bl_{2^2,1}S_6$} & $(2,6)$ & $-24$ & \makecell[l]{$C_6$, $\Sym_3$, $\Dih_6$, \\ $C_6 \times C_2$, $\Dih_6 \times C_2$, \\
$C_2^2 \times C_6$, $\Dih_6 \times C_2^2$} & $16$ \\
\hline
$\Bl_{5} S_6$ & $(0,0)$ & $-30$ & \makecell[l]{$T_5 \times M_6$ or \\ $T_5 \times_{Q_6} M_6$, \\ $Q_6 \in \{C_6,\Sym_3,\Dih_6\}$} & $30$ \\
\hline
$\Bl_{3,2} S_6$ & $(0,0)$ & $36$ & \makecell[l]{ $G \times C_2$ or $G \times_{C_2} C_2$, \\ $G \in \{M_6, M_6 \times C_3, M_6 \times \Sym_3\}$ } & $29$ \\
\hline
$\Bl_{4} S_5$ & $(0,0)$ & $-20$ & \makecell[l]{$T_4 \times_{Q_4} T_5$, \\ $q_4 \in \{C_2^2,C_4,\Dih_8,\Alt_4,\Sym_4\}$} & $44$ \\
\hline
$\Bl_{3} S_{4,I}$ & $(0,0)$ & $-12$ & \makecell[l]{$M_{4,I} \times T_3$, \\ $M_{4,I} \times_{C_3} T_3$ or \\ $M_{4,I} \times_{\Sym_3} \Sym_3$} & $568$ \\
\hline
$\Bl_{2,1} S_{4,I}$ & $(2,2)$ & $8$ & \makecell[l]{$M_{4,I} \times C_2$ or \\ $M_{4,I} \times_{C_2} C_2$} & $435$ \\
\hline
$\Bl_{1^3} S_{4,I}$ & $(12,6)$ & $-4$ & $M_{4,I}$ & $98$ \\
\hline
$\Bl_{3} S_{4,\II}$ & $(0,6)$ & $12$ & \makecell[l]{$M_{4,\II} \times T_3$, \\ $M_{4,\II} \times_{C_3} T_3$ or \\ $M_{4,\II} \times_{\Sym_3} \Sym_3$} & $150$ \\
\hline
$\Bl_{2,1} S_{4,\II}$ & $(6,12)$& $-8$ & \makecell[l]{$M_{4,\II} \times C_2$ or \\ $M_{4,\II} \times_{C_2} C_2$} & $108$  \\
\hline
$\Bl_{1^3} S_{4,\II}$ & $(24,24)$ & $4$ & $M_{4,\II}$ & $33$ \\
\hline
$\Bl_{2} S_{3,M_3}$ & $(0,0)$  & $-6$ & \makecell[l]{$M_3 \times C_2$ or \\ $M_3 \times_{C_2} C_2$} & $391$ \\
\hline
$\Bl_{1^2} S_{3,M_3}$ & $(6,0)$ & $3$ & $M_3$ & $137$ \\
\hline
$\Bl_{1} S_{2,I}$ & $(2,0)$ & $-2$ & $M_{2,I}$ & $5565$ \\
\hline
$\Bl_{1} S_{2,\II}$ & $(4,4)$ & $4$ & $M_{2,\II}$ & $1092$ \\
\hline
\end{longtable}

\subsection{Finite fields} \label{sec:TablesFinite}

In this section, we provide tables detailing the possible Galois actions on del Pezzo surfaces in each degree, together with associated properties.

Denoting by $X$ a del Pezzo surface over the finite field $\mathbb{F}_{q}$:
\begin{itemize}
    \item The ``Class'' column gives the conjugacy class in the Weyl group. For degree $d \geq 4$, we use our own numbering. For degree $3$, we use Manin's numbering from \cite[Table~1]{Man86}. For degrees $2$ and $1$, we use Urabe's numbering from \cite{Urabe96}.
    \item The ``Type'' column gives all descriptions as blowups of minimal surfaces.
    \item The ``$(L_X,C_X)$'' column lists the line and conic counts respectively.
    \item The ``$I_X$'' column lists the intersection invariant.
    \item The ``Order'' column gives the order $n$ of the cyclic splitting group $C_n$.
    \item The ``$a(X)$'' column gives the trace of Frobenius.
    \item The ``Restrictions'' column gives a necessary and sufficient condition on $q$ for existence coming from either \cite[\S5.2]{BFL19} or Section~\ref{sec:FiniteFieldsProofs}.
\end{itemize}

Lastly, in tables for degrees 1 and 2, we add columns giving the class and blowup type of the Bertini/Geiser twist respectively, leaving these entries blank when the class is self-dual.

\subsubsection{Degree 9} 1 relatively minimal class, 0 non-relatively minimal classes, 1 class in total.

\begin{longtable}{|l|l|l|l|l|l|l|}\caption{Relatively minimal del Pezzo surfaces of degree 9}\\
\hline
Class & Type & $(L_X,C_X)$ & $I_X$ & Order & $a(X)$ & Restrictions \\ 
\hline
1 & $S_9$ & $(0,0)$ & 1 & 1 & 1 & All $q$ \\
\hline
\end{longtable}

 \subsubsection{Degree 8} 2 relatively minimal classes, 1 non-relatively minimal class, 3 classes in total.

\begin{longtable}{|l|l|l|l|l|l|l|}\caption{Relatively minimal del Pezzo surfaces of degree 8}\\
\hline
Class & Type & $(L_X,C_X)$ & $I_X$ & Order & $a(X)$ & Restrictions \\ 
\hline
1 & $S_{8,I}$ & (0,0) & 2 & 2 & $0$ & All $q$ \\
\hline
2 & $S_{8,\II}$ & (0,2) & $-1$ & $1$ & 2 & All $q$ \\
\hline
\end{longtable}

\begin{longtable}{|l|l|l|l|l|l|l|}\caption{Non-relatively minimal del Pezzo surfaces of degree 8}\\
\hline
Class & Type & $(L_X,C_X)$ & $I_X$ & Order & $a(X)$ & Restrictions \\ 
\hline
3 & $\Bl_1S_9$ & $(1,1)$ & $-1$ & $1$ & $1$ & All $q$ \\
\hline
\end{longtable}

 \subsubsection{Degree 7} 0 relatively minimal classes, 2 non-relatively minimal classes, 2 classes in total.
 
\begin{longtable}{|l|l|l|l|l|l|l|}\caption{Non-relatively minimal del Pezzo surfaces of degree 7}\\
\hline
Class & Type & $(L_X,C_X)$ & $I_X$ & Order & $a(X)$ & Restrictions \\ 
\hline
1 & $\Bl_1S_{8,I}$, $\Bl_2 S_9$ & $(1,0)$ & $-2$ & 2 & 1 & All $q$ \\
\hline
2 & $\Bl_1S_{8,\II}$, $\Bl_{1,1}S_9$ & $(3,2)$ & 1 & 1 & 3 & All $q$  \\
\hline
\end{longtable}

 \subsubsection{Degree 6} 1 relatively minimal class, 5 non-relatively minimal classes, 6 classes in total.

\begin{longtable}{|l|l|l|l|l|l|l|}\caption{Relatively minimal del Pezzo surfaces of degree 6}\\
\hline
Class & Type & $(L_X,C_X)$ & $I_X$ & Order & $a(X)$ & Restrictions \\
\hline
1 & $S_6$ & $(0,0)$ & 6 & $6$ & $-1$ & All $q$ \\
\hline
\end{longtable}

\begin{longtable}{|l|l|l|l|l|l|l|}\caption{Non-relatively minimal del Pezzo surfaces of degree 6}\\
\hline
Class & Type & $(L_X,C_X)$ & $I_X$ & Order & $a(X)$ & Restrictions \\ 
\hline
2 & $\Bl_2S_{8,I}$ & $(0,1)$ & $-4$ & $2$ & 0 & All $q$ \\
\hline
3 & $\Bl_2S_{8,\II}$ & $(0,3)$ & $2$ & $2$ & 2 & All $q$ \\
\hline
4 & $\Bl_3S_{9}$ & $(0,0)$ & $-3$ & $3$ & 1 &  All $q$ \\
\hline
5 & \makecell[l]{$\Bl_{2,1}S_9$ \\
$ \Bl_{1^2}S_{8,I}$
}
& $(2,1)$ & $2$ & $2$ & $2$ & All $q$ \\
\hline
6 & \makecell[l]{$\Bl_{1^3}S_9 $ \\
$ \Bl_{1^2}S_{8,\II}$} & $(6,3)$ & $-1$ & 1 & 4 & All $q$ \\
\hline
\end{longtable}

 \subsubsection{Degree 5} 1 relatively minimal class, 6 non-relatively minimal classes, 7 classes in total.

\begin{longtable}{|l|l|l|l|l|l|l|}\caption{Relatively minimal del Pezzo surfaces of degree 5}\\
\hline
Class & Type & $(L_X,C_X)$ & $I_X$ & Order & $a(X)$ & Restrictions \\ 
\hline
1 & $S_5$ & $(0,0)$ & 5 & $5$ & 0 & All $q$ \\
\hline
\end{longtable}

\begin{longtable}{|l|l|l|l|l|l|l|}\caption{Non-relatively minimal del Pezzo surfaces of degree 5}\\
\hline
Class & Type & $(L_X,C_X)$ & $I_X$ & Order & $a(X)$ & Restrictions \\ 
\hline
2 & \makecell[l]{$\Bl_1 S_6 $ \\
$\Bl_3S_{8,I}$} & $(1,0)$ & $-6$ & $6$ & 0 & All $q$ \\
\hline
3 & $\Bl_4 S_9$ & $(0,1)$ & $-4$ & $4$ & 1 & All $q$ \\
\hline
4 & \makecell[l]{$\Bl_{3,1}S_9 $ \\
$ \Bl_3 S_{8,\II}$} & $(1,2)$ & $3$ & $3$ & 2 & All $q$ \\
\hline
5 & \makecell[l]{$\Bl_{2^2}S_9 $ \\
$\Bl_{2,1} S_{8,I}$} & $(2,1)$ & 4 & $2$ & 1 & All $q$ \\
\hline
6 & \makecell[l]{$\Bl_{2,1^2}S_9 $ \\
$\Bl_{1^3} S_{8,I}$ \\
$\Bl_{2,1}S_{8,\II}$} & $(4,3)$ & $-2$ & $2$ & 3 & All $q$ \\
\hline
7 & \makecell[l]{$\Bl_{1^4}S_9 $ \\
$ \Bl_{1^3} S_{8,\II}$} & $(10,5)$ & 1 & 1 & 5 & All $q$ \\
\hline
\end{longtable}

 \subsubsection{Degree 4} 6 relatively minimal classes, 12 non-relatively minimal classes, 18 classes in total.

\begin{longtable}{|l|l|l|l|l|l|l|}\caption{Relatively minimal del Pezzo surfaces of degree 4}\\
\hline
Class & Type & $(L_X,C_X)$ & $I_X$ & Order & $a(X)$ &  Restrictions \\
\hline
1 & $S_{4,I}^{(4,1)}$ & $(0,0)$ & $4$ & $8$ & $0$ & All $q$ \\
\hline
2 & $S_{4,I}^{(3,2)}$ & $(0,0)$ & $4$ & $12$ & $1$ & All $q$ \\
\hline
3 & $S_{4,I}^{(2,1^3)}$ & $(0,0)$ & $4$ & $4$ & $-2$ & All $q$ \\
\hline
4 & $S_{4,\II}^{(3,1)}$ & $(0,2)$ & $-4$ & $6$ & $1$ & All $q$ \\
\hline
5 & $S_{4,\II}^{(2^2)}$ & $(0,2)$ & $-4$ & $4$ & $2$ & $q \geq 3$ \\
\hline
6 & $S_{4,\II}^{(1^4)}$ & $(0,2)$ & $-4$ & $2$ & $-2$ & $q \geq 4$ \\
\hline
\end{longtable}

\begin{longtable}{|l|l|l|l|l|l|l|}\caption{Non-relatively minimal del Pezzo surfaces of degree 4}\\
\hline
Class & Type & $(L_X,C_X)$ & $I_X$ & Order & $a(X)$ & Restrictions \\ 
\hline
7 & $\Bl_2 S_6$ & $(0,0)$ & $-12$ & 6 & -1 & All $q$ \\
\hline
8 & $\Bl_4 S_{8,I}$ & $(0,0)$ & $-8$ & 4 & 0 & $q \geq 3$ \\
\hline
9 & $\Bl_{2^2}S_{8,I}$ & $(0,2)$ & 8 & $2$ & 0 & All $q$ \\
\hline
10 & $\Bl_{4}S_{8,\II}$ & $(0,4)$ & 4 & 4 & 2 & All $q$ \\
\hline
11 & $\Bl_{2^2}S_{8,\II}$ & $(0,6)$ & $-4$ & $2$ & 2 & $q \geq 4$ \\
\hline
12 & \makecell[l]{$\Bl_5 S_9 $ \\
$ \Bl_1 S_5$} & $(1,0)$ & $-5$ & $5$ & 1 & All $q$ \\
\hline
13 & $\Bl_{4,1} S_9$ & $(2,2)$ & $4$ & $4$ & 2 & All $q$ \\
\hline
14 & \makecell[l]{$\Bl_{3,2}S_9 $ \\
$ \Bl_{3,1}S_{8,I} $ \\
$ \Bl_{1,1}S_{6,M_6}$} & $(2,0)$ & $6$ & $6$ & 1 & All $q$ \\
\hline
15 & \makecell[l]{$\Bl_{3,1^2}S_9 $ \\
$\Bl_{3,1}S_{8,\II}$} & $(4,4)$ & $-3$ & $3$ & 3 & All $q$ \\
\hline
16 & \makecell[l]{$\Bl_{2^2,1}S_9 $ \\
$\Bl_{2,1^2}S_{8,I}$} & $(4,2)$& $-4$ & $2$ &  2 & $q \geq 3$ \\
\hline
17 & \makecell[l]{$\Bl_{2,1^3}S_9$ \\
$\Bl_{1^4}S_{8,I}$ \\
$\Bl_{2,1^2}S_{8,\II}$} & $(8,6)$ & 2 & $2$ & 4 & All $q$ \\
\hline
18 & \makecell[l]{$\Bl_{1^5}S_9$ \\
$\Bl_{1^4}S_{8,\II}$} & $(16,10)$ & $-1$ & 1 &  6 & $q \geq 4$ \\
\hline
\end{longtable}

 \subsubsection{Degree 3} 5 relatively minimal classes, 20 non-relatively minimal classes, 25 classes in total.

\begin{longtable}{|l|l|l|l|l|l|l|}\caption{Relatively minimal del Pezzo surfaces of degree 3}\\
\hline
Class & Type & $L_X = C_X$ & $I_X$ & Order & $a(X)$ & Restrictions \\ 
\hline
1 & $S_3^1$ & $0$ & $3$ & $12$ & 0 & All $q$ \\
\hline
2 & $S_3^2$ & $0$ & $3$ & $6$ & 2 & All $q$ \\
\hline
3 & $S_3^3$ & $0$ & $3$ & $3$ & -2 & All $q$ \\
\hline
4 & $S_3^4$ & $0$ & $3$ & $9$ & 1 & All $q$ \\
\hline
5 & $S_3^5$ & $0$ & $3$ & $6$ & -1 & $q \geq 3$ \\
\hline
\end{longtable}

\begin{longtable}{|l|l|l|l|l|l|l|}\caption{Non-relatively minimal del Pezzo surfaces of degree 3}\\
\hline
Class & Type & $L_X = C_X$ & $I_X$ & Order & $a(X)$ & Restrictions \\ 
\hline
6 & $\Bl_1 S_{4,I}^{(3,2)}$ & $1$ & $-4$ & $12$  & $2$ & All $q$ \\
\hline
7 & $\Bl_1 S_{4,I}^{(4,1)}$ & $1$ & $-4$ & $8$ & $1$ & All $q$ \\
\hline
8 & $\Bl_1 S_{4,\II}^{(3,1)}$ & $3$ & $4$ & $6$ & $2$ & All $q$ \\
\hline
9 & $\Bl_1 S_{4,I}^{(2,1^3)}$ & $1$ & $-4$ & $4$ & $-1$ & All $q$ \\
\hline
10 & $\Bl_1 S_{4,\II}^{(2^2)}$ & $3$ & $4$ & $4$ & $3$ & $q \geq 3$ \\
\hline
11 & $\Bl_1 S_{4,\II}^{(1^4)}$ & $3$ & $4$ & $2$ & $-1$ & $q \geq 4$ \\
\hline
12 & \makecell[l]{$\Bl_5S_{8,I}$ \\
$\Bl_2 S_5$} & $0$ & $-10$ & $10$ & $0$ & All $q$ \\
\hline
13 & $\Bl_3 S_6$ & $0$ & $-18$ & $6$ & $-1$ & All $q$ \\
\hline
14 & \makecell[l]{$\Bl_{3,2}S_{8,I}$ \\
$\Bl_{2,1}S_6$} & $1$ & $12$ & $6$ & $0$ & All $q$ \\
\hline
15 & $\Bl_6 S_9$ & $0$ & $-6$ & $6$ & $1$ & All $q$ \\
\hline
16 & \makecell[l]{$\Bl_{5,1} S_9$ \\
$\Bl_5 S_{8,\II}$ \\
$\Bl_{1^2}S_5$} & $2$ & $5$ & $5$ & $2$ & All $q$ \\
\hline
17 & \makecell[l]{$\Bl_{4,2}S_9$ \\
$ \Bl_{4,1}S_{8,I}$} & $1$ & 8 & $4$ & $1$ & $q \geq 3$ \\
\hline
18 & $\Bl_{3^2} S_9$ & $0$ & 9 & $3$ & $1$ & $q \geq 3$ \\
\hline
19 & \makecell[l]{$\Bl_{4,1^2}S_9 $ \\
$\Bl_{4,1}S_{8,\II}$} & $5$ & $-4$ & $4$ & $3$ & All $q$ \\
\hline
20 & \makecell[l]{$\Bl_{3,2,1}S_9 $ \\
$\Bl_{3,1^2}S_{8,I} $ \\
$\Bl_{3,2}S_{8,\II}$ \\
$\Bl_{1^3}S_6$} & $3$ & $-6$ & $6$ & $2$ & All $q$ \\
\hline
21 & \makecell[l]{$\Bl_{2^3}S_9 $ \\
$\Bl_{2^2,1}S_{8,I}$} & $3$ & $-8$ & $2$ & $1$ & All $q$ \\
\hline
22 & \makecell[l]{$\Bl_{3,1^3}S_9 $ \\
$\Bl_{3,1^2}S_{8,\II}$} & $9$ & $3$ & $3$ & $4$ & All $q$ \\
\hline
23 & \makecell[l]{$\Bl_{2^2,1^2}S_9 $ \\
$\Bl_{2,1^3}S_{8,I}$ \\
$ \Bl_{2^2,1}S_{8,\II}$} & $7$ & 4 & $2$ & $3$ & $q \geq 4$ \\
\hline
24 & \makecell[l]{$\Bl_{2,1^4}S_9$ \\
$\Bl_{1^5}S_{8,I}$ \\
$ \Bl_{2,1^3}S_{8,\II}$} & $15$ & $-2$ & $2$ & $5$ & All $q$ \\
\hline
25 & \makecell[l]{$\Bl_{1^6}S_9$ \\
$\Bl_{1^5}S_{8,\II}$} & $27$ & $1$ & $1$ & $7$ & $q \geq 7$ \\
\hline
\end{longtable}

 \subsubsection{Degree 2} 18 relatively minimal classes, 42 non-relatively minimal classes, 60 classes in total.

\begin{longtable}{|l|l|l|l|l|l|l|l|l|}\caption{Relatively minimal del Pezzo surfaces of degree 2}\\
\hline
Class & Type & \makecell[l]{Geiser \\ Class} & \makecell[l]{Geiser \\ Type} & $(L_X,C_X)$ & $I_X$ & Order & $a(X)$ & Restrictions \\ 
\hline
2 & $S_{2,\II}^{(1^6)}$ & $47$ & \makecell[l]{$\Bl_{2,1^5} S_9$ \\
$\Bl_{1^6}S_{8,I}$ \\
$\Bl_{2,1^4}S_{8,\II}$} & $(0,2)$ & $-4$ & 2 & $-4$ & $q \geq 5$ \\
\hline
3 & $S_{2,\II}^{(2^2,1^2)}$ & $31$ & $\Bl_2 S_{4,\II}^{(2^2)}$ & $(0,2)$ & $-4$ & 4 & $0$ & $q \geq 3$ \\
\hline
4 & $S_{2,\II}^{(3,1^3)}$ & $51$ & \makecell[l]{$\Bl_{3,2,1^2} S_9$ \\
$\Bl_{3,1^3}S_{8,I}$ \\
$\Bl_{3,2,1}S_{8,\II}$ \\
$\Bl_{1^4}S_6$} & $(0,2)$ & $-4$ & 6 & $-1$ & $q \geq 3$  \\
\hline
5 & $S_{2,\II}^{(5,1)}$ & $58$ & \makecell[l]{$\Bl_{5,2} S_9$ \\
$\Bl_{5,1}S_{8,I}$ \\
$\Bl_{2,1}S_5$} & $(0,2)$ & $-4$ & 10 & 1 & All $q$ \\
\hline
6 & $S_{2,\II}^{(4,2)}$ & $11$ & $S_{2,I}^4$ & $(0,2)$ & $-4$ & 8 & 2 & All $q$ \\
\hline
7 & $S_{2,\II}^{(3,3)}$ & $42$ & \makecell[l]{$\Bl_{3^2} S_{8,I}$ \\
$\Bl_{3,1}S_6$} & $(0,2)$ & $-4$ & 6 & 2 & All $q$  \\
\hline
8 & $S_{2,I}^1$ & $46$ & \makecell[l]{$\Bl_{1^7} S_9$ \\
$\Bl_{1^6}S_{8,\II}$} & $(0,0)$ & $2$ & 2 & $-6$ & $q 
\geq 9$ \\
\hline
9 & $S_{2,I}^2$ & $27$ & $\Bl_{1^2} S_{4,\II}^{(2^2)}$ & $(0,0)$ & $2$ & 4 & $-2$ & $q \geq 3$ \\
\hline
10 & $S_{2,I}^3$ & $24$ & $\Bl_1 S_{3}^{2}$ & $(0,0)$ & $2$ & 6 & $-1$ & All $q$ \\
\hline
11 & $S_{2,I}^4$ & $6$ & $S_{2,\II}^{(4,2)}$ & $(0,0)$ & $2$ & 8 & $0$ & All $q$ \\
\hline
12 & $S_{2,I}^5$ & $49$ & \makecell[l]{$\Bl_{3,1^4} S_9$ \\
$\Bl_{3,1^3}S_{8,\II}$} & $(0,0)$ & $2$ & 6 & $-3$ & $q \geq 3$  \\
\hline
13 & $S_{2,I}^6$ & $56$ & \makecell[l]{$\Bl_{5,1^2} S_9$ \\
$\Bl_{5,1}S_{8,\II}$ \\
$\Bl_{1^3}S_5$} & $(0,0)$ & $2$ & 10 & $-1$ & All $q$  \\
\hline
14 & $S_{2,I}^7$ & $54$ & \makecell[l]{$\Bl_{3^2,1} S_9$ \\
$\Bl_{3^2}S_{8,\II}$} & $(0,0)$ & $2$ & 6 & $0$ & $q \geq 3$  \\
\hline
15 & $S_{2,I}^8$ & $23$ & $\Bl_1 S_{3}^{4}$ & $(0,0)$ & $2$ & 18 & $0$ & All $q$ \\
\hline
16 & $S_{2,I}^9$ & $60$ & $\Bl_7 S_9$ & $(0,0)$ & $2$ & 14 & $1$ & All $q$ \\
\hline
17 & $S_{2,I}^{10}$ & $22$ & $\Bl_1 S_{3}^{1}$ & $(0,0)$ & $2$ & 12 & $1$ & All $q$ \\
\hline
18 & $S_{2,I}^{11}$ & $37$ & $\Bl_3 S_5$ & $(0,0)$ & $2$ & 30 & $2$ & All $q$ \\
\hline
19 & $S_{2,I}^{12}$ & $20$ & $\Bl_1 S_{3}^{3}$ & $(0,0)$ & $2$ & 6 & $3$ & $q \geq 3$ \\
\hline
\end{longtable}

\begin{longtable}{|l|l|l|l|l|l|l|l|l|}\caption{Non-relatively minimal del Pezzo surfaces of degree 2}\\
\hline
Class & Type & \makecell[l]{Geiser \\ Class} & \makecell[l]{Geiser \\ Type} & $(L_X,C_X)$ & $I_X$ & Order & $a(X)$ & Restrictions \\ 
\hline
1 \footnote{This class was incorrectly identified by Urabe \cite{Urabe96} as corresponding to a minimal surface, hence the numbering.} & $\Bl_{2^3}S_{8,I}$ & 50 & \makecell[l]{$\Bl_{2^3,1} S_9$ \\
$\Bl_{2^2,1^2}S_{8,I}$} & $(0,6)$ & $-16$ & 2 & $0$ & $q \geq 3$ \\
\hline
20 & $\Bl_1 S_{3}^{3}$ & 19 & $S_{2,I}^{12}$ & $(2,0)$ & $-3$ & $3$ & $-1$ & $q \geq 3$ \\
\hline
21 & $\Bl_1 S_{3}^{5}$ & 45 & $\Bl_6 S_{8,\II}$ & $(2,0)$ & $-3$ & 6 & $0$ & $q \geq 3$ \\
\hline
22 & $\Bl_1 S_{3}^{1}$ & 17 & $S_{2,I}^{10}$ & $(2,0)$ & $-3$ & 12 & $1$ & All $q$ \\
\hline
23 & $\Bl_1 S_{3}^{4}$ & 15 & $S_{2,I}^8$ & $(2,0)$ & $-3$ & 9 & $2$ & All $q$ \\
\hline
24 & $\Bl_1 S_{3}^{2}$ & 10 & $S_{2,I}^3$ & $(2,0)$ & $-3$ & 6 & $3$ & All $q$ \\
\hline
25 & $\Bl_{1^2} S_{4,\II}^{(1^4)}$ & 40 & $\Bl_{2^3} S_{8,\II}$ & $(8,6)$ & $-4$ & 2 & $0$ & $q \geq 7$ \\
\hline
26 & $\Bl_{1^2} S_{4,\II}^{(3,1)}$ & 38 & $\Bl_{2^2} S_6$ & $(8,6)$ & $-4$ & 6 & $3$ & All $q$ \\
\hline
27 & $\Bl_{1^2} S_{4,\II}^{(2^2)}$ & 9 & $S_{2,I}^2$ & $(8,6)$ & $-4$ & 4 & $4$ & $q \geq 3$ \\
\hline
28 & $\Bl_2 S_{4,\II}^{(1^4)}$ & 48 & \makecell[l]{$\Bl_{2^2,1^3} S_9$ \\
$\Bl_{2,1^4}S_{8,I}$ \\
$\Bl_{2^2,1^2}S_{8,\II}$} & $(0,4)$ & $8$ & 2 & $-2$ & $q \geq 5$ \\
\hline
29 & $\Bl_{1^2} S_{4,I}^{(2,1^3)}$ & 41 & $\Bl_{4,2} S_{8,\II}$ & $(4,2)$ & $4$ & 4 & $0$ & $q \geq 3$ \\
\hline
30 & $\Bl_2 S_{4,\II}^{(3,1)}$ & 53 & \makecell[l]{$\Bl_{3,2^2} S_9$ \\
$\Bl_{3,2,1}S_{8,I}$ \\
$\Bl_{2,1^2}S_6$} & $(0,4)$ & $8$ & 6 & $1$ & All $q$ \\
\hline
31 & $\Bl_2 S_{4,\II}^{(2^2)}$ & 3 & $S_{2,\II}^{(2^2,1^2)}$ & $(0,4)$ & $8$ & 4 & $2$ & $q \geq 3$ \\
\hline
32 & $\Bl_{1^2} S_{4,I}^{(4,1)}$ & 35 & $\Bl_2 S_{4,I}^{(4,1)}$ & $(4,2)$ & $4$ & 8 & $2$ & All $q$ \\
\hline
33 & $\Bl_{1^2} S_{4,I}^{(3,2)}$ & 39 & $\Bl_{4} S_6$ & $(4,2)$ & $4$ & 12 & $3$ & All $q$ \\
\hline
34 & $\Bl_2 S_{4,I}^{(2,1^3)}$ & 52 & \makecell[l]{$\Bl_{4,1^3} S_9$ \\
$\Bl_{4,1^2}S_{8,\II}$} & $(0,0)$ & $-8$ & 4 & $-2$ & $q \geq 3$ \\
\hline
35 & $\Bl_2 S_{4,I}^{(4,1)}$ & 32 & $\Bl_{1^2} S_{4,I}^{(4,1)}$ & $(0,0)$ & $-8$ & 8 & 0 & All $q$ \\
\hline
36 & $\Bl_2 S_{4,I}^{(3,2)}$ & 57 & $\Bl_{4,3} S_9$ & $(0,0)$ & $-8$ & 12 & 1 & All $q$ \\
\hline
37 & $\Bl_3 S_5$ & 18 & $S_{2,I}^{11}$ & $(0,0)$ & $-15$ & 15 & $0$ & All $q$ \\
\hline
38 & $\Bl_{2^2} S_6$ & 26 & $\Bl_{1^2} S_{4,\II}^{(3,1)}$ & $(0,0)$ & $-24$ & 6 & $-1$ & All $q$ \\
\hline
39 & $\Bl_{4} S_6$ & 33 & $\Bl_{1^2} S_{4,I}^{(3,2)}$ & $(0,0)$ & $24$ & 12 & $-1$ & All $q$ \\
\hline
40 & $\Bl_{2^3} S_{8,\II}$ & 25 & $\Bl_{1^2} S_{4,\II}^{(1^4)}$ & $(0,24)$ & $8$ & 2 & $-1$ & $q \geq 7$ \\
\hline
41 & $\Bl_{4,2} S_{8,\II}$ & 29 & $\Bl_{1^2} S_{4,I}^{(2,1^3)}$ & $(0,12)$ & $-8$ & 4 & $-1$ & $q \geq 3$ \\
\hline
42 & \makecell[l]{$\Bl_{3^2} S_{8,I}$ \\
$\Bl_{3,1}S_6$} & 7 & $S_{2,\II}^{(3,3)}$ & $(2,0)$ & $18$ & 6 & $2$ & All $q$ \\
\hline
43 & $\Bl_{4,2} S_{8,I}$ & 55 & \makecell[l]{$\Bl_{4,2,1} S_9$ \\
$\Bl_{4,1^2}S_{8,I}$} & $(0,2)$ & $16$ & 4 & $2$ & $q \geq 3$ \\
\hline
44 & $\Bl_6 S_{8,\II}$ & 59 & $\Bl_{6,1} S_9$ & $(0,6)$ & $6$ & 6 & 0 & All $q$ \\
\hline
45 & $\Bl_6 S_{8,I}$ & 21 & $\Bl_1 S_{3}^{5}$ & $(0,0)$ & $-12$ & 6 & 0 & $q \geq 3$ \\
\hline
46 & \makecell[l]{$\Bl_{1^7} S_9$ \\
$\Bl_{1^6}S_{8,\II}$} & 8 & $S_{2,I}^1$ & $(56, 126)$ & $-1$ & 1 & $8$ & $q \geq 9$ \\
\hline
47 & \makecell[l]{$\Bl_{2,1^5} S_9$ \\
$\Bl_{1^6}S_{8,I}$ \\
$\Bl_{2,1^4}S_{8,\II}$} & 2 & $S_{2,\II}^{(1^6)}$ & $(32,60)$ & $2$ & 2 & $6$ & $q \geq 5$ \\
\hline
48 & \makecell[l]{$\Bl_{2^2,1^3} S_9$ \\
$\Bl_{2,1^4}S_{8,I}$ \\
$\Bl_{2^2,1^2}S_{8,\II}$} & 28 & $\Bl_2 S_{4,\II}^{(1^4)}$ & $(16,26)$ & $-4$ & 2 & $4$ & $q \geq 5$ \\
\hline
49 & \makecell[l]{$\Bl_{3,1^4} S_9$ \\
$\Bl_{3,1^3}S_{8,\II}$} & 12 & $S_{2,I}^5$ & $(20,30)$ & $-3$ & 3 & $5$ & $q \geq 3$ \\
\hline
50 & \makecell[l]{$\Bl_{2^3,1} S_9$ \\
$\Bl_{2^2,1^2}S_{8,I}$} & 1 & $\Bl_{2^3}S_{8,I}$ & $(8,8)$ & $8$ & 2 & $2$ & $q \geq 3$ \\
\hline
51 & \makecell[l]{$\Bl_{3,2,1^2} S_9$ \\
$\Bl_{3,1^3}S_{8,I}$ \\
$\Bl_{3,2,1}S_{8,\II}$ \\
$\Bl_{1^4}S_6$} & 4 & $S_{2,\II}^{(3,1^3)}$ & $(8,12)$ & $6$ & 6 & $3$ & $q \geq 3$ \\
\hline
52 & \makecell[l]{$\Bl_{4,1^3} S_9$ \\
$\Bl_{4,1^2}S_{8,\II}$} & 34 & $\Bl_2 S_{4,I}^{(2,1^3)}$ & $(12,14)$ & 4 & 4 & 4 & $q \geq 3$ \\
\hline
53 & \makecell[l]{$\Bl_{3,2^2} S_9$ \\
$\Bl_{3,2,1}S_{8,I}$ \\
$\Bl_{2,1^2}S_6$} & 30 & $\Bl_2 S_{4,\II}^{(3,1)}$ & $(4,2)$ & $-12$ & 6 & 1 & All $q$ \\
\hline
54 & \makecell[l]{$\Bl_{3^2,1} S_9$ \\
$\Bl_{3^2}S_{8,\II}$} & 14 & $S_{2,I}^7$ & $(2,6)$ & $-9$ & 3 & $2$ & $q \geq 3$ \\
\hline
55 & \makecell[l]{$\Bl_{4,2,1} S_9$ \\
$\Bl_{4,1^2}S_{8,I}$} & 43 & $\Bl_{4,2} S_{8,I}$ & $(4,4)$ & $-8$ & 4 & 2 & $q \geq 3$ \\
\hline
56 & \makecell[l]{$\Bl_{5,1^2} S_9$ \\
$\Bl_{5,1}S_{8,\II}$ \\
$\Bl_{1^3}S_5$} & 13 & $S_{2,I}^6$ & $(6,6)$ & $-5$ & 5 & 3 & All $q$ \\
\hline
57 & $\Bl_{4,3} S_9$ & 36 & $\Bl_2 S_{4,I}^{(3,2)}$ & $(0,2)$ & $12$ & 12 & 1 & All $q$ \\
\hline
58 & \makecell[l]{$\Bl_{5,2} S_9$ \\
$\Bl_{5,1}S_{8,I}$ \\
$\Bl_{2,1}S_5$} & 5 & $S_{2,\II}^{(5,1)}$ & $(2,0)$ & $10$ & 10 & 1 & All $q$ \\
\hline
59 & $\Bl_{6,1} S_9$ & 44 & $\Bl_6 S_{8,I}$ & $(2,2)$ & $6$ & 6 & 2 & All $q$ \\
\hline
60 & $\Bl_7 S_9$ & 16 & $S_{2,I}^9$ & $(0,0)$ & $-7$ & 7 & 1 & All $q$ \\
\hline
\end{longtable}

 \subsubsection{Degree 1} 37 relatively minimal classes, 83 non-relatively minimal classes, 120 classes in total.

\begin{longtable}{|l|l|l|l|l|l|l|l|l|}\caption{Relatively minimal del Pezzo surfaces of degree 1}\\
\hline
Class & Type & \makecell[l]{Bertini \\ Class} & \makecell[l]{Bertini \\ Type} &  $(L_X,C_X)$ & $I_X$ & Order & $a(X)$ & Restrictions \\ 
\hline
1 & $S_{1,\II}^{(2,1^5)}$ & 97 & \makecell[l]{$\Bl_{4,1^4}S_9$ \\
$\Bl_{4,1^3}S_{8,\II}$} & $(0,2)$ & $-4$ & 4 & $-3$ & $q \geq 5$ \\
\hline
2 & $S_{1,\II}^{(3,2,1^2)}$ & 105 & \makecell[l]{$\Bl_{4,3,1}S_9$ \\
$\Bl_{4,3}S_{8,\II}$} & $(0,2)$ & $-4$ & 12 & $0$ & \\
\hline
3 & $S_{1,\II}^{(2^3,1)}$ & -- & -- & $(0,2)$ & $-4$ & 4 & $1$ & \\
\hline
4 & $S_{1,\II}^{(4,1^3)}$ & 73 & $\Bl_{1^3} S_{4,I}^{(4,1)}$ & $(0,2)$ & $-4$ & 8 & $-1$ & \\
\hline
5 & $S_{1,\II}^{(6,1)}$ & -- & -- & $(0,2)$ & $-4$ & 12 & $1$ & \\
\hline
6 & $S_{1,\II}^{(5,2)}$ & 84 & $\Bl_{4} S_5$ & $(0,2)$ & $-4$ & 20 & $2$ & \\
\hline
7 & $S_{1,\II}^{(4,3)}$ & 82 & $\Bl_{3} S_{4,I}^{(4,1)}$ & $(0,2)$ & $-4$ & 24 & $2$ & \\
\hline
8 & $S_{1,I}^1$ & 91 & \makecell[l]{$\Bl_{1^8}S_9$ \\
$\Bl_{1^7}S_{8,\II}$} & $(0,0)$ & $1$ & 2 & $-7$ & \makecell[l]{$q = 16$ \\ or \\ $q \geq 19$} \\
\hline
9 & $S_{1,I}^2$ & 37 & $S_{1,I}^{30}$ & $(0,0)$ & $1$ & 3 & $-3$ & \\
\hline
10 & $S_{1,I}^3$ & 68 & $\Bl_{1^3} S_{4,\II}^{(2^2)}$ & $(0,0)$ & $1$ & 4 & $-3$ & \\
\hline
11 & $S_{1,I}^4$ & 35 & $S_{1,I}^{28}$ & $(0,0)$ & $1$ & 5 & $-1$ & \\
\hline
12 & $S_{1,I}^5$ & 60 & $\Bl_{1^2} S_3^2$ & $(0,0)$ & $1$ & 6 & $-2$ & \\
\hline
13 & $S_{1,I}^6$ & 42 & $\Bl_1 S_{2,\II}^{(4,2)}$ & $(0,0)$ & $1$ & 8 & $-1$ & \\
\hline
14 & $S_{1,I}^7$ & 33 & $S_{1,I}^{26}$ & $(0,0)$ & $1$ & 9 & $0$ & \\
\hline
15 & $S_{1,I}^8$ & 94 & \makecell[l]{$\Bl_{3,1^5}S_9$ \\
$\Bl_{3,1^4}S_{8,\II}$} & $(0,0)$ & $1$ & 6 & $-4$ & $q \geq 7$ \\
\hline
16 & $S_{1,I}^9$ & 100 & \makecell[l]{$\Bl_{3^2,1^2}S_9$ \\
$\Bl_{3^2,1}S_{8,\II}$} & $(0,0)$ & $1$ & 6 & $-1$ & \\
\hline
17 & $S_{1,I}^{10}$ & -- & -- & $(0,0)$ & $1$ & 4 & $1$ & \\
\hline
18 & $S_{1,I}^{11}$ & 78 & $\Bl_{3} S_{4,\II}^{(2^2)}$ & $(0,0)$ & $1$ & 12 & $0$ & \\
\hline
19 & $S_{1,I}^{12}$ & 102 & \makecell[l]{$\Bl_{5,1^3}S_9$ \\
$\Bl_{5,1^2}S_{8,\II}$ \\
$\Bl_{1^4}S_5$} & $(0,0)$ & $1$ & 10 & $-2$ & \\
\hline
20 & $S_{1,I}^{13}$ & 111 & \makecell[l]{$\Bl_{7,1} S_9$ \\
$ \Bl_7 S_{8,\II}$} & $(0,0)$ & $1$ & 14 & $0$ & All $q$ \\
\hline
21 & $S_{1,I}^{14}$ & -- & -- & $(0,0)$ & $1$ & 12 & $1$ & \\
\hline
22 & $S_{1,I}^{15}$ & 109 & \makecell[l]{$\Bl_{5,3} S_9$ \\
$\Bl_{3,1}S_5$} & $(0,0)$ & $1$ & 30 & $1$ & \\
\hline
23 & $S_{1,I}^{16}$ & -- & -- & $(0,0)$ & $1$ & 8 & $1$ & \\
\hline
24 & $S_{1,I}^{17}$ & 36 & $S_{1,I}^{29}$ & $(0,0)$ & $1$ & 12 & $-1$ & \\
\hline
25 & $S_{1,I}^{18}$ & -- & -- & $(0,0)$ & $1$ & 6 & $1$ & \\
\hline
26 & $S_{1,I}^{19}$ & 59 & $\Bl_{1^2} S_3^4$ & $(0,0)$ & $1$ & 18 & $-1$ & \\
\hline
27 & $S_{1,I}^{20}$ & 58 & $\Bl_{1^2} S_3^1$ & $(0,0)$ & $1$ & 12 & $0$ & \\
\hline
28 & $S_{1,I}^{21}$ & 56 & $\Bl_{1^2} S_3^3$ & $(0,0)$ & $1$ & 6 & $2$ & \\
\hline
29 & $S_{1,I}^{22}$ & 34 & $S_{1,I}^{27}$ & $(0,0)$ & $1$ & 30 & $0$ & \\
\hline
30 & $S_{1,I}^{23}$ & -- & -- & $(0,0)$ & $1$ & 24 & $1$ & \\
\hline
31 & $S_{1,I}^{24}$ & -- & -- & $(0,0)$ & $1$ & 20 & $1$ & \\
\hline
32 & $S_{1,I}^{25}$ & -- & -- & $(0,0)$ & $1$ & 12 & $1$ & \\
\hline
33 & $S_{1,I}^{26}$ & 14 & $S_{1,I}^7$ & $(0,0)$ & $1$ & 18 & $2$ & \\
\hline
34 & $S_{1,I}^{27}$ & 29 & $S_{1,I}^{22}$ & $(0,0)$ & $1$ & 15 & $2$ & \\
\hline
35 & $S_{1,I}^{28}$ & 11 & $S_{1,I}^4$ & $(0,0)$ & $1$ & 10 & $3$ & \\
\hline
36 & $S_{1,I}^{29}$ & 24 & $S_{1,I}^{17}$ & $(0,0)$ & $1$ & 12 & $3$ & \\
\hline
37 & $S_{1,I}^{30}$ & 9 & $S_{1,I}^2$ & $(0,0)$ & $1$ & 6 & $5$ & \\
\hline
\end{longtable}

\begin{longtable}{|l|l|l|l|l|l|l|l|l|}\caption{Non-relatively minimal del Pezzo surfaces of degree 1}\\
\hline
Class & Type & \makecell[l]{Bertini 
\\Class} & \makecell[l]{Bertini \\ Type} & $(L_X,C_X)$ & $I_X$ & Order & $a(X)$ & Restrictions \\ 
\hline
38 & $\Bl_1 S_{2,\II}^{(1^6)}$ & 93 & \makecell[l]{$\Bl_{2^2,1^4}S_9$ \\
$\Bl_{2,1^5}S_{8,I}$ \\
$\Bl_{2^2,1^3}S_{8,\II}$} &  $(4,4)$ & $4$ & 2 & $-3$ & \\
\hline
39 & $\Bl_1 S_{2,\II}^{(2^2,1^2)}$ & -- & -- &  $(4,4)$ & $4$ & 4 & $1$ & \\
\hline
40 & $\Bl_1 S_{2,\II}^{(3,1^3)}$ & 99 & \makecell[l]{$\Bl_{3,2^2,1}S_9$ \\
$\Bl_{3,2,1^2}S_{8,I}$ \\
$\Bl_{2,1^3}S_6$} &  $(4,4)$ & $4$ & 6 & $0$ & \\
\hline
41 & $\Bl_1 S_{2,\II}^{(5,1)}$ & 89 & \makecell[l]{$\Bl_{5,2}S_{8,I}$ \\
$\Bl_{2^2}S_5$} &  $(4,4)$ & $4$ & 10 & $2$ & \\
\hline
42 & $\Bl_1 S_{2,\II}^{(4,2)}$ & 13 & $S_{1,I}^6$ &  $(4,4)$ & $4$ & 8 & $3$ & \\
\hline
43 & $\Bl_1 S_{2,\II}^{(3,3)}$ & 85 & $\Bl_{3,2} S_6$  &  $(4,4)$ & $4$ & 6 & $3$ & \\
\hline
44 & $\Bl_1 S_{2,I}^1$ & 92 & \makecell[l]{$\Bl_{2,1^6}S_9$ \\
$\Bl_{1^7}S_{8,I}$ \\
$\Bl_{2,1^5}S_{8,\II}$} &  $(2,0)$ & $-2$ & 2 & $-5$ & $q \geq 11$ \\
\hline
45 & $\Bl_1 S_{2,I}^2$ & 72 & $\Bl_{2,1} S_{4,\II}^{(2^2)}$ &  $(2,0)$ & $-2$ & 4 & $-1$ & \\
\hline
46 & $\Bl_1 S_{2,I}^3$ & 65 & $\Bl_{2} S_3^2$ &  $(2,0)$ & $-2$ & 6 & $0$ & \\
\hline
47 & $\Bl_1 S_{2,I}^4$ & 112 & $\Bl_8S_9$ &  $(2,0)$ & $-2$ & 8 & $1$ & All $q$ \\
\hline
48 & $\Bl_1 S_{2,I}^5$ & 96 & \makecell[l]{$\Bl_{3,2,1^3}S_9$ \\
$\Bl_{3,1^4}S_{8,I}$ \\
$\Bl_{3,2,1^2}S_{8,\II}$ \\
$ \Bl_{1^5}S_6$} &  $(2,0)$ & $-2$ & 6 & $-2$ & \\
\hline
49 & $\Bl_1 S_{2,I}^6$ & 106 & \makecell[l]{$\Bl_{5,2,1} S_9$ \\
$\Bl_{5,1^2}S_{8,I}$ \\
$\Bl_{2,1^2}S_5$} &  $(2,0)$ & $-2$ & 10 & $0$ & \\
\hline
50 & $\Bl_1 S_{2,I}^7$ & 103 & \makecell[l]{$\Bl_{3^2,2}S_9$ \\
$\Bl_{3^2,1}S_{8,I}$ \\
$\Bl_{3,1^2}S_6$} &  $(2,0)$ & $-2$ & 6 & $1$ & \\
\hline
51 & $\Bl_1 S_{2,I}^8$ & 64 & $\Bl_{2} S_3^4$ &  $(2,0)$ & $-2$ & 18 & $1$ & \\
\hline
52 & $\Bl_1 S_{2,I}^9$ & 90 & $\Bl_{7} S_{8,I}$ &  $(2,0)$ & $-2$ & 14 & $2$ &  \\
\hline
53 & $\Bl_1 S_{2,I}^{10}$ & 63 & $\Bl_{2} S_3^1$ &  $(2,0)$ & $-2$ & 12 & $2$ & \\
\hline
54 & $\Bl_1 S_{2,I}^{11}$ & 86 & $\Bl_{5} S_6$ &  $(2,0)$ & $-2$ & 30 & $3$ & \\
\hline
55 & $\Bl_1 S_{2,I}^{12}$ & 61 & $\Bl_{2} S_3^3$ &  $(2,0)$ & $-2$ & 6 & $4$ & \\
\hline
56 & $\Bl_{1^2} S_3^3$ & 28 & $S_{1,I}^{21}$ &  $(6,0)$ & $3$ & 3 & $0$ & \\
\hline
57 & $\Bl_{1^2} S_3^5$ & -- & -- &  $(6,0)$ & $3$ & 6 & $1$ & \\
\hline
58 & $\Bl_{1^2} S_3^1$ & 27 & $S_{1,I}^{20}$ &  $(6,0)$ & $3$ & 12 & $2$ & \\
\hline
59 & $\Bl_{1^2} S_3^4$ & 26 & $S_{1,I}^{19}$ &  $(6,0)$ & $3$ & 9 & $3$ & \\
\hline
60 & $\Bl_{1^2} S_3^2$ & 12 & $S_{1,I}^5$ &  $(6,0)$ & $3$ & 6 & $4$ & \\
\hline
61 & $\Bl_{2} S_3^3$ & 55 & $\Bl_1 S_{2,I}^{12}$ &  $(0,0)$ & $3$ & 6 & $-2$ & \\
\hline
62 & $\Bl_{2} S_3^5$ & 107 & \makecell[l]{$\Bl_{6,1^2} S_9$ \\
$\Bl_{6,1}S_{8,\II}$} &  $(0,0)$ & $3$ & 6 & $-1$ & $q \geq 3$ \\
\hline
63 & $\Bl_{2} S_3^1$ & 53 & $\Bl_1 S_{2,I}^{10}$ &  $(0,0)$ & $3$ & 12 & $0$ & \\
\hline
64 & $\Bl_{2} S_3^4$ & 51 & $\Bl_1 S_{2,I}^8$ &  $(0,0)$ & $3$ & 18 & $1$ & \\
\hline
65 & $\Bl_{2} S_3^2$ & 46 & $\Bl_1 S_{2,I}^3$ &  $(0,0)$ & $3$ & 6 & $2$ & \\
\hline
66 & $\Bl_{1^3} S_{4,\II}^{(1^4)}$ & -- & -- &  $(24,24)$ & $4$ & 2 & $1$ & \\
\hline
67 & $\Bl_{1^3} S_{4,\II}^{(3,1)}$ & 75 & $\Bl_{3} S_{4,\II}^{(1^4)}$ &  $(24,24)$ & $4$ & 6 & $4$ & \\
\hline
68 & $\Bl_{1^3} S_{4,\II}^{(2^2)}$ & 10 & $S_{1,I}^3$ &  $(24,24)$ & $4$ & 4 & $5$ & \\
\hline
69 & $\Bl_{2,1} S_{4,\II}^{(1^4)}$ & 95 & \makecell[l]{$\Bl_{2^3,1^2}S_9$ \\
$\Bl_{2^2,1^3}S_{8,I}$ \\
$\Bl_{2^3,1}S_{8,\II}$} &  $(6,12)$ & $-8$ & 2 & $-1$ & \\
\hline
70 & $\Bl_{1^3} S_{4,I}^{(2,1^3)}$ & -- & -- &  $(12,6)$ & $-4$ & 4 & $1$ & \\
\hline
71 & $\Bl_{2,1} S_{4,\II}^{(3,1)}$ & 87 & \makecell[l]{$\Bl_{3,2^2}S_{8,I}$ \\
$\Bl_{2^2,1}S_6$} &  $(6,12)$ & $-8$ & 6 & $2$ & \\
\hline
72 & $\Bl_{2,1} S_{4,\II}^{(2^2)}$ & 45 & $\Bl_1 S_{2,I}^2$ &  $(6,12)$ & $-8$ & 4 & $3$ & \\
\hline
73 & $\Bl_{1^3} S_{4,I}^{(4,1)}$ & 4 & $S_{1,\II}^{(4,1^3)}$ &  $(12,6)$ & $-4$ & 8 & $3$ & \\
\hline
74 & $\Bl_{1^3} S_{4,I}^{(3,2)}$ & 81 & $\Bl_{3} S_{4,I}^{(2,1^3)}$ &  $(12,6)$ & $-4$ & 12 & $4$ & \\
\hline
75 & $\Bl_{3} S_{4,\II}^{(1^4)}$ & 67 & $\Bl_{1^3} S_{4,\II}^{(3,1)}$ &  $(0,6)$ & $12$ & 6 & $-2$ & \\
\hline
76 & $\Bl_{2,1} S_{4,I}^{(2,1^3)}$ & 101 & \makecell[l]{$\Bl_{4,2,1^2}S_9$ \\
$\Bl_{4,1^3}S_{8,I}$} &  $(2,2)$ & $8$ & 4 & $-1$ & \\
\hline
77 & $\Bl_{3} S_{4,\II}^{(3,1)}$ & -- & -- &  $(0,6)$ & $12$ & 6 & $1$ & \\
\hline
78 & $\Bl_{3} S_{4,\II}^{(2^2)}$ & 18 & $S_{1,I}^{11}$ & $(0,6)$ & $12$ & 12 & $2$ & \\
\hline
79 & $\Bl_{2,1} S_{4,I}^{(4,1)}$ & -- & -- &  $(2,2)$ & $8$ & 8 & $1$ & \\
\hline
80 & $\Bl_{2,1} S_{4,I}^{(3,2)}$ & 88 & \makecell[l]{$\Bl_{4,3}S_{8,I}$ \\
$\Bl_{4,1}S_6$} &  $(2,2)$ & $8$ & 12 & $2$ & \\
\hline
81 & $\Bl_{3} S_{4,I}^{(2,1^3)}$ & 74 & $\Bl_{1^3} S_{4,I}^{(3,2)}$ &  $(0,0)$ & $-12$ & 12 & $-2$ & \\
\hline
82 & $\Bl_{3} S_{4,I}^{(4,1)}$ & 7 & $S_{1,\II}^{(4,3)}$ &  $(0,0)$ & $-12$ & 24 & $0$ & \\
\hline
83 & $\Bl_{3} S_{4,I}^{(3,2)}$ & -- & -- &  $(0,0)$ & $-12$ & 12 & $1$ & \\
\hline
84 & $\Bl_{4} S_5$ & 6 & $S_{1,\II}^{(5,2)}$ &  $(0,0)$ & $-20$ & 20 & $0$ & \\
\hline
85 & $\Bl_{3,2} S_6$ & 43 & $\Bl_1 S_{2,\II}^{(3,3)}$ &  $(0,0)$ & $36$ & 6 & $-1$ & \\
\hline
86 & $\Bl_{5} S_6$ & 54 & $\Bl_1 S_{2,I}^{11}$ &  $(0,0)$ & $-30$ & 30 & $-1$ & \\
\hline
87 & \makecell[l]{$\Bl_{3,2^2}S_{8,I}$ \\
$\Bl_{2^2,1}S_6$} & 71 & $\Bl_{2,1} S_{4,\II}^{(3,1)}$ &  $(2,6)$ & $-24$ & 6 & 0 & \\
\hline
88 & \makecell[l]{$\Bl_{4,3}S_{8,I}$ \\
$\Bl_{4,1}S_6$} & 80 & $\Bl_{2,1} S_{4,I}^{(3,2)}$ &  $(2,0)$ & $24$ & 12 & 0 &  \\
\hline
89 & \makecell[l]{$\Bl_{5,2}S_{8,I}$ \\
$\Bl_{2^2}S_5$} & 41 & $\Bl_1 S_{2,\II}^{(5,1)}$ &  $(0,2)$ & $20$ & 10 & 0 & \\
\hline
90 & $\Bl_{7} S_{8,I}$ & 52 & $\Bl_1 S_{2,I}^9$ &  $(0,0)$ & $-14$ & 14 & 0 & \\
\hline
91 & \makecell[l]{$\Bl_{1^8}S_9$ \\
$\Bl_{1^7}S_{8,\II}$} & 8 & $S_{1,I}^1$ &  $(240, 2160)$ & $1$ &  1 & 9 & \makecell[l]{$q =16$ \\ or \\ $q \geq 19$} \\
\hline
92 & \makecell[l]{$\Bl_{2,1^6}S_9$ \\
$\Bl_{1^7}S_{8,I}$ \\
$\Bl_{2,1^5}S_{8,\II}$} & 44 & $\Bl_1 S_{2,I}^1$ &  $(126,756)$ & $-2$ & 2 & 7 & $q \geq 11$ \\
\hline
93 & \makecell[l]{$\Bl_{2^2,1^4}S_9$ \\
$\Bl_{2,1^5}S_{8,I}$ \\
$\Bl_{2^2,1^3}S_{8,\II}$} & 38 & $\Bl_1 S_{2,\II}^{(1^6)}$ &  $(60,252)$ & $4$ & 2 & 5 & \\
\hline
94 & \makecell[l]{$\Bl_{3,1^5}S_9$ \\
$\Bl_{3,1^4}S_{8,\II}$} & 15 & $S_{1,I}^8$ &  $(72,270)$ & $3$ & 3 & 6 & $q \geq 7$ \\
\hline
95 & \makecell[l]{$\Bl_{2^3,1^2}S_9$ \\
$\Bl_{2^2,1^3}S_{8,I}$ \\
$\Bl_{2^3,1}S_{8,\II}$} & 69 & $\Bl_{2,1} S_{4,\II}^{(1^4)}$ &  $(26,72)$ & $-8$ & 2 & 3 & \\
\hline
96 & \makecell[l]{$\Bl_{3,2,1^3}S_9$ \\
$\Bl_{3,1^4}S_{8,I}$ \\
$\Bl_{3,2,1^2}S_{8,\II}$ \\
$ \Bl_{1^5}S_6$} & 48 & $\Bl_1 S_{2,I}^5$ &  $(30,90)$ & $-6$ &  6 & 4 & \\
\hline
97 & \makecell[l]{$\Bl_{4,1^4}S_9$ \\
$\Bl_{4,1^3}S_{8,\II}$} & 1 & $S_{1,\II}^{(2,1^5)}$ &  $(40,90)$ & $-4$ & 4 & 5 & $q \geq 5$  \\
\hline
98 & \makecell[l]{$\Bl_{2^4}S_9$ \\
$\Bl_{2^3,1}S_{8,I}$} & -- & -- &  $(8,24)$ & $16$ & 2 & 1 & \\
\hline
99 & \makecell[l]{$\Bl_{3,2^2,1}S_9$ \\
$\Bl_{3,2,1^2}S_{8,I}$ \\
$\Bl_{2,1^3}S_6$} & 40 & $\Bl_1 S_{2,\II}^{(3,1^3)}$ &  $(12,18)$ & $12$ & 6 & 2 & \\
\hline
100 & \makecell[l]{$\Bl_{3^2,1^2}S_9$ \\
$\Bl_{3^2,1}S_{8,\II}$} & 16 & $S_{1,I}^9$ &  $(12,36)$ & $9$ & 3 & 3 & \\
\hline
101 & \makecell[l]{$\Bl_{4,2,1^2}S_9$ \\
$\Bl_{4,1^3}S_{8,I}$} & 76 & $\Bl_{2,1} S_{4,I}^{(2,1^3)}$ &  $(14,30)$ & $8$ & 4 & 3 & \\
\hline
102 & \makecell[l]{$\Bl_{5,1^3}S_9$ \\
$\Bl_{5,1^2}S_{8,\II}$ \\
$\Bl_{1^4}S_5$} & 19 & $S_{1,I}^{12}$ &  $(20,30)$ & $5$ & 5 & 4 & $q \geq 3$ \\
\hline
103 & \makecell[l]{$\Bl_{3^2,2}S_9$ \\
$\Bl_{3^2,1}S_{8,I}$ \\
$\Bl_{3,1^2}S_6$} & 50 & $\Bl_1 S_{2,I}^7$ &  $(6,0)$ & $-18$ & 6 & 1 & \\
\hline
104 & $\Bl_{4,2^2} S_9$ & -- & -- &  $(4,6)$ & $-16$ & 4 & 1 & \\
\hline
105 & \makecell[l]{$\Bl_{4,3,1}S_9$ \\
$\Bl_{4,3}S_{8,\II}$} & 2 & $S_{1,\II}^{(3,2,1^2)}$ &  $(4,12)$ & $-12$ & 12 & 2 & \\
\hline
106 & \makecell[l]{$\Bl_{5,2,1} S_9$ \\
$\Bl_{5,1^2}S_{8,I}$ \\
$\Bl_{2,1^2}S_5$} & 49 & $\Bl_1 S_{2,I}^6$ &  $(6,6)$ & $-10$ & 10 & 2 &  \\
\hline
107 & \makecell[l]{$\Bl_{6,1^2} S_9$ \\
$\Bl_{6,1}S_{8,\II}$} & 62 & $\Bl_{2} S_3^5$ &  $(8,12)$ & $-6$ & 6 & 3 & $q \geq 3$ \\
\hline
108 & $\Bl_{4^2} S_9$ & -- & -- &  $(0,4)$ & $16$ & 4 & 1 &  \\
\hline
109 & \makecell[l]{$\Bl_{5,3} S_9$ \\
$\Bl_{3,1}S_5$} & 22 & $S_{1,I}^{15}$ &  $(2,0)$ & $15$ & 15 & 1 & \\
\hline
110 & \makecell[l]{$\Bl_{6,2} S_9$ \\
$ \Bl_{6,1}S_{8,I}$} & -- & -- &  $(2,0)$ & $12$ & 6 & 1 & \\
\hline
111 & \makecell[l]{$\Bl_{7,1} S_9$ \\
$ \Bl_7 S_{8,\II}$} & 20 & $S_{1,I}^{13}$ &  $(2,4)$ & $7$ & 7 & 2 & All $q$ \\
\hline
112 & $\Bl_8S_9$ & 47 & $\Bl_1 S_{2,I}^4$ &  $(0,0)$ & $-8$ & 8 & 1 & All $q$  \\
\hline
\end{longtable}

\section{Points in general position}\label{sec:general position appendix}
In this appendix, we prove Theorem~\ref{thm:gen position P2}. Before we do so, we need a lemma that states that we can find irreducible polynomials with prescribed penultimate coefficients. Over finite fields, this is a special case of the Hansen--Mullen conjecture, which was proved by Wan \cite{Wan97} and Ham--Mullen \cite{HM98}. Over infinite fields, the proof is straightforward, albeit a little tedious \cite[Lem.~4.3]{McK22}.

\begin{lem}\label{lem:prescribed coeff}
Let $k$ be a field. Assume that $k$ admits a separable extension of degree $n\geq 2$. Pick $a\in k$, which we assume is non-zero if $\Char(k)=n=2$. Then there is a monic, separable, irreducible polynomial $f(t)\in k[t]$ of degree $n$ whose degree $n-1$ coefficient is $a$.
\end{lem}

We are now ready to prove Theorem~\ref{thm:gen position P2}.

\begin{proof}[Proof of Theorem~\ref{thm:gen position P2}]
For $n\leq 6$ the argument follows from \cite[\S 4.2]{McK22}. In \textit{loc.~cit.}, the author uses the parameterisation $\{[1:t:t^3]\}$ of the cuspidal cubic $C=\mb{V}(y^3-x^2z)\subset\mb{P}^2_k$ to construct a set of closed points in general position for each partition of 6. The idea is that three distinct geometric points $\{[1:t_i:t_i^3]\}_{i=1}^3$ lie on a line $\mb{V}(ax+by+cz)$ if and only if each $t_i$ is a root of
\[a+bt+ct^3,\] 
which occurs only if $t_1+t_2+t_3=0$. Similarly, six distinct geometric points $\{[1:t_i:t_i^3]\}_{i=1}^6$ lie on a conic $\mb{V}(ax^2 + bxy + cy^2 + dxz + eyz + fz^2)$
if and only if each $t_i$ is a root of 
\[a+bt+ct^2+dt^3+et^4+ft^6,\]
which occurs only if $\sum_{i=1}^6 t_i=0$. Sets of closed points in general position can therefore be obtained by constructing sextic polynomials with no repeated roots, a non-zero quintic term, and such that no three roots sum to 0. The degrees of the closed points arising from this construction correspond to the degrees of the irreducible factors of the sextic polynomial. A suitable sextic for each partition of 6 can be constructed when $|k|\geq 22$, but of course there is no field with precisely 22 elements.

Every partition of $1\leq n\leq 5$ can be obtained by some partition of 6 by omitting some terms. A set of closed points in general position stays in general position after omitting some closed points, so these cases follow from treating partitions of 6. A more direct treatment of these cases would yield a better bound on $|k|$; for example, any field will do for $n=1$ or 2.

We now address the remaining cases.

\begin{enumerate}[(i)]
\setcounter{enumi}{1}
\item $n=7$. Of the 15 partitions of 7, there are 11 that take the form $\lambda'=(1+n_1,\ldots,n_6,0)$, where $\lambda=(n_1,\ldots,n_6)$ is a partition of 6. By case (i), there is a Galois-invariant set $M$ of 6 geometric points (contained in the cuspidal cubic $C$) in general position corresponding to the partition $\lambda$. In order to construct a Galois-invariant set of 7 geometric points in general position corresponding to the partition $\lambda'$, it suffices to find a $k$-rational point on $C$ that does not lie on a line through any two points of $M$ or on a conic through any five points of $M$.

Since each pair of points determines a unique line, and each quintuple of points determines a unique conic, there are $\binom{6}{2}=15$ lines and $\binom{6}{5}=6$ conics to consider. Moreover, each line meets $C$ with multiplicity 3, and each conic meets $C$ with multiplicity 6, so we must avoid at most $3\cdot 6+6\cdot 6=54$ geometric points when selecting our rational point. This is possible if $C$ has at least 54 points in the open affine $\{[1:x:y]\}$, i.e.~if $|k|>54$.

The remaining partitions of 7, which are
\begin{equation}\label{eq:partitions of 7}
2+2+3,\quad 2+5,\quad 3+4, \quad 7,
\end{equation}

can all be treated by the same argument. We will construct degree 7 polynomials with irreducible factors of the degrees corresponding to each of these partitions. To see that no six roots sum to zero in each of these cases, note that the sum of all seven roots is an element of $k$ (in particular, the negative of the degree 6 coefficient). If any six roots summed to zero, the remaining root would have to be rational, but none of these polynomials have a rational root.

Fix a partition of 7 among those listed in Equation~\ref{eq:partitions of 7}. Let $f(t)$ be a monic, separable polynomial of degree 7 whose irreducible factors over $k$ have degrees corresponding to our fixed partition. By assumption, we are guaranteed at least one monic, separable, irreducible polynomial of the required degrees. For the partition $2+2+3$, we need two distinct quadratic polynomials. If $g(t)$ is a monic, separable, irreducible quadratic over $k$, then so is $g(t+x)$ for all $x\in k$. To ensure that $g(t)\neq g(t+x)$, it suffice to ensure that $\{r_1,r_2\}\neq\{r_1-x,r_2-x\}$, where $r_1,r_2$ are the roots of $g$. This can be done by picking $x\in k-\{0,\pm(r_1-r_2)\}$.

If no three roots of $f$ sum to 0, then we are done. Otherwise, let $\{s_1,\ldots,s_{35}\}$ be the set of sums of all $\binom{7}{3}$ subsets of three roots of $f$. Pick $x\in k-\{s_1,\ldots,s_{35}\}$. If $\Char(k)\neq 3$, then any three roots of $f(t+\frac{x}{3})$ sum to $s_i-x\neq 0$, so $f(t+\frac{x}{3})$ satisfies the desired criteria. 
    
If $\Char(k)=3$, let $f^*(t):=f(0)^{-1}\cdot t^7\cdot f(\frac{1}{t})$ be the scaled reciprocal of $f$, which is again monic, separable, and has irreducible factors of the same degrees as those of $f$. (Note that $f(0)\neq 0$, since $f$ has no linear factors over $k$.) Denote the roots of $f$ by $r_1,\ldots,r_7$. For any $\alpha\in k$, let $f_\alpha(t)=f(t-\alpha)$, so that the roots of $f^*_\alpha$ are given by $\frac{1}{r_1+\alpha},\ldots,\frac{1}{r_7+\alpha}$. The sum of any three roots of $f^*_\alpha$ is given by
\begin{align*}
    \frac{1}{r_i+\alpha}+\frac{1}{r_j+\alpha}+\frac{1}{r_\ell+\alpha}&=\frac{r_ir_j+r_ir_\ell+r_jr_\ell+2\alpha(r_i+r_j+r_\ell)+3\alpha^2}{(r_i+\alpha)(r_j+\alpha)(r_\ell+\alpha)}\\
    &=\frac{r_ir_j+r_ir_\ell+r_jr_\ell-\alpha(r_i+r_j+r_\ell)}{(r_i+\alpha)(r_j+\alpha)(r_\ell+\alpha)}.
\end{align*}
If $r_i+r_j+r_\ell=0$, then $r_i=-r_j-r_\ell$, so that
\begin{align*}
    r_ir_j+r_ir_\ell+r_jr_\ell&=-(r_j^2+2r_jr_\ell+r_\ell^2)+r_jr_\ell\\
    &=-(r_j^2-2r_jr_\ell+r_\ell^2)\\
    &=-(r_j-r_\ell)^2.
\end{align*}
Since $f$ is separable, we have $r_j\neq r_\ell$ and hence $r_ir_j+r_ir_\ell+r_jr_\ell\neq 0$. It follows that if $r_i+r_j+r_\ell=0$, then $\frac{1}{r_i+\alpha}+\frac{1}{r_j+\alpha}+\frac{1}{r_\ell+\alpha}\neq 0$.

If $r_i+r_j+r_\ell\neq 0$, then $r_ir_j+r_ir_\ell+r_jr_\ell-\alpha(r_i+r_j+r_\ell)\neq 0$ for all $\alpha\neq\frac{r_ir_j+r_ir_\ell+r_jr_\ell}{r_i+r_j+r_\ell}$. In particular, if we pick
\[\alpha\in k-\left\{\frac{r_ir_j+r_ir_\ell+r_jr_\ell}{r_i+r_j+r_\ell}:r_i+r_j+r_\ell\neq 0\right\},\]
then $f^*_\alpha$ will satisfy the desired criteria. There are at most 34 such triples of roots to consider (or else no three roots of $f$ sum to 0). We may thus find a suitable polynomial whenever $|k|\geq 37$.

\item $n=8$. We will follow the same strategy as in case (ii), but there is an extra criterion to worry about: we need to ensure that our eight geometric points do not lie on a singular cubic with one of the points at the singularity. We will prove in Lemma~\ref{lem:cayley-bacharach} that if none of our eight points lie on the cusp of $C$, then no cubic through these eight points has a singular point at one of the eight points.

As in the $n=7$ case, we would like to reduce most partitions of 8 to partitions of 7. In order to do this while still using Lemma~\ref{lem:cayley-bacharach}, we need to make sure that each set of six points corresponding to a partition of 6 avoids the cusp point of $C$; for this purpose, the partitions of 6 given in \cite{McK22} were constructed so as to avoid the cusp point of $C$ (see \cite[Rem.~4.2]{McK22}). The partitions of 7 of the form $(1+n_1,\ldots,n_6,0)$ can be made to avoid the cusp point by assuming $|k|>53+1$. The partitions of 7 of the form $(0,n_2,\ldots,n_7)$ avoid the cusp point. Given a partition $(n_1,\ldots,n_7)$ of 7, we can thus obtain a partition $(1+n_1,n_2,\ldots,n_7,0)$ of 8 that avoids the cusp point of $C$ by avoiding at most 8 rational points of $C$ in the open affine $\{[1:x:y]\}$ (the cusp point and the points from our partition of 7). This is always possible if $|k|\geq 9$.

It follows that the only remaining cases are partitions of 8 with $n_1=0$. As in case (ii), we begin with a monic, separable polynomial $f$ of degree 8 whose irreducible factors over $k$ correspond to one of the following partitions:
\begin{equation}\label{eq:partitions of 8}
 2+2+2+2,\quad  2+2+4,\quad 2+3+3,\quad 2+6,\quad 3+5,\quad 4+4,\quad 8. 
\end{equation}

In order to address the partitions with repeated summands, we again need to justify why having a single monic, separable, irreducible polynomial of a given degree over $k$ gives us more monic, separable, irreducible polynomials of the same degree.
\begin{itemize}
    \item Constructing two distinct quadratics proceeds exactly as in case (ii).
    \item For four distinct quadratics, let $g_1(t),g_2(t)$ be two distinct monic, separable, irreducible quadratics over $k$. Let $r_1,r_2$ be the roots of $g_1$ and $s_1,s_2$ be the roots of $g_2$. Pick $x\in k-\{0,\pm(r_1-r_2),r_1-s_1,r_1-s_2,r_2-s_1,r_2-s_2\}$. Then the roots of $g_3(t):=g_1(t+x)$ are disjoint from $\{r_1,r_2,s_1,s_2\}$, so $g_3$ is a monic, separable, irreducible quadratic distinct from $g_1$ and $g_2$. To obtain a fourth distinct quadratic, we take $g_4(t+y)$ with $y\in k-\{0,\pm(r_1-r_2)\}\cup\{r_i-s_j:i,j=1,2\}\cup\{r_i-w_j:i,j=1,2\}$, where $w_1,w_2$ are the roots of $g_3$. Now the roots of $g_4$ are distinct from those of $g_1,g_2,g_3$, so we have four distinct monic, separable, irreducible quadratics over $k$. This construction requires $|k|\geq 13$.
    \item Given a monic, separable, irreducible cubic $g(t)$ over $k$, we can construct another irreducible cubic by picking $x\in k-\{r_i-r_j:1\leq i,j\leq 3\}$, where $r_1,r_2,r_3$ are the roots of $g$, and considering $g(t+x)$. By construction, $g(r_i+x)\neq 0$ for $1\leq i\leq 3$, so $g$ and $g(t+x)$ are distinct monic, separable, irreducible cubics. This construction requires $|k|\geq 5$.
    \item We construct two distinct quartics in the same was as cubics: given an irreducible quartic $g(t)$, we consider $g(t+x)$ with $x\in k-\{r_i-r_j:1\leq i,j\leq 4\}$, where $r_1,\ldots,r_4$ are the roots of $g$. This requires $|k|\geq 8$.
\end{itemize}
We now know that given a partition among those listed in Equation~\ref{eq:partitions of 8}, we can construct a monic, separable polynomial $f(t)$ of degree 8 corresponding to this partition. Moreover, by repeating the arguments in step (ii), we can guarantee that no three roots of $f(t)$ sum to 0 as long as $|k|>\binom{8}{3}$. However, in contrast to case (ii), we need to verify that no six roots of $f(t)$ sum to 0. 
\begin{itemize}
\item If $\Char(k)\neq 2$ or $3$, let $\{s_1,\ldots,s_{56}\}$ and $\{u_1,\ldots,u_{28}\}$ denote the sets of sums of all $\binom{8}{3}$ subsets of three roots and $\binom{8}{6}$ subsets of six roots, respectively, of $f$. Pick $x\in k-\{\frac{s_1}{3},\ldots,\frac{s_{56}}{3},\frac{u_1}{6},\ldots,\frac{u_{28}}{6}\}$. Now any three roots of $f(t+x)$ sum to $s_i-3x\neq 0$, and any six roots of $f(t+x)$ sum to $u_i-6x\neq 0$. This requires $|k|\geq 89$.
\item If $\Char(k)=2$, any sum of three roots of $f^*_\alpha$ is given by $\frac{r_ir_j+r_ir_\ell+r_jr_\ell+\alpha^2}{(r_i+\alpha)(r_j+\alpha)(r_\ell+\alpha)}$, where $r_i,r_j,r_\ell$ are distinct roots of $f$. Similarly, the sum of six roots $r_1,\ldots,r_6$ of $f^*_\alpha$ is given by
\begin{equation}\label{eq:sum of 6}\frac{\sum_{i=1}^6\prod_{j\neq i}r_j+\alpha^2\sum_{i,j,\ell}r_ir_jr_\ell+\alpha^4(r_1+\ldots+r_6)}{\prod_{i=1}^6(r_i+\alpha)}.\end{equation}
In order to pick $\alpha$ such that no three or six roots of $f^*_\alpha$ sum to 0, it suffices to pick $\alpha$ not a root of the quadratic polynomials $r_ir_j+r_ir_\ell+r_jr_\ell+x^2$ or the quartic polynomials given by the numerators of Equation~\ref{eq:sum of 6}. We therefore have at most $2\binom{8}{3}+4\binom{8}{6}=224$ roots to avoid, so we can pick a suitable $\alpha$ if $|k|\geq 227$.
\item The $\Char(k)=3$ case is analogous to the $\Char(k)=2$ case. We again have a linear polynomial
\[r_ir_j+r_ir_\ell+r_jr_\ell+2\alpha(r_i+r_j+r_\ell)\]
and a quartic polynomial
\[\sum_{i=1}^6\prod_{j\neq i}r_j+2\alpha\sum_{i,j,\ell,n}r_ir_jr_\ell r_n+\alpha^3\sum_{i,j}r_ir_j+2\alpha^4(r_1+\ldots+r_6)\]
whose roots correspond to a sum of three or six roots, respectively, of $f^*_\alpha$ summing to 0. We thus have $\binom{8}{3}+4\binom{8}{6}=168$ potential roots to avoid. It follows that if $|k|\geq 169$, then there exists $\alpha\in k$ such that $f^*_\alpha$ satisfies the desired conditions.\qedhere
\end{itemize}
\end{enumerate}
\end{proof}

\begin{lem}\label{lem:cayley-bacharach}
Let $k$ be a field. Let $C\subset\mb{P}^2_k$ be a singular cubic. Let $P_1,\ldots,P_8$ be a set of distinct geometric points contained in the smooth locus of $C$. Then there is no singular cubic $C'$ through $P_1,\ldots,P_8$ such that some $P_i$ lies on the singularity of $C'$.
\end{lem}
\begin{proof}
By the Cayley--Bacharach theorem (see, e.g.~\cite[Thm~CB3]{EGH96}), the points $P_1,\ldots,P_8$ determine a distinct ninth point $P_9$ such that any cubic through $P_1,\ldots,P_8$ must pass through $P_9$. In particular, if a singular cubic $C'$ passes through $P_1,\ldots,P_8$ with one of $P_i$ at its singular point, then $C$ and $C'$ meet with multiplicity 10 --- both curves must contain $P_1,\ldots,P_9$, and they meet with multiplicity 2 at $P_i$. This contradicts B\'ezout's theorem, so such a $C'$ cannot exist.
\end{proof}

\section{Open questions} \label{sec:Questions}
While preparing this article, we encountered a few questions that we were not able to answer. The first of these pertains to points in general position on del Pezzo surfaces. We showed in Proposition~\ref{prop:general position dP} that over an infinite field $k$ and for $d\geq 5$, every blowup type of degree-$d$ del Pezzo surfaces admits a surface with a set of points in general position for every possible partition (i.e.~those not obstructed by the Galois theory of $k$). The key was finding a rational surface of each blowup type.

For $d\leq 4$, this method breaks down, which is why we turned our attention to various characteristics of the base field. However, if one could show that every \emph{unirational} del Pezzo surface contains a set of points in general position for each partition, one could likely dispense with the assumptions on the base field and treat all infinite fields uniformly. This leads to the following question.

\begin{ques}
    Let $k$ be an infinite field, and let $X$ be a unirational del Pezzo surface of degree $d\geq 2$. Let $1\leq n\leq d-1$, and let $\lambda=(n_1,\ldots,n_8)$ be a partition of $n$. If $k$ admits a separable extension of degree $i$ for each $n_i>0$, does $X$ contain a set of closed points in general position corresponding to $\lambda$?
\end{ques}

It is known that for $d\geq 3$, degree-$d$ del Pezzo surfaces with a rational point are unirational. For unirationality of degree-2 del Pezzo surfaces, see \cite[Thm.~29.4]{Man86} and \cite{STVA14}.

Although we found a suitable way to construct points in general position on low degree del Pezzo surfaces over Hilbertian fields, we spent some time trying to prove a stronger statement that seems quite interesting.

\begin{ques}
    Let $X$ be a degree-$d$ del Pezzo surface over a Hilbertian field $k$. Let $n<d$. Let $X^{[n]}$ denote the Hilbert scheme of $n$ points on $X$. Does $X^{[n]}$ satisfy the Hilbert property? That is, is $X^{[n]}(k)$ not a thin set?
\end{ques}

Of course, $\mb{P}^1$ and $\mb{P}^2$, along with their blowups, satisfy the Hilbert property, so the remaining open cases are relatively minimal del Pezzo surfaces of degree at most 6, as well as their blowups.

As mentioned in Remark~\ref{rem:IGPfordP1s}, there are 17 classes of degree-1 del Pezzo surfaces over finite fields for which the surface and its Bertini twist are both of Picard rank 1. These are the surfaces currently obstructing the resolution of $\IGP_1(\mb{F}_q)$.

\begin{ques}
For any of the $17$ classes of degree-1 del Pezzo surfaces over finite fields with Picard rank 1 and Bertini twist of Picard rank 1, what are the exact bounds on $q$ for their existence over $\mathbb{F}_q$? 
\end{ques}

\bibliography{del-pezzo}{}
\bibliographystyle{alpha}
\end{document}